\documentclass[11pt,reqno]{amsart}
\usepackage{amsfonts}
\usepackage{graphicx}
\usepackage[normalem]{ulem}
\usepackage{url}

\usepackage{multirow}
\usepackage{rotating}
\usepackage{booktabs}

\usepackage{amsthm,amsfonts,amssymb,amsmath,esint,amscd,amsxtra}
\usepackage[colorlinks=true, urlcolor=blue,bookmarks=true,bookmarksopen=true,
citecolor=blue]{hyperref}
\usepackage{amscd}
\usepackage{enumerate,bbm}

\usepackage{fullpage}
\usepackage[all]{xy}
\usepackage{xspace}

\newcommand{\BA}{{\mathbb {A}}} 
\newcommand{\BC}{{\mathbb {C}}} 
 \newcommand{\BF}{{\mathbb {F}}}

 \newcommand{\BN}{{\mathbb {N}}}
 
\newcommand{\BQ}{{\mathbb {Q}}} \newcommand{\BR}{{\mathbb {R}}}

 \newcommand{\BZ}{{\mathbb {Z}}}

\newcommand{\CA}{{\mathcal {A}}} 
\newcommand{\CC}{{\mathcal {C}}}

\newcommand{\CK}{{\mathcal {K}}} 
\newcommand{\CM}{{\mathcal {M}}} 
\newcommand{\CO}{{\mathcal {O}}} \newcommand{\CP}{{\mathcal {P}}}
 
\newcommand{\CS}{{\mathcal {S}}} 
 
\newcommand{\CW}{{\mathcal {W}}}

\newcommand{\RU}{{\mathrm {U}}}

 \newcommand{\fp}{{\mathfrak{p}}}

\newcommand{\ad}{{\mathrm{ad}}}

\newcommand{\diag}{{\mathrm{diag}}}

 \newcommand{\fin}{{\mathrm{fin}}}

\newcommand{\Gal}{{\mathrm{Gal}}} \newcommand{\GL}{{\mathrm{GL}}}
\newcommand{\GSp}{{\mathrm{GSp}}}
\newcommand{\Hom}{{\mathrm{Hom}}}

\newcommand{\Ind}{{\mathrm{Ind}}}

\renewcommand{\mod}{\ \mathrm{mod}\ }\renewcommand{\Re}{{\mathrm{Re}}}

\newcommand{\Res}{{\mathrm{Res}}}

\newcommand{\SL}{{\mathrm{SL}}}

\newcommand{\St}{{\mathrm{St}}}
\newcommand{\Sym}{{\mathrm{Sym}}}
\newcommand{\Stab}{{\mathrm{Stab}}}

\newcommand{\val}{{\mathrm{val}}}
\newcommand{\Vol}{{\mathrm{Vol}}}

\newcommand{\Ext}{{\mathrm{Ext}}}

\newcommand{\wt}[1]{{\widetilde {#1}}}

\newcommand{\sk}{\medskip}
\newcommand{\lra}{\longrightarrow}

\newcommand{\bs}{\backslash}

\newcommand{\s}{\sk\noindent}

\makeatletter\def\varW@#1#2{%
\vtop{\m@th\ialign{##\cr
\hfil$#1 \mathrm{colim} $\hfil\cr
\noalign{\nointerlineskip\kern1.5\ex@}#2\cr
\noalign{\nointerlineskip\kern-\ex@}\cr}}
}
\makeatother
\makeatletter
\def\colim{%
\mathop{\mathpalette\varW@{}}\nmlimits@
}\makeatother

\theoremstyle{plain}
\newtheorem{thm}{Theorem}[section] \newtheorem{cor}[thm]{Corollary}

\newtheorem{lem}[thm]{Lemma}  \newtheorem{prop}[thm]{Proposition}
\newtheorem {conj}[thm]{Conjecture}

\theoremstyle{remark} \newtheorem{remark}[thm]{Remark}
\theoremstyle{definition} 
\theoremstyle{definition} \newtheorem{example}[thm]{Example} 
\newtheorem{defn+lem}[thm]{Definition and Lemma}

\numberwithin{equation}{section}

\newcommand{\cind}{\operatorname{c-Ind}}

\newcommand{\Nm}{\mathrm{Nm}}
\newcommand{\Tr}{\mathrm{Tr}}

\newcommand*{\sheafhom}{\mathrm{H}\kern -.5pt om}
\newcommand{\revisionredon}{\begingroup\color{red}\small}
\newcommand{\revisionredoff}{\endgroup}

\begin{document}
	\title{Co-periods and Central Symmetric Cube $L$-values}
	
    \begin{abstract}
In this article, we  study  the co-period  integral attached to an automorphic form  on $\GL(2)$ and two exceptional theta series on the cubic Kazhdan-Patterson cover of $\GL(2)$.  In the local aspect, we show the $\Hom$-space is multiplicity one and dictated by the symmetric cube epsilon-factor, and conduct the unramified calculations. In the
global aspect, we give the Euler decomposition for the co-period integrals of Eisenstein series and propose an Ichino-Ikeda type conjecture relating the co-period integrals of cuspidal forms to the central critical value of symmetric cube $L$-functions. We also deduce from the local multiplicity one result that there exist cuspidal automorphic forms with prescribed local components and non-vanishing central symmetric cube $L$-values.
\end{abstract}

\author{Li Cai}
\address{Academy for Multidisciplinary Studies\\
Beijing National Center for Applied Mathematics\\
Capital Normal University\\
Beijing, 100048, People's Republic of China}
\email{caili@cnu.edu.cn}

\author{Yangyu Fan}
\address{Key Laboratory of Algebraic Lie Theory and Analysis of Ministry of Education\\ School of Mathematics and Statistics\\ Beijing Institute of Technology\\ Beijing, 100081, People's Republic of China}
\email{yangyu.fan@bit.edu.cn}

\author{Dongming She}
\address{Yanqi Lake Beijing Institute of Mathematical Sciences and Applications (BIMSA)\\ Huairou District, 100084, Beijing, People's Republic of China.}
\email{shedongming@bimsa.cn}
\maketitle
	\tableofcontents{}

    \section{Introduction}
Let $G$ be a locally compact group. An important aspect of the representation theory of $G$ concerns the restriction of  $G$-representations to a suitable subgroup $H\subset G$.

Consider the case  $(G,H)$ is a pair of algebraic groups over a field $F$. When $F$ is a local field, a basic problem in the
representation theory is to study the $\Hom$-space $\Hom_{H(F)}(\pi,\BC)$ where $\pi$ is an admissible $G(F)$-representation. 
When $F$ is a global field with $\BA$ its adelic ring, a basic problem in the automorphic representation theory is to study the period integral
\[\int_{H(F) \bs H(\BA)} \varphi(h)dh, \quad \varphi \in \pi\]
where $\pi$ is an automorphic  $G(\BA)$-representation. For a large class of $(G,H)$, called {\em spherical pairs}, it is expected that
the local $\Hom$-spaces can be characterized in terms of $L$-parameters and the global period integrals are closely related
to $L$-values. The work of Sakellaridis-Venkatesh \cite{SV} gives explicit conjectures for such spherical pairs. Recently, these conjectures 
have been largely generalized to hyperspherical Hamiltonian spaces by Ben-Zvi, Sakellaridis, and
Venkatesh  in \cite{BZSV}. 

Comparing with algebraic groups, the study of such problems for  covering groups is  still relatively limited. Among the few examples,  the  {\em co-period} integrals
\[ \int_{\GL_r(F)\BA^\times \bs \GL_r(\BA)} ff_+f_-(g)dg,\quad f\in \sigma, f_\pm\in\sigma_\pm\]
are particularly interesting for its close relation with the symmetric power L-functions. Here  $\sigma$ is an automorphic   $\GL_r(\BA)$-representation and  $\sigma_\pm$ are the exceptional theta series on $n$-th metaplectic cover of $\GL_r(\BA)$ (with opposite $\mu_n$-action). 

For the symmetric square $L$-function $L(s,\sigma,\Sym^2)$, Shahidi \cite{Sha88}  established its meromorphic continuation and functional equation using the Langlands-Shahidi method. Bump and Ginzburg \cite{BG92} subsequently provided its integral representation. Takeda \cite{Tak14} later extended this to the twisted symmetric square for a Hecke character $\chi$.

For the symmetric cube $L$-function $L(s,\sigma,\Sym^3)$ with $r=2$, the meromorphic continuation and  functional equation is  given by Shahidi \cite{Sha89}. Its integral representation is then provided by Bump-Ginzburg-Hoffstein \cite{BGH96}, while Kim and Shahidi \cite{KS99} showed it is entire unless $\sigma$ is monomial. Many other pivotal works exist on symmetric power $L$-functions. Among them, we mention the work of Newton-Thorne on the
modularity of the symmetric power for regular algebraic cuspidal automorphic
representations on $\GL_2(\BA_\BQ)$ \cite{NT1,NT2}.

Results for co-period integrals in the case  $n=2$ are relatively complete.   The local problem is studied by Kable \cite{Kab01}, Kaplan \cite{Kap17} and  Yamana \cite{Yama17}. 
For a nonarchimedean place $v$, the  local $\Hom$-space is at most one-dimensional for any irreducible admissible unitary representation
$\sigma_v$ on $\GL_r(F_v)$. For the global aspect, the celebrated work of Bump-Ginzburg essentially proves that 
the co-period integrals are nonvanishing if and only if the symmetric square $L$-function $L(s,\sigma,\Sym^2)$ has a simple pole at $s=1$ (See \cite[Theorem 7.6]{BG92}). Moreover, there is an Euler decomposition of the co-period integrals in terms of the residue of 
$L(s,\sigma,\Sym^2)$ at $s=1$ and the local zeta integrals (for this decomposition, in addition to the work of Bump-Ginzburg mentioned above, 
the characterization of local $L$-factors with respect to local zeta integrals is also needed \cite{Yama17}).

The present article will concentrate on the case $n=3$ and $r=2$, where only partial results are known. By the work of Ginzburg-Jiang-Rallis \cite{GJR01}, the nonvanishing of global co-period integrals implies that the
nonvanishing of the central value of the symmetric cube $L$-series $L(s,\sigma,\Sym^3)$. The proof is based on the work of Kim-Shahidi \cite{KS99} and the residue method. Recently, the aforementioned result of
Ginzburg-Jiang-Rallis is applied to study the nonvanishing of the central value of the symmetric cube $L$-function \cite{HJL23}. 

Limited to our knowledge, no results for co-period  exist in the literature when $n \geq 3$ and $r \geq 3$.

In this paper, we conduct a systematic study of the co-period integral in the case $n=3$ and $r=2$. Hopefully our results will shed some light on how to incorporate period integrals of covering groups into the framework of Ben-Zvi, Sakellaridis and Venkatesh \cite{BZSV}. For the local aspect, we obtain multiplicity-free property for the generic representations and propose a symmetric cube epsilon-criterion for distinguished representations. When everything is unramified, we relate the local integral of matrix coefficients to the local symmetric cube $L$-factor. For the global aspect, we give the Euler decomposition for the co-period integrals of Eisenstein series and propose an Ichino--Ikeda type conjecture for co-period integrals of cuspidal forms based on our local results.

Subsequently, we furnish full details for the results and proofs presented therein. 

Let $F$ be a number field containing  all $3$-th roots of  unity and fix an embedding $\epsilon: \mu_3\hookrightarrow \BC^\times$. Let $G=\GL_2$ and $Z, N, T\subset G$ be the center, the upper unipotent subgroup and the diagonal torus.  Let $\wt{G}(\BA):=\wt{G}^{(c)}(\BA)$ be the cubic metaplectic covering group constructed by the Kubota 2-cocycle $\sigma_K^{(c)}$. It sits in the exact sequence
    \[0\lra\mu_3\lra \wt{G}(\BA)\stackrel{p}{\lra} G(\BA)\lra 0.\]
    The projection map  $p:\ \wt{G}(\BA)\to G(\BA)$ has a preferred splitting $s$ (of sets).  When  restricted to 
$G(F)$ (resp. $N(\BA)$), $s$ is an injective group homomorphism.  Let $\wt{Z^{(c)}}(\BA)\subset \wt{G}(\BA)$ be the center and $\wt{Z}(\BA)\subset \wt{T}(\BA)$ be the pre-image of $Z(\BA)\subset T(\BA)$ in $\wt{G}(\BA)$. 

Take $T_*(\BA)\subset T(\BA)$ such that  its pre-image $\wt{T_*}(\BA)\subset \wt{T}(\BA)$ is a maximal abelian subgroup  containing $T(F)$.  Let $\chi_{\pm}$ be {\em exceptional characters} on $T(F)\bs \wt{T_*}(\BA)$ such that \[\chi_{\pm}|_{\mu_3}=\epsilon^\pm, \quad 
\chi_\pm(s(\diag\{x^3,x^{-3}\}))=|x|_{\BA}.\]  The usual parabolic induction $I_{\wt{T_*}(\BA)N(\BA)}^{\wt{G}(\BA)}\chi_\pm$ admits an irreducible quotient $\sigma_\pm=\otimes_{\BC[\mu_3]}^\prime \sigma_{\pm,v}$ which  can be realized in the space of automorphic forms $\CA(G(F)\bs \wt{G}(\BA))$ via taking residues of the Eisenstein series attached to $I_{\wt{T_*}(\BA)N(\BA)}^{\wt{G}(\BA)}\chi_\pm$.
 Denote by $\omega_{\pm}$ the central character of $\sigma_{\pm}$.  Let $\sigma=\otimes^\prime\sigma_v$ be an irreducible  automorphic $G(\BA)$-representation with central character $\omega$. For the clarity of  exposition, in the introduction we assume that
 \[\chi_+\chi_-(s(\diag\{x^3,y^3\}))=|x|_\BA|y|_\BA^{-1}, \quad \omega\omega_+\omega_-=1.\]
In particular, $\omega^3 = 1$ so that $\omega$ is unitary and for all $v\leq\infty$,
\[
 \varepsilon(\sigma_v,\Sym^3):=
 \varepsilon(1/2,\sigma_v,\Sym^3)\in\{\pm1\}.
\]
Note that  $\varepsilon(\sigma_v,\Sym^3)=1$ when  $\sigma_v$ is non-special and $v\nmid 3$ (For more details, we refer to  Appendix \ref{epsilon-factor}).  We remark that all the following results hold uniformly for all Kazhdan--Patterson cubic covering groups, i.e. allowing all Kubota $2$-cocycles $\sigma_K^{(c)}$, $c\in\BZ/3\BZ$. In particular, the exceptional theta series may have multiple Whittaker functionals. 
 
 Our first result is the local multiplicity-free result:
\begin{thm}[Proposition \ref{n-SC} and Theorem \ref{sc}] \label{Mulone} Let $v$ be a place of $F$. If $v\nmid 3$ or $\sigma_v$ is non-supercuspidal, then
\[
 \dim \Hom_{G(F_v)}(\sigma_v\otimes\sigma_{+,v}\otimes \sigma_{-,v},\BC) =\frac{1+\varepsilon(\sigma_v,\Sym^3)}2\leq1.
\]
The dimension is one if $\sigma_v$ is non-special. 

\end{thm}

 We believe the above multiplicity formula holds even when $\sigma_v$ is supercuspidal and   $v\mid 3$. As an evidence, combining Theorem \ref{Mulone} with the globalization method in Theorem \ref{Non-vanishing} below, one can show  that if $v\mid 3$ and $\sigma_v$ is supercuspidal, then 
 \[
 \dim \Hom_{G(F_v)}(\sigma_v\otimes\sigma_{+,v}\otimes \sigma_{-,v},\BC) \neq 0\implies \varepsilon(\sigma_v,\Sym^3)=1.
\]
We postpone a brief explanation of 
 our proof to the end of this introduction.  Our method is actually applicable to study the local branching problem for general $n$-th covering of $\GL_2$. We plan to return to this topic in the near future. We also remark that Patel and Prasad considered the branching problem for metaplectic groups in other circumstances (see e.g. \cite{PP16,PP18}).

  Our second result concerns the canonical local period and the unramified calculation. Note that in the archimedean case, the situation reduces to the usual triple product case (see Subsection \ref{Archimedean} for details).  We shall only consider 
  the nonarchimedean case in the following. 
  
\begin{thm}[Theorem \ref{Canonical local period} and Theorem \ref{unramified calculation}]
\label{Local-Unr}
Let $v\nmid 3$ be a nonarchimedean place of $F$. For $*=\emptyset,\pm$, take a non-degenerate $\wt{G}(F_v)$-invariant pairing $(-,-)_{*,v}$ on $\sigma_{*,v}$. When $\varepsilon(\sigma_v,\Sym^3)=1$ , for $f_{*,v} \in \sigma_{*,v}$ and $
 f_{*,v}^\vee \in \sigma_{*,v}^\vee$,  the integration of matrix coefficients
\[\int_{Z^{(c)}(F_v)\bs G(F_v)}( \sigma_v(g_v)f_v, f_v^\vee)_v( \sigma_{+,v}(g_v)f_{+,v}, f_{+,v}^\vee)_{+,v} ( \sigma_{-,v}(g_v)f_{-,v}, f_{-,v}^\vee)_{-,v} d g_v,\]
 denoted by $I(f_v\otimes f_{+,v}\otimes f_{-,v}, f_v^\vee\otimes f_{+,v}^\vee\otimes f_{-,v}^\vee)$, 
is  a generator of \[\Hom_{G(F_v)}(\sigma_v\otimes\sigma_{+,v}\otimes \sigma_{-,v},\BC)\times \Hom_{G(F_v)}(\sigma_v^\vee\otimes\sigma^\vee_{+,v}\otimes \sigma^\vee_{-,v},\BC).\] 

Moreover, assume that 
\begin{itemize}
\item $\sigma_v=I_{B(F_v)}^{G(F_v)}\theta_v$ and $\sigma_{\pm,v}$ are all unramified,
\item $f_v,f_v^\vee, f_{\pm,v},f_{\pm,v}^\vee$ are all spherical and $( f_v, f_v^\vee)_v=( f_{+,v}, f_{+,v}^\vee)_{+,v} =( f_{-,v}, f_{-,v}^\vee)_{-,v}=1$.
\end{itemize}
Then up to a normalization of the measure $d g_v$,
\begin{align*}
    I(f_v \otimes f_{v,+}\otimes f_{v,-}, f_v^\vee \otimes f_{+,v}^\vee\otimes f_{-,v}^\vee)=\frac{L(1/2,\sigma_v,\Sym^3)}{L(1,\sigma_v,\ad)}.
\end{align*}
\end{thm}
In Theorem \ref{Local-Unr},  we obtain that the integration of matrix coefficients  is a generator  by adapting the method in \cite[Proposition 5.2]{Hsi21} to  decompose $I$ into the product of two local co-period integrals (see Proposition \ref{Tri-I} and \ref{Tri-II} for details) assuming the $\wt{G}(F_v)$-pairings are realized explicitly in the Whittaker model and the induced model (see Subsection \ref{InvP} for details).  The local co-period integral can be studied analogously to the Rankin-Selberg zeta integral. With this decomposition formula, the result on the unramified calculation  follows immediately from the explicit description of the spherical Whittaker function in \cite[Theorem I.4.2]{KP84}.

With all these local preparations, we study the  global co-period integral 
 \[\int_{Z_*(\BA)G(F)\bs G(\BA)}f f_+ f_-(g)d^\tau g,\quad f\in\sigma,\ f_\pm\in\sigma_\pm.\]
 Here $Z_*(\BA)=Z(\BA)\cap T_*(\BA)$ and   the
  Haar measure $d^\tau g$  satisfies that 
$\Vol(G(F)Z(\BA)\bs G(\BA),d^\tau g)=2$. 
In particular, we shall consider its relation with the central  symmetric cube $L$-value $L(1/2,\sigma,\Sym^3)$.

Our first global result is an immediate consequence of Theorem \ref{Mulone}. 
\begin{thm}\label{Non-vanishing} Let $S$ be a non-empty finite set of non-archimedean $F$-places disjoint with those places dividing $3$. For each $v\in S$, fix a supercuspidal $G(F_v)$-representation $\sigma_v^\circ$. Then there exists a cuspidal automorphic $G(\BA)$-representation $\sigma$ such that $\sigma_v=\sigma_v^\circ$ for all $v\in S$ and there exist automorphic forms $f\in\sigma$, $f_\pm\in\sigma_\pm$ such that $\CP(f,f_+,f_-)\neq0$. In particular, there are infinitely many cuspidal automorphic $G(\BA)$-representations $\sigma$ with $L(1/2,\sigma,\Sym^3)\neq0$.
\end{thm}
\begin{proof}By Theorem \ref{Mulone}, one deduces that for any $v \in S$,
$\sigma_v^\circ$ appears as a quotient of $\sigma_{+,v}\otimes\sigma_{-,v}$. Now apply the Burger-Sarnak type principle \'a la Prasad to the pair  $(\Delta\mu_3\bs \wt{G}(\BA)\times\wt{G}(\BA), \GL_2(\BA))$ (see \cite[Lemma 1]{Pra07}\footnote{In \textit{loc.cit},  the larger group  is assumed to be reductive, but the proof carries over verbatim.}).  We deduce that  
there exist a cuspidal automorphic $G(\BA)$-representation $\sigma$ such that $\sigma_v=\sigma_v^\circ$ for all $v\in S$ and there are automorphic forms $f\in\sigma$, $f_\pm\in\sigma_\pm$ such that $\CP(f,f_+,f_-)\neq0$. As $\CP(f,f_+,f_-)\neq0$ implies $L(1/2,\sigma,\Sym^3)\neq0$
by \cite{GJR01}, we deduce that there are infinitely many $\sigma$ with $L(1/2,\sigma,\Sym^3)\neq0$ from the infinitude of $\sigma_v^\circ$.
\end{proof}
\begin{remark}
Similar result on the non-vanishing of $L(1/2,\sigma,\Sym^3)$
is obtained in \cite{HJL23} with $\sigma$ unramified outside
$3$. Our method is completely different from theirs. 
 In \textit{loc.cit},  the author deduces the non-vanishing result from an approximating result of certain weighted average of the co-period.
\end{remark}
Our second global result is an  Ichino-Ikeda type conjecture for  cuspidal automorphic $\GL_2(\BA)$-representations. Assume $\sigma$ is cuspidal. Consider the Petersson inner product  \[(-,-): \sigma\times\sigma^\vee\to\BC,\quad  (f,f^\vee)\to\int_{G(F)Z(\BA)\bs G(\BA)}f(g)f^\vee(g)d^\tau g\]
and the Petersson inner product 
 \[(-,-)_\pm: \sigma_\pm\times\sigma_\pm^\vee\to\BC,\quad  (f_\pm,f_\pm^\vee)\to\int_{G(F)Z_*(\BA)\bs G(\BA)}f_\pm(g)f_\pm^\vee(g)d^\tau g.\]
Fix a decomposition $(-,-)_*=\prod_v(-,-)_{*,v}$ for $*=\emptyset,\pm$. Fix a decomposition $d^\tau g=\prod_v d^\tau g_v$. For any place $v$ of $F$, any $f_{v,*}^\prime\in\sigma_{v,*}^\prime$ with $*=\emptyset,\pm$ and $\prime=\emptyset,\vee$, set 
\[I^\sharp(f_v\otimes f_{+,v}\otimes f_{-,v},f_v^\vee\otimes f_{+,v}^\vee\otimes f_{-,v}^\vee)
=\frac{L(1,\sigma_v,\ad)}{L(1/2,\sigma_v,\Sym^3)}I(f_v\otimes f_{+,v}\otimes f_{-,v},f_v^\vee\otimes f_{+,v}^\vee\otimes f_{-,v}^\vee).\]
For $f_*^\prime\in\sigma_*^\prime$ and $\prime=\emptyset,\vee$, set \[\CP^\prime(f^\prime,f_+^\prime,f_-^\prime)=\int_{Z_*(\BA)G(F)\bs G(\BA)}f^\prime f_+^\prime f^\prime_-(g)d^\tau g.\]

\begin{conj}\label{co-period-conj}
The co-period trilinear form $\CP$ is non-zero if and only if
\[
 L(1/2,\sigma,\Sym^3)\neq0
 \quad\text{and}\quad
 \varepsilon(\sigma_v,\Sym^3)=1\quad \forall v.
\]
 Moreover when  $\varepsilon(\sigma_v, \Sym^3)=1$ for all places $v$, there exists a non-zero constant $C\in\BQ^\times$ independent of $\sigma$ such that for any pure tensor $f_*^\prime=\otimes f_{v,*}^\prime\in\sigma_*^\prime=\otimes_{\BC[\mu_3]}^\prime \sigma_{v,*}^\prime $ with $*=\pm,\emptyset$ and $\prime=\emptyset,\vee$,

\[\CP(f,f_+,f_-)\CP^\vee(f^\vee,f_+^\vee,f_-^\vee)=C\frac{L(1/2,\sigma,\Sym^3)}{L(1,\sigma,\ad)}\prod_v I^\sharp(f_v\otimes f_{v,+}\otimes f_{v,-},f_v^\vee\otimes f_{v,+}^\vee\otimes f_{v,-}^\vee).\]
\end{conj}
Similar as Waldspurger's toric period and the trilinear period, we believe the co-period satisfy a epsilon-dichotomy phenomenon. In a forthcoming companion paper, we shall discuss this in detail.

\begin{remark} 
We give the following remarks.
\begin{itemize}
 \item The co-period $\CP$ is nonvanishing if and only if the co-period $\CP^\vee$ for dual representations is nonvanishing.
 \item In the paper \cite{GJR01}, the implication $\CP\neq0\implies L(1/2, \sigma,\Sym^3)\neq0$ is established. In \textit{loc.cit}, the authors mentioned the possibility of  the opposite implication and an formula relating co-period and the central symmetric cube $L$-value. Note that the co-period in \cite{GJR01} is defined as an integration over $G(F) \bs G(\BA)^1$ which is slightly different from here. Since $G(F) Z_*(\BA) \bs G(\BA) = G(F)(Z_*(\BA) \cap G(\BA)^1) \bs G(\BA)^1$, the difference is just the (finite) volume of $(Z_*(\BA) \cap G(F)) \bs (Z_*(\BA) \cap G(\BA)^1)$.
 \end{itemize}
\end{remark}
\begin{remark}Let $\pi_i$, $i=1,2,3$ be irreducible cuspidal automorphic $G(\BA)$-representations with central
characters $\omega_i$, $i=1,2,3$ and assume $\omega_1\omega_2\omega_3 = 1$. 
The Ichino formula \cite{Ich08} relates the period integral 
\[\int_{Z(\BA)G(F)\bs G(\BA)}\phi_1\phi_2\phi_3(g)d^\tau g,\quad \phi_i\in\pi_i\]
to the central value $L(1/2,\pi_1\times\pi_2\times\pi_3)$ of the triple product $L$-function. When $\pi_1=\pi_2=\pi_3$ with central character $\omega$, one has \[L(s,\pi\times\pi\times\pi)=L(s,\pi,\Sym^3)L(s,\pi\otimes\omega)^2.\] 
This allows one to study $L(1/2,\pi,\Sym^3)$ via Ichino formula when $L(1/2, \pi\otimes\omega)\neq0$. When $L(1/2, \pi\otimes\omega)=0$, Conjecture \ref{co-period-conj} seems inevitable.
\end{remark}
Our final result concerns  the co-period integral in the Eisenstein case, that is, $\sigma$ is spanned by the 
Eisenstein series $E(-,\theta,f)$ with $f\in I_{B(\BA)}^{G(\BA)}\theta$ for some character $\theta=(\theta_1,\theta_2):\ T(F)\bs T(\BA)\to\BC^\times$. In this case, 
we consider the following inner product on $\sigma\times\sigma^\vee$:
\[\left(E(-,\theta,f),E(-,\theta^{-1},f^\vee)\right)\mapsto \int_K ff^\vee(k)dk\]
where $K\subset G(\BA)$ is the standard maximal open compact subgroup and $dk$ is the Haar measure on $K$ with total volume one. As the naive co-period integral may diverge in the Eisenstein case, we apply the mixed truncation operator considered by Yamana in the trilinear case \cite{Yam18} to do regularization (see Theorem \ref{co-period for Eisenstein family} and Proposition \ref{EisenI} for details). Let $\lambda(\theta)$  be the real number such that $\theta\delta^{-\lambda(\theta)}$ is unitary. By an unfolding argument, we obtain the following result.
\begin{thm}[Theorem \ref{EisenII}] \label{Eisen} When $-1/6<\lambda(\theta)<1/6$, for $f_*^\prime=\otimes f_{v,*}^\prime\in\sigma_*^\prime=\otimes_{\BC[\mu_3]}^\prime \sigma_{v,*}^\prime $ with $*=\pm,\emptyset$ and $\prime=\emptyset,\vee$, 
\[\CP(f,f_+,f_-)\CP^\vee(f^\vee,f_+^\vee,f_-^\vee)=\frac{L(1/2,\sigma,\Sym^3)}{L(1,\theta_1\theta_2^{-1})L(1,\theta_2\theta_1^{-1})}\prod_v I^\sharp(f_v\otimes f_{v,+}\otimes f_{v,-},f_v^\vee\otimes f_{v,+}^\vee\otimes f_{v,-}^\vee).\]
\end{thm}

In the rest of this introduction, let us briefly explain our strategy to Theorem \ref{Mulone}. We shall omit all subscripts $v$ in the following and write the groups $G(F_v)$ etc simply by $G$ etc.

When $\sigma$ 
is non-supercuspidal, we apply Frobenius reciprocity law and the geometric lemma to compute $\Hom_{G}(\sigma_+\otimes\sigma_-,\sigma^\vee)$ (See Proposition \ref{PS} and Proposition \ref{PS2}) for details). This reduces to   the Jacquet module $(\sigma_+\otimes\sigma_{-})_N$ of $\sigma_+\otimes\sigma_{-}$ (See Proposition \ref{Jac} and \ref{Jac of tensor}), which in turn are reduced to computing the Jacquet module and Whittaker functionals of exceptional representations(available by  \cite{KP84}).  To deal with the case $\sigma$ is special (See Proposition \ref{St} for details), we also need to consider higher extension groups, which is done by using  the Schneider-Stuhler duality theorem and standard results on the Euler-Poincare numbers.

The  case $\sigma$ ($|3|=1$) is supercuspidal  is much more involved. Let $\pi_{\pm}$ be the $\wt{G}$-principal series containing $\sigma_\pm$. Again by the Frobenius reciprocity law and the parabolic induction nature of $\pi_-$, one can show $\dim\Hom_G(\sigma\otimes\sigma_+, \pi_-^\vee)=3$ similar to the $\sigma$ non-supercuspidal case (see Proposition \ref{ub} for details). This gives an upper bound of $\dim\Hom_G(\sigma\otimes\sigma_+\otimes\sigma_-,\BC)$. 

To obtain the multiplicity one, we 
consider the construction of supercuspidal representations via the compact induction. In particular, we can write $\sigma=i_{\CK}^G\tau^\vee$, where $\CK=ZK$ or $J$ depending on the conductor of $\sigma$ and $\tau$ is a very cuspidal representation of $\CK$. Here  $K=\GL_2(\CO)$ and \[J=ZK_1\bigsqcup Z\begin{pmatrix} 0 & 1\\ \varpi & 0\end{pmatrix} K_1\]
with $K_1$ the Iwahori subgroup of $K$.

By the Frobenius reciprocity law, we are lead to determine the restriction of $\sigma_{\pm}$ to the maximal compact-modulo-center subgroup $\wt{\CK}$ where $\tilde{\CK}$ stands for the preimage of $\CK$ in $\wt{G}$. By twisting, we can and will assume the exceptional characters $\chi_{\pm}$ are unramified and normalized.

Set  $V^\prime=\sigma_{\pm}$ and $V = \pi_{\pm}$. The structure of $V^\prime|_{\wt{\CK}}$ has  its independent interests. For example, we can recover  (and slightly generalize) the level structure result  in \cite{Pat98} for $n=3$:
\[\dim V^{\prime,K}=1,\quad \dim V^{\prime,K_m}=m, \quad K_m=\left\{\begin{pmatrix} a & b\\ c & d\end{pmatrix}\in K \Bigg| \varpi^m \mid c \right\}, \quad m\geq 1.\]
 
 The study of the restriction of $V'$ to $\wt{\CK}$ is the most technical part of the whole paper. We now give more details. 
To determine $V^\prime|_{\wt{\CK}}$,  we first follow  \cite{Cas73} to determine $V|_{\wt{\CK}}$ using the Mackey theory as following.
\begin{itemize}
\item By the parabolic induction nature of $V$ and the Iwasawa decomposition $G=BK$, we obtain that (see Lemma \ref{K-type I} for details)  \[V|_{\wt{ZK}}=V_{0,0}\bigoplus\left(\bigoplus_{i\in\BZ/3\BZ}V_{1,i}\right)\bigoplus \left(\bigoplus_{k\geq2}V_{k}^{\oplus 3}\right)\]
where $V_{0,0}$ and $V_{1,i}$ and $V_{k}$ are mutually non-isomorphic irreducible $\wt{ZK}$-representations. Here $V_{0,0}$ is the  component containing the spherical vector.
\item Consider the  restriction of $V$ to $\wt{ZK_1}$.  The above subrepresentations are decomposed into the following irreducible ones: 
\[V_{0,0}= V_{0,0}^2,\quad V_{1,0}=V_{1,0}^1\oplus V_{1}^2,\quad
V_{1,i}=V_{0,i}^2\oplus V_{1,i}^1\oplus V_{1}^2, \quad  i=1,2, \quad 
V_{k}=V_{k}^1\oplus V_{k}^2,\quad k\geq2.\]
\item Set 
\[\tilde{V}_{1,i}=V_{0,i}^2\oplus V_{1,i+1}^1, \quad i\in\BZ/3\BZ, \quad  \tilde{V}_{k+1}=V_{k}^2\oplus V_{k+1}^1, \quad k\geq1.\] 
Then (See the paragraph following Lemma \ref{irreIV} for details) \begin{itemize}
    \item $\tilde{V}_{1,1}=\tilde{V}_{1,1}^+\oplus \tilde{V}_{1,1}^-$ with  $\tilde{V}_{1,1}^+$ and $ \tilde{V}_{1,1}^-$ non-isomorphic irreducible $\wt{J}$-representations.
    \item all  $\tilde{V}_{k}$, $k\geq 2$ and $\tilde{V}_{1,i}$, $i\neq 1$, $\tilde{V}_{1,1}^\pm$ are irreducible non-isomorphic $\wt{J}$-representations except $\tilde{V}_{1,0}\cong \tilde{V}_{1,2}$.
\end{itemize}   
Moreover (see Lemma \ref{K-type II} for details), 
\[V|_{\wt{J}}=\left(\bigoplus_{i \in \BZ/3\BZ} \tilde{V}_{1,i}\right)\bigoplus \left(\bigoplus_{k\geq2}\tilde{V}_k^{\oplus 3}\right).\]  
\end{itemize}

To deduce $V^\prime|_{\wt{\CK}}$, we consider the interplay between  $V^\prime|_{\wt{J}}$ and $V^\prime|_{\wt{ZK}}$ via the restriction to $\wt{ZK_1}$. Starting from $V_{0,0}=V_{0,0}^2\subset V^\prime$, we have $V_{1,1}^1\subset V^\prime|_{\wt{ZK_1}}$ by the $\wt{J}$-module structure and then $V_{0,1}^2\subset V^\prime|_{\wt{ZK_1}}$
by the $\wt{ZK}$-module structure. Repeating this argument, we have (see Theorem \ref{K-type} for details)
\[V^\prime|_{\wt{ZK}}=V_{0,0}\bigoplus V_{1,1}\bigoplus\left(\bigoplus_{k\geq 2}V_k \right), \quad V^\prime|_{\wt{J}}=\tilde{V}_{1,1}^*\bigoplus \tilde{V}_{1,0}\bigoplus \left(\bigoplus_{k\geq2}\tilde{V}_k\right)\]
for some $*\in\{\pm\}$.

Once the decompositions of $\sigma_\pm|_{\wt{\CK}}$ and $\pi_\pm|_{\wt{\CK}}$ are known,  we deduce together with the upper bound $\Hom_{G}(\sigma\otimes\sigma_+\otimes\pi_-,\BC)=3$
that $\dim \Hom_{\wt{\CK}}(\tau_+\otimes\tau_-,\tau)=1$
for precisely one irreducible factor $\tau_{\pm}\subset \sigma_{\pm}|_{\wt{\CK}}$ by the very cuspidal property of $\tau$ (Note the multiplicity $3$ of $V_k$ and $\wt{V}_k$ in the decomposition of $V|_{\wt{\CK}}$ for $k\geq2$).
Now   $\dim \Hom_{G}(\sigma\otimes\sigma_+\otimes\sigma_-,\BC)=1$ follows from the Frobenius reciprocity law.

\section{Local co-period}
In this section, we study the   period in the local setting. We shall establish the multiplicity one property, construct the distinguished   period and conduct the unramified calculation. 

Let  $F$ be a local field containing all $3$-th roots of unity $\mu_3$. Fix an additive character $\psi:\ F\to\BC^\times$. Fix an embedding $\epsilon:\ \mu_3\to\BC^\times$. Let  $(-,-)_3: F^\times\times F^\times\to\mu_3(\xrightarrow{\epsilon}\BC^\times)$ be the $3$-th Hilbert symbol. Fix $c\in\BZ/3\BZ$ and let $d=\frac{3}{\gcd\{1+4c,3\}}$. Let  $G=\GL_2(F)$ and $Z$, $A$, $T=AZ$, $N$,$B=TN\subset G$ be the center, the torus $\diag\{*,1\}$, the diagonal torus, the upper unipotent radical and  the standard upper Borel  respectively. Let $Z^{(c)}=F^{\times,d}$, \[A_3=
\left\{\diag\{a,1\}|a\in F^{\times,3}\right\}\subset T_3=\{\diag\{a,b\}\in T| a,b\in F^{\times,3}\}\subset B_3=T_3N.\]
Let $\delta$ be the modulus character on $B$:  \[\delta:\ B\lra\BC^\times,\quad \begin{pmatrix}a & b \\ 0 & d\end{pmatrix}\mapsto \Big|\frac{a}{d}\Big|.\]

When $F$ is non-archimedean, let $\CO$ be the ring of integers, $\varpi$ be a uniformizer and let  $|\cdot|$ be the norm on $F^\times$ such that $|\varpi|^{-1}=\sharp\CO/\varpi$. Let $\zeta_F(s):=(1-|\varpi|^s)^{-1}$ be the usual $L$-factor. Let $dx$  be the self-dual Haar measure on $F$  with respect to $\psi$ and let $d^\times y=\zeta_F(1)\frac{dy}{|y|}$ be the Haar measure on $F^\times$. Note that $\Vol(\CO,dx)=\Vol(\CO^\times,d^\times x)=|\varpi|^{\frac{d(\psi)}{2}}$ where $-d(\psi)$ is the minimal integer $n$ such that $\psi$ is trivial on $\varpi^n\CO$.  Let $K=\GL_2(\CO)\subset \GL_2(F)$ be the maximal open compact subgroup. 

When $F$ is archimedean (thus $F\cong\BC$), assume $\psi(z):=e^{2\pi i(z+\bar{z})}$ and let $|z|=z\bar{z}$. The self-dual Haar measure $dx$ on $F$ with respect to $\psi$ is twice the usual Lebesgue measure. Let $d^\times y=\zeta_F(1)\frac{d y}{|y|}$ with  $\zeta_F(s)=2(2\pi)^{-s}\Gamma(s)$. Let $K=U(2)\subset G$ be the maximal compact open subgroup.

In all cases, let $dk$ be the Haar measure on $K$ with total volume $1$ and $dg$ be the left Haar measure on $\GL_2(F)$ given by rule \[dz\begin{pmatrix} y & x \\ 0 & 1\end{pmatrix}k=|y|^{-1}d^\times z dx d^\times y dk,\quad z\in Z,\ y\in F^\times,\ x\in F,\ k\in K.\] 
We will freely use the following integration formula in \cite[Section 3.1.6]{MV10}:
\begin{lem}\label{int}For any $f\in L^1(ZN\bs G,dg)$, one has 
\begin{align*}\int_{ZN\bs G}f(g)dg &=\int_K\int_{F^\times} f(\begin{pmatrix}y & 0\\0 & 1\end{pmatrix}k)\frac{d^\times y}{|y|} dk\\
&=\begin{cases}|\varpi|^{d(\psi)}\frac{\zeta_F(2)}{\zeta_F(1)}\int_{F^\times}\int_Ff\left[\begin{pmatrix} y & 0\\ 0 & 1\end{pmatrix} \begin{pmatrix} 1 & 0 \\ x & 1\end{pmatrix}\right]dx\frac{d^\times y}{|y|}, & \ \text{if $F$ is non-archimedean};\\
2\pi\int_{F^\times}\int_Ff\left[\begin{pmatrix} y & 0\\ 0 & 1\end{pmatrix} \begin{pmatrix} 1 & 0 \\ x & 1\end{pmatrix}\right]dx\frac{d^\times y}{|y|}, &\ \text{if $F=\BC$}.\end{cases}
\end{align*}
\end{lem}
\subsection{Metaplectic group}
Let $\wt{G}:=\wt{G}^{(c)}$ be the central extension $$0\lra \mu_3\stackrel{i}{\lra} \wt{G}^{(c)}\stackrel{p}{\lra} G\lra0$$
together with a preferred section $s:\ G\to \wt{G}^{(c)}$ determined by the Kubota 2-cocycle $\sigma_K:=\sigma_K^{(c)}$, i.e. $\wt{G}=G\times\mu_3$ equipped with the group structure 
$$(g_1,z_1)\cdot (g_2,z_2)=\left(g_1g_2,z_1z_2\sigma_K^{(c)}(g_1,g_2)\right), \quad g_1,g_2\in G,\ z_1,z_2\in\mu_3;$$
$$\sigma_K^{(c)}(g_1,g_2)=\left(\frac{K(g_1g_2)}{K(g_1)}, \frac{K(g_1g_2)}{K(g_2)\det(g_1)}\right)_3(\det(g_1),\det(g_2))_3^c;\quad K\left[\begin{pmatrix}a & b\\ c & d\end{pmatrix}\right]=
\begin{cases} c,\ &\ \text{if}\ c\neq0;\\
d,\ &\ \text{otherwise}.\end{cases}$$
Moreover $i(z)=(1,z)$, $p(g,z)=g$ and $s(g)=(g,1)$. Let $w=\begin{pmatrix} 0 & 1\\ 1 & 0\end{pmatrix}$, $\tau=\begin{pmatrix} 0 & 1\\ \varpi & 0\end{pmatrix}$. Via $s$, view $w$,$\tau$ as elements in $\wt{G}$ and $N$ as subgroup of $\wt{G}$.  

For  any subgroup $H\subset G$, let $\wt{H}:=p^{-1}(H)$. Note that the center of $\wt{G}$ (resp.$\wt{T}$ ) is $\wt{Z^{(c)}}$ (resp. $\wt{Z^{(c)}}\wt{T_3}$). Let $\wt{T_*}\subset \wt{T}$ be a maximal abelian subgroup. Set $\wt{Z_*}=\wt{T_*}\cap \wt{Z}$, which is also a maximal abelian subgroup in $\wt{Z}$. Set $B_*=T_*N$.

When $F$ is non-archimedean and $|3|=1$, we can and shall take 
\[Z_*=\Gamma_2\subset T_*=\{\diag\{a\delta,b\delta\}\in T|a,b\in \Gamma_1,\delta\in\Gamma_2\} \]
where $\Gamma_1=\CO^\times F^{\times,3}\subset \Gamma_2=\CO^\times F^{\times,d}\subset F^\times$.
Moreover in this case, the covering map $p:\ \wt{G}\to G$ admits the Kubota splitting on $K=\GL_2(\CO)$ $$\kappa:\ K\lra \wt{G},\quad g=\begin{pmatrix} a & b\\ c & d\end{pmatrix}\mapsto (g,k)\quad\text{ with } k=\begin{cases} (c,d/\det(g))_3, & \mathrm{if}\ 0<|c|<1;\\
1, & \mathrm{if}\ |c|=0,1.\end{cases}$$  

When $F$ is archimedean (thus $F=\BC$), the covering map $p:\ \wt{G}\to G$ splits. In this case, $\wt{G}=\mu_3\times G$, $\wt{B_*}=\mu_3\times B$,  $\wt{T_*}=\mu_3\times T$ and $\wt{Z_*}=\wt{Z}=\mu_3\times Z$.

For any subgroup $H\subset G$, a $\wt{H}$-representation will be called  $\epsilon$-genuine if $\mu_3$ acts on it via $\epsilon$. When $\epsilon$ is clear, we shall simply call it genuine. We shall freely view $H$-representations as $\wt{H}$-representations via the covering map $p:\ \wt{H}\to H$.
\subsection{Preliminaries on exceptional representations}
From now on, we assume $F$ is non-archimedean until Subsection \ref{Archimedean}. For any genuine character $\chi$ of $\wt{Z}^{(c)}\wt{T_3}$,  take an arbitrary extension  of $\chi$ to $\wt{T}_*$. Let $V(\chi)$ be the usual   smooth induction $$\left\{f:\ \wt{G}\to\BC\ \text{smooth}\ \Big| f(bg)=\chi(b)\delta^{1/2}(b)f(g),\ b\in\wt{B}_*,\ g\in\wt{G}\right\}.$$ 
 By \cite[Section I.1-2]{KP84}, \begin{itemize}
    \item $V(\chi)$  depends only on $\chi$ (independent of the extension);
    \item the smooth dual $V(\chi)^\vee$ of $V(\chi)$ is isomorphic to $V(\chi^{-1})$;
    \item Unless $\chi(s(\diag\{x^3,x^{-3})\})=|x|^\pm$ for all $x\in F^\times$, $V(\chi)$ is irreducible;
    \item If $\chi$ is {\em exceptional}, i.e. $\chi(s(\diag\{x^3,x^{-3}\}))=|x|$ for all $x\in F^\times$, there exists an exact sequence
    $$0\lra V_0(\chi^w)\lra V(\chi)\lra V_0(\chi)\lra 0$$
    where $V_0(\chi)$ is the unique irreducible quotient (resp. subrepresentation) of $V(\chi)$ (resp. $V(\chi^w)$) and  $V_0(\chi^w)$ is the unique irreducible quotient (resp. subrepresentation) of $V(\chi^w)$ (resp. $V(\chi)$).  The representation $V_0(\chi)$ is the \em{exceptional representations} of $\wt{G}$ associated to $\chi$.
    \item For exceptional characters $\chi$, $\mu$ of $\wt{Z}^{(c)}\wt{T_3}$, $V_0(\chi)\cong V_0(\mu)^\vee$ if and only if $\chi\mu^w=1$.
    \end{itemize}  
    \subsubsection{Jacquet module and twisted Jacquet module}
    For any $\wt{G}$-representation $\sigma$, the Jacquet module and $\psi$-twisted Jacquet module of $\sigma$ are \[\sigma_N:=\sigma/\langle \sigma(n)v-v,\ \forall n\in N, v\in\sigma\rangle,\quad \sigma_\psi:=\sigma/\langle \sigma(n)v-\psi(n)v,\ \forall n\in N, v\in\sigma\rangle\]
 Clearly, $\sigma_N$ (resp. $\sigma_\psi$) is a $\wt{B}$-module (resp. $\wt{Z}N$-module). 
 
 Since the center of $\wt{Z}$ (resp. $\wt{T}$) is $\wt{Z^{(c)}}$ (resp.$\wt{Z^{(c)}}\wt{T}_3$), all genuine representations of $\wt{Z}$ (resp. $\wt{T}$) have the form  $I_{\wt{Z_*}}^{\wt{Z}}\chi$ (resp. $I_{\wt{T_*}}^{\wt{T}}\chi$) for some character $\chi$ on $\wt{Z_*}$ (resp. $\wt{T_*}$) by the Stone-Von Neumann theory in \cite[Section 0.3]{KP84}. They depend only on $\chi|_{\wt{Z^{(c)}}}$ (resp. $\chi|_{\wt{Z^{(c)}T_3}}$). In particular, all 
 genuine representations of $\wt{Z}$ (resp. $\wt{T}$)  have dimension $[\wt{Z}:\wt{Z^{(c)}}]^{1/2}=d|d|^{-1/2}$ (resp. $3d|3d|^{-1/2}$).
  Assume $\sigma$ has central character $\omega$ and extend $\omega$ to $\wt{Z_*}$. Then $\sigma_\psi$ is the direct sum of $I_{\wt{Z_*}N}^{\wt{Z}N}\omega\boxtimes\psi$.
\begin{prop}\label{Jac}Let $\chi$ be an exceptional character. Then
\begin{itemize}
\item $V(\chi)_N=I_{w\wt{T}_*w}^{\wt{T}}(\chi^w\delta^{1/2})\oplus I_{\wt{T}_*}^{\wt{T}}(\chi\delta^{1/2})$.
\item $V_0(\chi)_N$ is the irreducible $\wt{T}$-representation $I_{w\wt{T}_*w}^{\wt{T}}(\chi^w\delta^{1/2})$. 
\item $V(\chi)_\psi$ has dimension $[\wt{T}: \wt{T}_*]=3d|3d|^{-1/2}$ and decomposes into $3|3|^{-1/2}$ isomorphic irreducible genuine $\wt{Z}$-representations.
\item   $V_0(\chi)_\psi$ is an irreducible genuine $\wt{Z}$-representation.
\end{itemize}
\end{prop}
\begin{proof}The result on $V(\chi)$ is \cite[Lemma I.2.1]{KP84}. The result on $V_0(\chi)_N$ follows from \cite[Theorem I.2.9(e)]{KP84}. The dimension result on $V(\chi)_\psi$ is \cite[Lemma I.3.2]{KP84}. Since the action of $\wt{Z^{(c)}}$ is fixed, we  deduce the $\wt{Z}$-module structure of $V(\chi)_\psi$.  When $|3|=1$,  $\dim V_0(\chi)_\psi=d$ by \cite[Theorem I.3.5]{KP84}.  Then one deduces $V_0(\chi)_\psi$ is an irreducible $\wt{Z}$-representation. Then we can take the $\wt{Z}$-module structure into consideration and slightly generalize \cite[Theorem II.5(b)]{KP84} to deduce $V_0(\chi)_\psi$ is irreducible even when $|3|<1$.
\end{proof}
Actually by \cite[Lemma I.3.2]{KP84}, $\dim V(\chi)_\psi=[\wt{T}: \wt{T_*}]=3d|3d|^{-1/2}$  for all characters $\chi:\ \wt{T_*}\to\BC^\times$. Moreover one can identify   $(I_{\wt{B_*}}^{\wt{G}}\chi)_{\psi}^\vee$ with the space of functions 
\[S(\chi,\psi):=\{C:\ \wt{T}\to\BC| C(ht)=(\delta^{1/2}\chi(h))^{-1}C(t)\quad\forall h\in \wt{T_*},\ t\in \wt{T}\}\] by the rule
\[(C, f):=\sum_{\eta\in\wt{T_*}\bs\wt{T}}C(\eta)\int_{N}\bar{\psi}(n)f(\eta w n) dn.\]

Assume $|3|=1$ and $\psi$ has conductor $0$. By \cite[Corollary I.3.4]{KP84}\footnote{The formula in \text{loc.cit} has some typos. We corrected the formula following their approach.}, $C\in (V_0(\chi))_\psi^\vee$ if and only if 
\[\label{cons}\ C(\eta_{b-1,a+1})=|\varpi|^{-(b-a-2-\lfloor{\frac{a-b}{3}}\rfloor)}g^{(b-a-1),\psi}C(\eta_{a,b}),\quad \eta_{a,b}=s(\begin{pmatrix}\varpi^a & 0\\ 0 & \varpi^b\end{pmatrix})\quad\ \forall (a,b)\in\BZ^2.\]
Here $\lfloor{x}\rfloor$ means the maximal integer smaller or equal to $x$ and  $g^{(i),\psi}=|\varpi|^{-1}\int_{\CO^\times} (\varpi,x)_3^i\psi(\frac{x}{\varpi})d^\times x$ is the Gauss sum.  
Note that $g^{(i),\psi}g^{(-i),\bar{\psi}}=|\varpi|^{-1}.$
\subsubsection{Whittaker model}
Take any smooth admissible $\wt{G}$-representation $\sigma$. Consider the induced representations \begin{align*}
     I_{\wt{Z}N}^{\wt{G}}\sigma_\psi:&=\{f:\wt{G}\to \sigma_\psi\mid f(ng)=\sigma_\psi(n)f(g)\quad \forall\ n\in\wt{Z}N, g\in\wt{G}\}\\
      I_{\wt{Z}N}^{\wt{B}}\sigma_\psi:&=\{f:\wt{B}\to \sigma_\psi\mid f(nb)=\sigma_\psi(n)f(b)\quad \forall\ n\in\wt{Z}N, b\in\wt{B}\}.
 \end{align*}
Then one has a  $\wt{B}$-equivariant  surjection
$$I_{\wt{Z}N}^{\wt{G}}\sigma_\psi\to I_{\wt{Z}N}^{\wt{B}}\sigma_\psi;\quad f\mapsto f|_{\wt{B}}.$$
 Let $\ell: \sigma\to\sigma_\psi$ be the natural quotient map. When $\sigma$ is irreducible and $\sigma_\psi\neq0$,  the  map 
 $$ \sigma\to I_{\wt{Z}N}^{\wt{G}}\sigma_\psi;\quad v\mapsto(W_v:\ g\mapsto \ell(g\cdot v))$$
 is injective and $\wt{G}$-equivariant. The image $\CW(\sigma,\psi)$ is called the Whittaker model of $\sigma$ with respect to $\psi$. Composed with the restriction map, one obtains the $\wt{B}$-equivariant map $\sigma\to I_{\wt{Z}N}^{\wt{B}}\sigma_\psi$. Note that $I_{\wt{Z}N}^{\wt{B}}\sigma_\psi$ has the sub-$\wt{B}$-representation $i_{\wt{Z}N}^{\wt{B}}\sigma_\psi$. Here $i_{\wt{ZN}}^{\wt{B}}$ stands for the compact induction.

 \begin{lem}\label{Kiri}Assume $\sigma$ is irreducible and $\sigma_\psi\neq0$. Then $\sigma^N=0$ and 
 the morphism $\sigma\to I_{\wt{Z}N}^{\wt{B}}\sigma_\psi$
is injective. Denote the image by $\CK(\sigma,\psi)$. Then one has an exact sequence of $\wt{B}$-representations
$$0\lra i_{\wt{Z}N}^{\wt{B}}\sigma_\psi\lra\CK(\sigma,\psi)\lra \sigma_N\lra 0.$$
\end{lem}
\begin{proof}Suppose $\sigma^N\neq 0$ and take $0\neq v\in \sigma^N$. Since $N$ is normal in $\tilde{B}$, one has  $\tilde{B}v\subset \sigma^N$. Let $\Stab(v)$ be the stabilizer of $v$. As $\Stab(v)$ is open, $\Stab(v)\cap \tilde{B}wN\neq\emptyset$. This implies  $\tilde{B}wN v\subset \sigma^N$.  By the Bruhat decomposition $\wt{G}=\wt{B}\cup \wt{B}\omega N$, we deduce  $\wt{G}\cdot v=\sigma\subset \sigma^N$.  Therefore $\sigma_\psi\subset (\sigma^N)_\psi=0$. Hence if $\sigma_\psi\neq0$, one has $\sigma^N=0$. To show the morphism \[\sigma\to I_{\wt{Z}N}^{\wt{B}}\sigma_\psi,\quad v\mapsto\xi_v\] is injective,  it suffices to show   $\varphi(x):=\sigma\begin{pmatrix} 1 & x \\ & 1 
\end{pmatrix}v$ is constant on $F$ if $\xi_v=0$. We follow the strategy in \cite[Theorem 3.1]{PR01}. 
 Let $\varphi_M(t):=\int_{\mathfrak{p}^{-M}}\bar{\psi}(tx)\varphi(x)dx$. Note that the function $u\mapsto \sigma\begin{pmatrix}
    u & \\
    & 1
\end{pmatrix}v$ is constant for $u\in 1+\mathfrak{a}$ with some ideal $\mathfrak{a}\subset \mathcal{O}_F$. Then for $\forall t\in F, M\in \mathbb{N}$, and $u\in 1+\mathfrak{a}$,
$$\varphi_M(tu)=\int_{\mathfrak{p}^{-M}}\bar{\psi}(tux)\sigma\begin{pmatrix} 1 & x\\ & 1
\end{pmatrix}v dx=\int_{\mathfrak{p}^{-M}}\bar{\psi}(tx)\sigma\begin{pmatrix}1 & u^{-1}x\\ & 1
\end{pmatrix}v dx$$
$$=\sigma\begin{pmatrix} u^{-1} & \\  & 1
\end{pmatrix}\int_{\mathfrak{p}^{-M}}\bar{\psi}(tx)\sigma\left[\begin{pmatrix}1 & x\\ & 1
\end{pmatrix}\begin{pmatrix}u & \\ & 1
\end{pmatrix}\right]v dx=\sigma\begin{pmatrix}u^{-1} & \\ & 1
\end{pmatrix}\int_{\mathfrak{p}^{-M}}\bar{\psi}(tx)\sigma\begin{pmatrix} 1 & x\\ & 1
\end{pmatrix}v dx.$$
This shows that
$\varphi_M(tu)=\sigma\begin{pmatrix} u^{-1} & \\ & 1\\
\end{pmatrix}\varphi_M(t)$ for all $t\in F$, $M\in \mathbb{N}$ and $u\in 1+\mathfrak{a}$. By definition $\xi_v=0$ if and only if $\sigma(s(\diag\{t,1\}))v\in \ker \ell$, $\forall\ t\in F^\times$. Note that $v\in \ker \ell$ if and only if for large enough $M$,
$$\int_{\fp^{-M}}\bar{\psi}(x)\sigma\left[\begin{pmatrix} 1 & x \\ 0 & 1\end{pmatrix}\right]vdx=0.$$
  Hence if $\xi_v=0$, one has 
$$\int_{\fp^{-M}}\bar{\psi}(tx)\varphi(x)dx=\varphi_M(t)=0$$
for all large enough $M$ depending on $t$. Since $1+\mathfrak{a}$ has finite index in any compact subset $C\subset F^\times$, we see that $\varphi_M(t)=0$ for all $t\in C$ and sufficiently large $M$. On the other hand, for any $\phi\in\CS(F)$, we may assume that $\phi$ vanishes outside $\mathfrak{p}^{-M}$ for sufficiently large $M$, by Fourier inversion formula,
$$\int_{F}\phi(x)\varphi(x)dx=\int_{\mathfrak{p}^{-M}}\varphi(x)\int_{F}\hat{\phi}(t)\bar{\psi}(tx)dt dx=\int_{F}\varphi_M(t)\hat{\phi}(t)dt=0.$$
This implies that $\varphi(x)$ is constant, so $\sigma(n)v=v$ for all $n\in N$ and consequently, $v=0$. 
\end{proof}
Now we consider the Jacquet module of $\sigma\otimes\pi$ where $\sigma$ and $\pi$ are irreducible  $\wt{G}$-representations. Note that the $\wt{Z}N$-action on $\sigma_\psi\otimes\pi_{\bar{\psi}}$ factors through $\wt{Z}$.
\begin{prop}\label{Jac of tensor}Let $\sigma$ and $\pi$ be irreducible  $\wt{G}$-representations. Then 
one has an exact sequence of $\wt{B}$-representations (for the diagonal $\wt{B}$-action)
\[\quad 0\lra i_{\wt{Z}N}^{\wt{B}}(\sigma_\psi\otimes\pi_{\bar{\psi}})\lra (\sigma\otimes\pi)_N\lra \sigma_N\otimes\pi_N\lra 0.\]
\end{prop}
\begin{proof}By \cite[lemma 6.3]{PT11}, we have 
\[\big(i_{\wt{Z}N}^{\wt{B}}\sigma_\psi\otimes i_{\wt{Z}N}^{\wt{B}}\pi_{\bar{\psi}}\big)_N \cong i_{\wt{Z}N}^{\wt{B}}(\sigma_\psi\otimes\pi_{\bar{\psi}}).\]
 Since $(i_{\wt{Z}N}^{\wt{B}}\sigma_\psi\otimes \pi_N)_N=0$ and $(\sigma_N\otimes\pi)_N=\sigma_N\otimes\pi_N$, we deduce from  Lemma \ref{Kiri} that 
\[\big(i_{\wt{Z}N}^{\wt{B}}\sigma_\psi\otimes i_{\wt{Z}N}^{\wt{B}}\pi_{\bar{\psi}})_N=(i_{\wt{Z}N}^{\wt{B}}\sigma_\psi\otimes\pi)_N\] and we have an exact sequence of $\wt{B}$-representations
\[0\lra(i_{\wt{Z}N}^{\wt{B}}\sigma_\psi\otimes\pi)_N\lra(\sigma\otimes \pi)_N\lra \sigma_N\otimes\pi_N\lra 0. \]
Combing all these together, we are done.
\end{proof}
\subsubsection{Invariant pairing}\label{InvP}
In this subsection, we explicitly exhibit several $\wt{G}$-invariant pairings on $V_0(\chi)\times V_0(\chi)^\vee$ for an exceptional character $\chi$. We start with a general result: 
\begin{prop}\label{Inv}Let $\sigma$ be any irreducible smooth admissible $\wt{G}$-representation. If  $\sigma_\psi$ is an irreducible $\wt{Z}N$-representation, then any  $\wt{B}$-invariant pairing on $\sigma\times\sigma^\vee$ is automatically $\wt{G}$-invariant.
\end{prop}
\begin{proof}The assumption $\sigma_\psi$ is $\wt{Z}$-irreducible implies that the $\wt{B}$-representation $\CS(F^\times, \sigma_\psi)=i_{\wt{ZN}}^{\wt{B}}\sigma_\psi$ is irreducible. Thus $\dim\Hom_{\wt{B}}(i_{\wt{ZN}}^{\wt{B}}\sigma_\psi\otimes  i_{\wt{ZN}}^{\wt{B}}\sigma_{\bar{\psi}}^\vee,\BC)=1$. Easy to check that 
\[\Hom_{\wt{B}}(\sigma_N\otimes \sigma_{N}^\vee,\BC)= \Hom_{\wt{B}}(i_{\wt{ZN}}^{\wt{B}}\sigma_\psi\otimes \sigma_N^\vee,\BC)=\Hom_{\wt{B}}(\sigma_N\otimes i_{\wt{ZN}}^{\wt{B}}\sigma_{\bar{\psi}}^\vee,\BC)=0.\]
Thus by the exact sequence in Lemma \ref{Kiri},   $\dim\Hom_{\wt{B}}(\sigma\otimes\sigma^\vee,\BC)=1$. Consequently, the natural injection
\[\Hom_{\wt{G}}(\sigma\otimes\sigma^\vee,\BC)\hookrightarrow \Hom_{\wt{B}}(\sigma\otimes\sigma^\vee,\BC)\]
is actually an isomorphism and the stated result follows.
\end{proof}
Assume the central character $\omega$ of $\sigma:=V_0(\chi)$ is unitary and   extend $\omega$ to $\wt{Z_*}$. Realize $\sigma_{\psi}$ (resp. $\sigma_{\bar{\psi}}^{\vee}$) as $I_{\wt{Z_*}N}^{\wt{Z}N}\omega\boxtimes\psi$ (resp. $I_{\wt{Z_*}N}^{\wt{Z}N}\omega^{-1}\boxtimes\bar{\psi}$ and embed $\sigma$ (resp. $\sigma^\vee$) into  $I_{\wt{Z}N}^{G}\sigma_\psi=I_{\wt{Z_*}N}^{\wt{G}}\omega\boxtimes\psi$ (resp. $I_{\wt{Z_*}N}^{\wt{G}}\omega^{-1}\boxtimes\bar{\psi}$). By the gauge estimation in \cite[Theorem I.4.1]{KP84}, the integration \[\int_{\wt{Z_*}N\bs \wt{B}}WW^\vee(b)db=\int_{F^\times}\sum_{z\in \wt{Z_*}\bs \wt{Z}}WW^\vee\left[\begin{pmatrix} y & 0\\ 0 & 1\end{pmatrix}z\right]d^\times y\] is absolutely convergent for any $W\in\sigma$ and $W^\vee\in\sigma^\vee$. Thus one has the $\wt{B}$-invariant pairing 
\[\sigma\times\sigma^\vee\lra\BC,\quad (W,W^\vee)\mapsto\int_{F^\times}\sum_{z\in \wt{Z_*}\bs\wt{Z}}WW^\vee\left[\begin{pmatrix}y & 0\\ 0 & 1\end{pmatrix}z\right]d^\times y.\]
Since $\sigma_{\psi}$ is irreducible, it is actually $\wt{G}$-invariant by Proposition \ref{Inv}.

Finally realize $V_0(\chi)$ as a subrepresentation of $V(\chi^w)=I_{\wt{B_*}}^{\wt{G}}\chi^w$.  
 The pairing 
\[V(\chi^w)\times V(\chi^w)^\vee\lra\BC,\quad  (f,f^\vee)\mapsto\sum_{t\in \wt{T_*}\bs\wt{T}}\delta^{-1}(t)\int_Kf(tk)f^\vee(tk)dk.\]
is non-degenerate and $\wt{G}$-invariant. Note that the natural restriction map is an isomorphism
\[\Hom_{\wt{G}}(V(\chi^w)\otimes V(\chi^w)^\vee,\BC)\cong \Hom_{\wt{G}}(V_0(\chi)\otimes V_0(\chi)^\vee,\BC).\]
 Thus the pairing on $V(\chi^w)\times V(\chi^w)^\vee$ induces a non-degenerate $\wt{G}$-invariant pairing on $V_0(\chi)\times V_0(\chi)^\vee$.
\subsection{Restriction to maximal compact-modulo-center subgroup}\label{K-type Section}
In this subsection, we assume $|3|=1$. For $m\geq 1$, let 
\[K_m=\left\{\begin{pmatrix} a & b \\ c & d\end{pmatrix}\in K\Bigg| \varpi^m\mid c\right\},\quad K_{1,m}=\left\{\begin{pmatrix}a & b \\ c & d\end{pmatrix}\in K\Bigg| \varpi\mid c,\ \varpi^m\mid b\right\}.\]
Let $J=ZK_1\sqcup Z\tau K_1$ with $\tau=\begin{pmatrix} 0 & 1\\ \varpi & 0\end{pmatrix}$. We shall consider the restriction to $\wt{ZK}$ and $\wt{J}$ of the exceptional representation $V_0(\chi)$.

Let $\kappa:\ K\to\wt{G}$ be the Kubota lifting. Let $B(\CO)^*=\kappa(B(\CO))\subset K^*=\kappa(K)$ and for $m\geq1$, $K_m^*=\kappa(K_m)$ and $K_{1,m}^*=\kappa(K_{1,m})$. A character $\chi$ on $\wt{Z}^{(c)}\wt{T_3}$ is  {\em unramified} if $\chi(s(h^3))=1$ for all $h\in K\cap T$ and moreover    {\em normalized} if $\chi$ is trivial on $\wt{Z}\cap K^*$. 
\begin{itemize}
\item By \cite[Lemma I.1.1]{KP84}, when $\chi$ is unramified, there exists a character $\theta:\ F^\times\to\BC^\times$ such that $\chi\cdot \theta\circ\det\circ p$ is normalized,
\item When $\chi$ is normalized, then it  admits a unique extension $\wt{T}_*\to\BC^\times$, the canonical extension,  which is   trivial on $\wt{T}_*\cap K^*$.
\end{itemize}
By twisting, we can and will  make the exceptional character $\chi$  unramified and normalized. Moreover, we will take the canonical extension to  define the model $V_0(\chi)$.

 Note that $$\wt{G}=\bigsqcup_{(a,b)\in(\BZ^2/3\BZ^2)/\sim}\wt{B_*}s(\eta_{a,b})K^*,\quad \eta_{a,b}=\begin{pmatrix}\varpi^a & 0\\ 0 & \varpi^b\end{pmatrix}\in G.$$
Here $(a,b)\sim(a^\prime,b^\prime)$ if $d\mid a-a^\prime=b-b^\prime$. Thus \[V(\chi)|_{K^*}=\bigoplus_{(a,b)\in(\BZ^2/3\BZ^2)/\sim}I_{B(\CO)^*}^{K^*}\chi_{a,b}\]
where  the character $\chi_{a,b}$ on $B(\CO)^*$
has the form 
\[\chi_{a,b}(g):=\chi^{\eta_{a,b}}(g):=\chi(s(\eta_{a,b})gs(\eta_{a,b})^{-1})=(\varpi^{b+2c(a+b)}, x)_3(\varpi^{a+2c(a+b)},y)_3, \quad g=\left[\begin{pmatrix} x & n \\ 0 & y\end{pmatrix},1\right].\]
Easy to check that $\chi_{a,b}$ extends to $K_m^*$, $m\geq 1$ by the rule \[K_m^*\lra \BC,\quad \kappa\left[\begin{pmatrix} x & n \\ m & y\end{pmatrix}
\right]
\mapsto (\varpi^{b+2c(a+b)}, x)_3(\varpi^{a+2c(a+b)},y)_3.\]
Moreover it  extends to $K^*$ if and only if 
\[\tag{$**$} \quad 3\mid a-b,\quad d\mid a.\]
In this case, we denote the resulting character  by $\mu_{0,a,b}$. 
 \begin{itemize}
     \item For $m\geq 1$, let $\mu_{m+1,a,b}$ be the complementary of $I_{K_m^*}^{K^*} \chi_{a,b}$ in $I_{K_{m+1}^*}^{K^*}\chi_{a,b}$. 
     \item When   $(**)$ holds, let $\mu_{1,a,b}$ be the complementary of $\mu_{0,a,b}$ in $I_{K_1^*}^{K^*}\chi_{a,b}$.
     \item When $(**)$ fails, let $\mu_{1,a,b}=I_{K_1^*}^{K^*} \chi_{a,b}$ and $\mu_{0,a,b}=0$.
 \end{itemize}  
By \cite{Cas73}, each $\mu_{k,a,b}$ is irreducible  and 
$I_{B(\CO)^*}^{K^*}\chi_{a,b}=\bigoplus_{k\geq0}\mu_{k,a,b}$.
\begin{lem}\label{irreI}$\mu_{k,a,b}\cong \mu_{k^\prime,a^\prime,b^\prime}$ if and only if 
\begin{itemize}
    \item $k=k^\prime\geq2$ and $a+b\equiv a^\prime+b^\prime \mod d$; 
    \item $k=k^\prime=1$, $a-b\equiv a^\prime-b^\prime \mod 3$, $d\mid a-a^\prime$ or \ $a-b\equiv b^\prime-a^\prime \mod 3$, $d\mid a-b^\prime$;
    \item $k=k^\prime=0$, $a\equiv b\equiv a^\prime\equiv b^\prime \mod 3$ and $d\mid a-a^\prime$.
\end{itemize} 
 Moreover,  $\dim \mu_{k,a,b}^{K_m^*}=\begin{cases} 1, &\ \text{if }2\leq k\leq m,\ d\mid a+b;\\ 
 1, &\ \text{if } k=0,1,\ 3\mid a-b,\ d\mid a;\\
0, &\ \text{otherwise}.\end{cases}$
\end{lem}
\begin{proof}Take positive integers $m^\prime,m$, set $m^{\prime\prime}=\min\{m^\prime,m\}$ and 
\[K=K_{m^{\prime\prime}}\bigsqcup K_{m^\prime} w K_m \bigsqcup \left(\bigsqcup_{j=1}^{m^{\prime\prime}-1} K_{m^\prime} w_j K_m\right),\quad w=\begin{pmatrix} 0 & 1\\ 1 & 0\end{pmatrix},\ w_j=\begin{pmatrix}1 & 0\\ \varpi^j & 1\end{pmatrix}.\]
By the Frobenius reciprocity law, 
\[\dim\ (I_{K_{m^\prime}^*}^{K^*}\chi_{a,b})^{K_m^*}=\begin{cases} m^{\prime\prime}+1,\quad &\ \text{if $(**)$ holds}; \\
m^{\prime\prime}-1,\quad &\ \text{if $d\mid a+b$, $(**)$ fails};\\ 
0,\quad &\  \text{otherwise}.\end{cases}\]
More generally for $(a,b),(a^\prime,b^\prime)\in(\BZ/3\BZ)^2$, consider 
the following conditions:
\begin{enumerate}
\item[(I)] $a-b\equiv a^\prime-b^\prime \mod 3,\ d\mid a-a^\prime$.
\item[(II)] $a-b\equiv b^\prime-a^\prime \mod 3,\ d\mid a-b^\prime$.
\item[(III)] $a+b\equiv a^\prime+b^\prime \mod d$.
\end{enumerate}
We have
\[\dim \Hom_{K^*}(I_{K_{m^\prime}^*}^{K^*}\chi_{a,b}, I_{K_m^*}^{K^*}\chi_{a^\prime,b^\prime})=\begin{cases} m^{\prime\prime}+1,\quad & \text{if both (I) and (II) hold};\\
m^{\prime\prime},\quad & \text{if only one of  (I) or (II) holds};  \\
m^{\prime\prime}-1,\quad &  
\text{if none of  (I) and (II) holds, but (III) holds};  \\
0, \quad &\text{otherwise.}\end{cases}\]
The stated results follow immediately.
\end{proof}
To determine $\mu_{k,a,b}|_{K_1^*}$, consider the following $K_1^*$-representations:
\begin{itemize}
\item Take $\mu_{0,a,b}^1=0$ and $\mu_{1,a,b}^1=\chi_{a,b}$. Take $\mu_{0,a,b}^2=\chi_{a,b}^w$.
\item For $m\geq2$, let $\mu_{m,a,b}^1$ be the complementary of $I_{K_{m-1}^*}^{K_1^*}\chi_{a,b}$ in $I_{K_{m}^*}^{K_1^*}\chi_{a,b}$.  
\item  For $m\geq1$, let $\mu_{m,a,b}^2$ be the complementary of $I_{K_{1,m-1}^*}^{K_1^*}\chi^w_{a,b}$ in $I_{K_{1,m}^*}^{K_1^*}\chi^w_{a,b}$ with \[\chi_{a,b}^w:\ K_m^*\lra \BC,\quad \kappa\left[\begin{pmatrix} x & n \\ m & y\end{pmatrix}\right]
\mapsto (\varpi^{b+2c(a+b)}, y)_3(\varpi^{a+2c(a+b)},x)_3.\]
 
\end{itemize}

\begin{lem}\label{irreII} All $\mu_{k,a,b}^i$, $i=1,2$, $k\geq1$ are irreducible $K_1^*$-representations. Moreover, we have the following
\begin{itemize}
 \item when $k\geq 2$ or $k=0,1$ and (**) holds, $\mu_{k,a,b}|_{K_1^*}=\mu_{k,a,b}^1\oplus \mu_{k,a,b}^2$.
 \item when $k=1$ and (**) fails, $\mu_{k,a,b}|_{K_1^*}=\mu_{0,a,b}^2\oplus \mu_{k,a,b}^1\oplus \mu_{k,a,b}^2$.
\item $\mu_{k,a,b}^i\cong \mu_{k^\prime,a^\prime,b^\prime}^{i^\prime}$ if and only if 
\begin{itemize}
\item $(k,i)=(k^\prime,i^\prime)=(0,2), (1,1)$ and (I) holds;
 \item $(k,i)=(k^\prime,i^\prime)\neq (0,2), (1,1)$ and $a+b\equiv a^\prime+b^\prime\ \mod d$;
\item $(k,i)=(0,2)$, $(k^\prime,i^\prime)=(1,1)$ and (II) holds.
   \end{itemize}
\item $\dim \mu_{k,a,b}^{i,K_m^*}=\begin{cases} 1, & \text{if } i=1, 2\leq k\leq m,\ d\mid a+b;\\ 
 1, & \text{if } (k,i)=(0,2), (1,1),\ 3\mid a-b,\ d\mid a;\\
0, & \text{otherwise.}\end{cases}$
\end{itemize}
\end{lem}
\begin{proof}
Note that $K=K_1\sqcup K_1wK_1$ and $K_1=K_{1,i} K_{j}$ for $i,j\geq1$.
\[K_1=\bigsqcup_{i=1}^m K_mw_iK_m=\bigsqcup_{i=0}^mK_{1,m}w^iK_{1,m},\quad w_i=\begin{pmatrix} 1 & 0\\ \varpi^i & 1\end{pmatrix},\quad w^i=ww_iw.\]
The stated results follow from the following dimension formulas \[\dim\Hom_{K_1^*}(I_{K_m^*}^{K_1^*}\chi_{a,b}, I_{K_{m}^*}^{K_1^*}\chi_{a^\prime,b^\prime})=\begin{cases} m, & \text{if (I) holds};\\
m-1, & \text{if (I) fails but (III) holds};\\
0, & \text{otherwise.}\end{cases}\]
\[\dim\Hom_{K_1^*}(I_{K_m^*}^{K_1^*}\chi_{a,b}, I_{K_{1,m}^*}^{K_1^*}\chi^w_{a^\prime,b^\prime})=\begin{cases} 1, & 
\text{ if (II) holds};\\
0, & \text{otherwise.}\end{cases}\]
\[\dim\Hom_{K_1^*}(I_{K_{1,m}^*}^{K_1^*}\chi^w_{a,b}, I_{K_{1,m}^*}^{K_1^*}\chi^w_{a^\prime,b^\prime})=\begin{cases} m+1, & 
\text{if (I) holds};\\
m, & \text{if (I) fails but (III) holds};\\
0, & \text{otherwise.}\end{cases}\]
\end{proof}
  Extend $\mu_{k,a,b}$ and $I_{K_m^*}^{K^*}\chi_{a,b}$  to  $\wt{Z_*K}$-representations using  the central character of $V(\chi)$. 
\begin{lem}\label{irreIII}$V_{k,a,b}:=I_{\wt{Z_*K}}^{\wt{ZK}}\mu_{k,a,b}$ is irreducible and \[V_{k,a,b}|_{\wt{Z_*K}}=\bigoplus_{i=0}^{d-1}\mu_{k,a+2i,b+2i},\quad \dim V_{k,a,b}^{K_m^*}=\begin{cases} 1, &\ 
\text{if } 2\leq k\leq m;\\ 1, &\ \text{if } k=0,1,\ 3\mid a-b;\\
0,  & \text{otherwise}.\end{cases}\]
Moreover, 
\begin{itemize}
    \item when $k=0,1$, $V_{k,a,b}\cong V_{k,a^\prime,b^\prime}$ if and only if  $a-b\equiv \pm(a^\prime-b^\prime) \mod n$;
    \item when $k\geq2$, $V_{k,a,b}$ is independent of $a,b$.
\end{itemize}

\end{lem}

\begin{proof}Note that 
\[\wt{ZK}=\bigsqcup_{i=0}^{d-1} \wt{Z_*K}s\left[\begin{pmatrix}\varpi^i & 0\\ 0 & \varpi^i\end{pmatrix}\right]\wt{Z_*K}.\]
thus $I_{\wt{Z_*K}}^{\wt{ZK}}I_{K_m^*}^{K^*}\chi_{a,b}|_{\wt{\Gamma_2K}}=\bigoplus_{i=0}^{d-1}I_{K_m^*}^{K^*}\chi_{a+2i,b+2i}$.
All stated result follow from Lemma \ref{irreI}.
\end{proof}
For $k=0,1$, we shall   denote $V_{k,a,b}$ by $V_{k,a-b}$ and for $k\geq2$,   denote $V_{k,a,b}$ by $V_k$. Note that $V_{0,i}=0$ for $i\neq0$. Then except $V_{k,i}\cong V_{k,-i}$, $k=0,1$, these  $V_{k,i}$  and $V_k$ are mutually distinct irreducible $\wt{ZK}$-representations.  We can now determine $V(\chi)|_{\wt{ZK}}$.

\begin{lem}\label{K-type I} $V(\chi)|_{\wt{ZK}}=V_{0,0}\bigoplus\left(\bigoplus_{i\in\BZ/3\BZ}V_{1,i}\right)\bigoplus \left(\bigoplus_{k\geq2}V_k^{\oplus 3}\right)$.
\end{lem}
\begin{proof}Note that \[V(\chi)|_{K^*}=\bigoplus_{(a,b)\in(\BZ^2/3\BZ^2)/\sim}I_{B(\CO)^*}^{K^*}\chi_{a,b}=\bigoplus_{(a,b)\in(\BZ^2/3\BZ^2)/\sim}\bigoplus_{k\geq0}\mu_{k,a,b}\]
where $(a,b)\sim (a^\prime,b^\prime)$ if $d\mid a-a^\prime=b-b^\prime$. Thus 
\[\dim \Hom_{\wt{ZK}}(I_{\wt{Z_*K_m}}^{\wt{ZK}}\chi_{a,b}, V(\chi))=\sum_{(a^\prime,b^\prime)\in(\BZ/3\BZ)^2/\sim}\dim\Hom_{\wt{Z_*K_m}}(\chi_{a,b}, I_{\wt{Z_*B(\CO)}}^{\wt{Z_*K}}\chi_{a^\prime,b^\prime})=3(m-1)+2.\]
Hence by Lemma \ref{irreIII} and the discussions above, in $V_0(\chi)| _{\wt{ZK}}$, 
\begin{itemize}
    \item $V_k$, $k\geq 2$ occurs precisely $3$-times,
    \item  $V_{k,0}$, $k=0,1$ occurs precisely once,
    \item and $V_{1,i}$, $i\neq0$ occurs precisely twice.
\end{itemize} 
The desired result follows immediately.
\end{proof}

\begin{lem}\label{irreIV}$V_{k,a,b}^i:=I_{\wt{Z_*K}}^{\wt{ZK}}\mu^i_{k,a,b}$ is irreducible and \[V_{k,a,b}^i|_{\wt{Z_*K}}=\bigoplus_{j=0}^{d-1}\mu^i_{k,a+2j,b+2j},\quad \dim V_{k,a,b}^{i,K_m^*}=\begin{cases} 1, &\text{if }i=1,\ 2\leq k\leq m;\\ 1, &\text{if } (k,i)=(0,2),(1,1), 3\mid a-b;\\
0,  & \text{otherwise}.\end{cases}\]
 Moreover, 
\begin{itemize}
    \item when $(k,i)=(0,2),(1,1)$, $V_{k,a,b}^i\cong V_{k,a^\prime,b^\prime}^i$ if and only if  $a-b\equiv a^\prime-b^\prime \mod 3$.
    
    \item when $(k,i)\neq(0,2),(1,1)$, $V^i_{k,a,b}$ is independent of the choice of $a,b$.
    \item for $\{i,i^\prime\}=\{1,2\}$, $V^i_{k,a,b}\cong V^{i^\prime}_{k^\prime,a^\prime,b^\prime}$ if and only $\{(k,i), (k^\prime,i^\prime)\}=\{(0,2),(1,1)\}$, $a-b^\prime\equiv b-a^\prime \mod 3$.
    \item  $V_{k+1,a,b}^{1,\tau}\cong V_{k,a,b+1}^{2}$ for $\tau=s\left[\begin{pmatrix} 0 & 1\\ \varpi & 0\end{pmatrix}\right]$. 
\end{itemize}

\end{lem}

\begin{proof}Note that \[I_{\wt{Z_*K_m}}^{\wt{ZK_1}}\chi_{a,b}|_{K_1^*}=I_{\wt{Z_*K_1}}^{\wt{ZK_1}}I_{K_m^*}^{K_1^*}\chi_{a,b}|_{K_1^*}=\bigoplus_{j=0}^{d-1}I_{K_m^*}^{K_1^*}\chi_{a+2j,b+2j},\]
\[I_{\wt{Z_*K_m}}^{\wt{ZK_1}}\chi_{a,b}^w|_{K_1^*}=I_{\wt{Z_*K_1}}^{\wt{ZK_1}}I_{K_{1,m}^*}^{K_1^*}\chi^w_{a,b}|_{K_1^*}=\bigoplus_{j=0}^{d-1}I_{K_{1,m}^*}^{K_1^*}\chi^w_{a+2j,b+2j},\]
\[(I_{\wt{Z_*K_m}}^{\wt{ZK_1}}\chi_{a,b})^\tau=I_{\wt{Z_*K}_{1,m-1}}^{\wt{ZK_1}}(\chi_{a,b}^\tau)=I_{\wt{Z_*K}_{1,m-1}}^{\wt{ZK_1}}\chi_{a,b+1}^w.\]
Now all stated results follow from Lemma \ref{irreII}. 
\end{proof}
For $(k,i)=(0,2),(1,1)$, we shall   denote $V^i_{k,a,b}$ by $V^i_{k,a-b}$ and in other cases,   denote $V^i_{k,a,b}$ by $V_k^i$. Set $\tilde{V}_{1,i}=V_{0,i}^2\oplus V_{1,i+1}^1$ and $\tilde{V}_{k+1}=V_k^2\oplus V_{k+1}^1$ for $k\geq1$. Then by Lemma \ref{irreIV}, 
\begin{itemize}
\item $\tilde{V}_{1,1}=\tilde{V}_{1,1}^+\oplus \tilde{V}_{1,1}^-$ with  $\tilde{V}_{1,1}^+$ and $ \tilde{V}_{1,1}^-$ non-isomorphic irreducible $\wt{J}$-representations.
    \item Except $\tilde{V}_{1,i}\cong \tilde{V}_{1,2-i}$, all  $\tilde{V}_{k}$, $k\geq 2$ and $\tilde{V}_{1,i}$, $i\neq 1$, $\tilde{V}_{1,1}^\pm$ are irreducible non-isomorphic $\wt{J}$-representations. 
\end{itemize}
\begin{lem}\label{K-type II}$V(\chi)|_{\wt{J}}=\left(\bigoplus_{i=0}^{2} \tilde{V}_{1,i}\right)\bigoplus \left(\bigoplus_{k\geq2}\tilde{V}_k^{\oplus 3}\right)$.
\end{lem}
\begin{proof}By Lemma \ref{K-type I}, we deduce that  \[V(\chi)|_{\wt{ZK_1}}=\left(\bigoplus_{i=0}^{2}V_{1,i}^1\right)\bigoplus
\left(\bigoplus_{k\geq2}V_k^1\right)\bigoplus \left(\bigoplus_{i=0}^{2}V_{0,i}^2\right)\bigoplus
\left(\bigoplus_{k\geq1}V_k^2\right).\] 
By the above discussion and the Frobenius reciprocity law, one deduce
\[V(\chi)|_{\wt{J}}=\left(\bigoplus_{i=0}^{2} \tilde{V}_{1,i}\right)\bigoplus \left(\bigoplus_{k\geq2}\tilde{V}_k^{\oplus 3}\right).\] 
\end{proof}

\begin{thm}\label{K-type} 
We have
\[V_0(\chi)|_{\wt{ZK}}=V_{0,0}\bigoplus V_{1,1}\bigoplus \left(\bigoplus_{k\geq 2}V_k\right)\]
and for some $*\in\{\pm\}$, \[V_0(\chi)|_{\wt{J}}=\tilde{V}_{1,1}^*\bigoplus \tilde{V}_{1,0}\bigoplus \left(\bigoplus_{k\geq2}\tilde{V}_k\right).\]
\end{thm}
\begin{proof} By \cite[Theorem I.2.9(e)]{KP84}, $\dim V_0(\chi)^{K^*}=1$. We deduce $V_{0,0}\subset V_0(\chi)|_{\wt{ZK}}$. Note that $V_{0,0}=V_{0,0}^2$ and for $i\neq0$, $V_{1,i}|_{\wt{ZK_1}}=V_{0,i}^2\oplus V_{1,i}^i\oplus V_{1}^2$ by Lemma \ref{irreII} and Lemma \ref{irreIV}. As $V_0(\chi)$ has the structure of $\wt{J}$-representations, we have that $V_{1,1}^1\subset V_0(\chi)|_{\wt{ZK_1}}$ and hence $V_{1,1}\subset V_0(\chi)|_{\wt{ZK}}$ by Lemma \ref{irreIII} and \ref{irreIV}. Inductively by Lemma \ref{irreIV}, we deduce that for some $*\in\{\pm\}$, 
\[V_0(\chi)|_{\wt{J}}\supset \tilde{V}_{1,1}^*\bigoplus \tilde{V}_{1,0}\bigoplus \left(\bigoplus_{k\geq2}\tilde{V}_k\right).\]
The inclusion is actually an equality. Otherwise, we can apply the same argument to deduce 
\[V_0(\chi)|_{\wt{ZK}}=V_{0,0}\bigoplus V_{1,0}\bigoplus \left(\bigoplus_{k\geq2}V_k^{\oplus 3}\right)\bigoplus V_{1,1}^{\oplus2}\] 
from the $\wt{ZK}$-representation structure of $V_0(\chi)$ and  Lemma \ref{irreIII} and \ref{irreIV}. This in turn implies  
\[V(\chi)|_{\wt{J}}=\left(\bigoplus_{i=0}^{2} \tilde{V}_{1,i}\right)\bigoplus \left(\bigoplus_{k\geq2}\tilde{V}_k^{\oplus 3}\right).\] 
Since $\wt{J}$ and $\wt{ZK}$ generate $\wt{G}$, we have $V_0(\chi)=V(\chi)$ by Lemma \ref{K-type I} and \ref{K-type II}. A contradiction. 
Now the stated result follows.
\end{proof}
The following corollary  recovers \cite[Theorem 4.1]{Pat98} when $c=0$ and $n=3$.
\begin{cor} $\dim V_0(\chi)^{K^*}=1$ and
$\dim V_0(\chi)^{K_m^*}=m$ for $m\geq1$.
\end{cor}

\subsection{Multiplicity formula of   co-periods}\label{Mul-one}
 Let $\chi_\pm$ be $\epsilon^\pm$-genuine characters of $\wt{Z}^{(c)}\wt{T_3}$.  Denote by $\omega_\pm$  the central character of $\pi_\pm:=V(\chi^w_\pm)$. Assume $\chi_{\pm}$ is exceptional and let $\sigma_\pm=V_0(\chi_\pm)$. Let $\sigma^{\pm}$ be the quotient of $\pi_{\pm}$ by $\sigma_{\pm}$. 
 Let $\sigma$ be an irreducible smooth admissible representation of $G$  with central character $\omega$. Recall that the center of $\wt{G}$ is $\wt{Z^{(c)}}$. Clearly unless the $\wt{Z^{(c)}}$-character $\omega\omega_+\omega_-$ is trivial, 
\[\Ext^i_G(\sigma \otimes \sigma_+ \otimes \sigma_-, \BC)=0,\quad \forall\ i\geq0. \]
In this subsection, assume
 \[\omega\omega_+\omega_-=1,
 \qquad
 \chi_+\chi_-\bigl(s(\diag\{x^3,y^3\})\bigr)=|x||y|^{-1}.
 \]
 In particular $\omega^3=1$.  For a generic $G$-representation $\sigma$ we write
 \[
 \varepsilon(\sigma,\Sym^3)=\varepsilon(1/2,\sigma,\Sym^3)\in\{\pm1\}.
 \]
  We shall consider
 \begin{center}
     $\Ext^i_G(\sigma \otimes\sigma_+ \otimes \sigma_-, \BC)$.
 \end{center}

We shall freely use the results on higher extension groups in \cite{Pra18, FP22}.
\begin{prop}\label{ub}Assume  $\sigma$ is essentially discrete. Then $$\dim\Hom_G(\sigma\otimes \sigma_+\otimes \pi_-,\BC)=3|3|^{-1/2}.$$

\end{prop}
\begin{proof}
Note that by Frobenius reciprocity law, 
$$\Hom_G(\sigma\otimes \sigma_+\otimes \pi_-,\BC)=\Hom_{\wt{B}}(\sigma\otimes \sigma_+, \delta^{1/2}I_{\wt{T}_*}^{\wt{T}}\chi_-^{w,-1}).$$
Since  $\sigma$ is essentially discrete, one has \[\Hom_{\wt{T}}(\sigma_N\otimes \sigma_{+,N} , \delta^{1/2}I_{\wt{T}_*}^{\wt{T}}\chi_{-}^{w,-1})=0.\]
Thus by  Proposition \ref{Jac of tensor},
$$\Hom_G(\sigma\otimes \sigma_+\otimes \pi_-,\BC)=\Hom_{\wt{T}}((\sigma\otimes \sigma_+)_N, \delta^{1/2}I_{\wt{T}_*}^{\wt{T}}\chi_{-}^{w,-1})=\Hom_{\wt{T}}(i_{\wt{Z}}^{\wt{T}}\sigma_\psi\otimes\sigma_{+,\bar{\psi}}, \delta^{1/2}I_{\wt{T}_*}^{\wt{T}}\chi_{-}^{w,-1}).$$
By Frobenius reciprocity and Proposition \ref{Jac}, \[\dim\Hom_{\wt{T}}(i_{\wt{Z}}^{\wt{T}}\sigma_\psi\otimes\sigma_{+,\bar{\psi}}, \delta^{1/2}I_{\wt{T}_*}^{\wt{T}}\chi_{-}^{w,-1})=\dim\Hom_{\wt{Z}}(\sigma_\psi\otimes\sigma_{+,\bar{\psi}}, I_{\wt{T}_*}^{\wt{T}}\chi_{-}^{w,-1})=3|3|^{-1/2}.\]
\end{proof}

\subsubsection{Non-supercuspidal case}\label{Non-supercuspidal}
We first deal with the case $\sigma$ is non-supercuspidal. \begin{lem}\label{one-dim}Let $\xi$ be a character of $F^\times$. Then \[\Ext^1_{Z_*\bs G}(\St\otimes\xi\circ\det\otimes \sigma_{+}\otimes \sigma_{-},\BC)=0\]
 and $\dim\Hom_G(\xi\circ\det\otimes \sigma_{+}\otimes \sigma_{-},\BC)\leq1$ with
 the equality holds iff on $\wt{Z^{(c)}}\wt{T_3}$, $\chi_+\chi_{-}^w\cdot\xi\circ\det=1$. 

\end{lem}
\begin{proof}Easy to see that$$\dim\Hom_{\wt{G}}(\xi\circ\det\otimes \sigma_{+}\otimes \sigma_{-},\BC)\leq1$$
with the equality holds if and only if  $\chi_+\chi_{-}^w\cdot \xi\circ\det=1$ on $\wt{Z^{(c)}}\wt{T_3}$. 

By Stuhler-Schneider duality, one has 
\[\dim\Ext^1_{Z_*\bs G}(\St\otimes\xi\circ\det\otimes \sigma_{+}\otimes \sigma_{-},\BC)=\dim\Hom_G(\BC, \xi\circ\det\otimes \sigma_{+}\otimes \sigma_{-})=\dim \left(\xi\circ\det\otimes \sigma_{+}\otimes \sigma_{-}\right)^G.\]
Standard map identifies $\left(\xi\circ\det\otimes \sigma_{+}\otimes \sigma_{-}\right)^G$ with $\Hom_{\tilde{G}}^{\fin}(\sigma_-^\vee, \xi\circ\det\otimes \sigma_{+})$ where the superscript fin stands for operators of finite rank, which must be zero.
\end{proof}
 \begin{prop}\label{PS}Assume  $\sigma=I_B^G\theta$ with $\theta(\diag\{x^{-3},x^3\})\neq|x|^3$
 ( $\sigma$ not necessarily irreducible). Then  for $\sigma_1=\sigma_+$, $\sigma^+$ and $\pi_+$, $\sigma_2=\sigma_-$, $\sigma^-$ and $\pi_-$
 $$\Ext^i_G(\sigma\otimes \sigma_1\otimes \sigma_2,\BC)=0,\quad \forall i\geq2;$$ 
$$\dim\Hom_G(\sigma\otimes\sigma_1\otimes \sigma_2,\BC)=\dim\Ext^1_G(\sigma\otimes\sigma_1\otimes \sigma_2,\BC)=\dim \Hom_{\wt{Z}}(\sigma_{1,\psi}, \sigma_{2,\bar{\psi}}^\vee\otimes\theta^{-1}).$$
\end{prop}
\begin{proof} Note that  \begin{align*}
\Hom_G(\sigma\otimes\sigma_1\otimes \sigma_2,\BC) =\Hom_{G}( \sigma_1\otimes \sigma_2, \sigma^\vee)=\Hom_{B}(\sigma_1\otimes \sigma_2, \delta^{1/2}\theta^{-1})=\Hom_{T}( \big(\sigma_1\otimes \sigma_2\big)_N,\delta^{1/2}\theta^{-1})
\end{align*}
By Proposition \ref{Jac of tensor}, we have the long exact sequence
\begin{align*}
    0&\to \Hom_T(\sigma_{1,N}\otimes \sigma_{2,N},\delta^{1/2}\theta^{-1})\to \Hom_{T}( \big(\sigma_1\otimes\sigma_2\big)_N,\delta^{1/2}\theta^{-1})\to \Hom_T(i_Z^T\sigma_{1,\psi}\otimes \sigma_{2,\bar{\psi}}, \delta^{1/2}\theta^{-1})\\
  & \to \Ext^1_T(\sigma_{1,N}\otimes \sigma_{2,N}, \delta^{1/2}\theta^{-1})\to \Ext^1_{T}( \big(\sigma_1\otimes\sigma_2\big)_N,\delta^{1/2}\theta^{-1}) \to \Ext^1_T(i_Z^T\sigma_{1,\psi}\otimes \sigma_{2,\bar{\psi}}, \delta^{1/2}\theta^{-1})\\
 & \to \Ext^2_T(\sigma_{1,N}\otimes \sigma_{2,N}, \delta^{1/2}\theta^{-1})
     \to \Ext^2_{T}(\big(\sigma_1\otimes\sigma_2\big)_N,\delta^{1/2}\theta^{-1})
     \to \Ext^2_T(i_Z^T\sigma_{1,\psi}\otimes \sigma_{2,\bar{\psi}}, \delta^{1/2}\theta^{-1})\to0.
\end{align*}
  By \cite[Prop 2.8]{Pra18},$$\Ext^i_T(i_Z^T\left(\sigma_{2,\psi}\otimes \sigma_{2,\bar{\psi}}\right), \delta^{1/2}\theta^{-1})\cong\Ext^i_Z(\sigma_{1,\psi}\otimes \sigma_{2,\bar{\psi}},\theta^{-1}),$$ which is zero when $i=2$ by \cite[Prop 2.9]{Pra18}. Together with the Schneider-Stuhler duality, we have
\begin{align*}\dim\Hom_T(i_Z^T\left(\sigma_{1,\psi}\otimes \sigma_{2,\bar{\psi}}\right), \delta^{1/2}\theta^{-1})=\dim\Ext^1_T(i_Z^T\left(\sigma_{1,\psi}\otimes \sigma_{2,\bar{\psi}}\right), \delta^{1/2}\theta^{-1})=\dim\Hom_{\wt{Z}}(\sigma_{1,\psi}, \sigma_{2,\bar{\psi}}^\vee\otimes\theta^{-1}).
   \end{align*}
 Again by Schneider-Stuhler duality theorem, one has 
$$\dim\Hom_T(\sigma_{1,N}\otimes\sigma_{2,N}, \theta^{-1}\delta^{1/2})=\dim\Ext^2_T(\sigma_{1,N}\otimes\sigma_{2,N}, \theta^{-1}\delta^{1/2}).$$
 By \cite[Prop 2.9]{Pra18}, we have
 \begin{align*} 
    \dim\Ext^1_T(\sigma_{1,N}\otimes\sigma_{2,N}, \theta^{-1}\delta^{1/2})=\dim\Hom_T(\sigma_{1,N}\otimes\sigma_{2,N}, \theta^{-1}\delta^{1/2})+\dim\Ext^2_T(\sigma_{1,N}\otimes\sigma_{2,N}, \theta^{-1}\delta^{1/2}).
\end{align*}
By Proposition \ref{Jac}, one has 
$$\dim\Hom_T(\sigma_{1,N}\otimes\sigma_{2,N}, \theta^{-1}\delta^{1/2})
=\dim\Hom_{\wt{T}}(\sigma_{1,N},\sigma_{2,N}^\vee\otimes\theta^{-1}\delta^{1/2})$$
which is non-zero only if on $\wt{Z^{(c)}}\wt{T_3}$, some of the following equalities hold \[\chi_+^w\chi_{-}^w\theta\delta^{1/2}=1,\quad \chi_+\chi_{-}\theta\delta^{1/2}=1,\quad \chi_+^w\chi_{-}\theta\delta^{1/2}=1,\quad \chi_+\chi_{-}^w\theta\delta^{1/2}=1\]
Since $\chi_{\pm}$ is exceptional, this holds only if 
\[\theta(\diag\{x^{-3},x^{3}\})=|x|^{3+2},\ |x|^3,\ |x|^{3-2}\]
In the first and third case, $\sigma=I_B^G\theta$ is irreducible and isomorphic to $I_B^G\theta^w$. By considering the model $I_B^G\theta^w$, we find 
$\dim\Ext^i_T(\sigma_{1,N}\otimes\sigma_{2,N},\theta^{-1}\delta^{1/2})=0$ and the stated result follows from the long exact sequence.
\end{proof}
By Proposition \ref{Jac}, we have $\dim\Hom_{\wt{Z}}(\sigma_{+,\psi},\sigma_{-,\bar{\psi}}^\vee\otimes\theta^{-1})=1$. 
\begin{prop}\label{St}For $\sigma=\St\otimes\xi$,  $\Ext^i_G(\sigma\otimes \sigma_{+}\otimes \sigma_{-},\BC)=0$ for $i\geq2$ and 
$$\dim\Hom_G(\sigma\otimes \sigma_{+}\otimes \sigma_{-},\BC)=\dim\Ext^1_G(\sigma\otimes \sigma_{+}\otimes \sigma_{-},\BC)\leq1$$
where the equality holds if and only if on $\wt{Z^{(c)}}\wt{T_3}$, $\chi_+\chi_{-}^w\cdot\xi\circ\det\neq1$. 
\end{prop}
\begin{proof}By cohomology dimension bound, $\Ext^i_G(\sigma\otimes \sigma_{+}\otimes \sigma_{-},\BC)=0$ for $i\geq3$. The argument in Lemma \ref{one-dim} shows that $\Ext^2_G(\sigma\otimes \sigma_{+}\otimes \sigma_{-},\BC)=0$.

Tensor the short exact sequence 
\[0\to \sigma\to I_B^G\theta \to \xi\circ\det\to0\]
with $\sigma_+\otimes\sigma_-$. First consider  the induced long exact of $Z_*\bs G$-representations. Combine Lemma \ref{one-dim} and Proposition \ref{PS}, together with the numerology above, we deduce $$\dim\Hom_G(\sigma\otimes \sigma_{+}\otimes \sigma_{-},\BC)\leq1$$ and the equality holds if and only if on $\wt{Z^{(c)}}\wt{T_3}$, $\chi_+\chi_{-}^w\cdot\xi\circ\det\neq1$.

By Proposition \ref{PS}, we deduce that 
$\Ext^2_G(\xi\circ\det\otimes \sigma_{+}\otimes \sigma_{-},\BC)=0$. Now consider the induced  long exact sequence of $G$-representations, we have 
\begin{align*}
    0&\to \Hom_G(\xi\circ\det\otimes\sigma_+\otimes\sigma_-,\BC)\to\Hom_G(I_B^G\theta\otimes\sigma_+\otimes\sigma_-,\BC)\to \Hom_G(\sigma\otimes\sigma_+\otimes\sigma_-,\BC)\\
  & \to \Ext^1_G(\xi\circ\det\otimes\sigma_+\otimes\sigma_-,\BC)\to\Ext^1_G(I_B^G\theta\otimes\sigma_+\otimes\sigma_-,\BC)\to \Ext^1_G(\sigma\otimes\sigma_+\otimes\sigma_-,\BC)\to0.
\end{align*}
By the vanishing of the Euler-Poincare formula for finite length representations, we have \[\dim\Hom_G(\xi\circ\det\otimes\sigma_+\otimes\sigma_-,\BC)=\dim\Ext^1_G(\xi\circ\det\otimes\sigma_+\otimes\sigma_-,\BC).\]
Together with Proposition \ref{PS}, we deduce  \[\dim\Hom_G(\sigma\otimes\sigma_+\otimes\sigma_-,\BC)=\dim\Ext^1_G(\sigma\otimes\sigma_+\otimes\sigma_-,\BC).\]
\end{proof}
 
\begin{cor}\label{PS2}For $\sigma=I_B^G\theta$ (not necessarily irreducible), one has
\[\dim\Hom_G(\sigma\otimes\sigma_+\otimes \sigma_-,\BC)=\dim\Ext^1_G(\sigma\otimes\sigma_+\otimes \sigma_-,\BC)=1\]
\end{cor}
\begin{proof} For $\theta=(\theta_1,\theta_2):\ T\to\BC^\times$ such that $\theta(\diag\{x^{-3},x^3\})=|x|^3$, one has $\theta_1\theta_2^{-1}=|\cdot| \mu$ with $\mu^3=1$. When $\mu\neq 1$, $\sigma=I_B^G\theta$ is irreducible. The argument in Proposition \ref{PS} works and its result holds in this setting. Together with the case$\theta(\diag\{x^{-3},x^3\})\neq|x|^3$, we can deduce the desired result  from Proposition \ref{PS} and the numerology above.

 To deal with the remaining $\mu=1$ case,  write $\theta_1=\xi|\cdot|^{1/2}$ and $\theta_2=\xi|\cdot|^{-1/2}$. Then one has a short exact sequence
\[0\to \St\otimes\xi\to\sigma\to\xi\circ\det\to0\]
Tensor with $\sigma_+\otimes\sigma_-$ and consider the induced long exact sequence, we deduce the desired result from Proposition \ref{St}.
\end{proof}

When $\chi_+\chi_{-}(s(\diag\{x^3,y^3\})=|x||y|^{-1}$, $\chi_+\chi_{-}^w\cdot\xi\circ\det\neq1$ on 
$\wt{Z^{(c)}}\wt{T_3}$ is equivalent to $\xi^3\neq1$. Now combining  Proposition \ref{St}, Proposition \ref{PS2} and Proposition \ref{epsilon nsc}, we deduce
 \begin{prop}\label{n-SC}Assume $\chi_+\chi_{-}(s(\diag\{x^3,y^3\})=|x||y|^{-1}$. When $\sigma$ is non-supercuspidal and generic, 
 \[\dim \Hom_G(\sigma\otimes\sigma_+\otimes\sigma_-,\BC)=\frac{1+\varepsilon(\sigma,\Sym^3)}{2}.\]
 \end{prop}

\subsubsection{Supercuspidal case}\label{supercuspidal case}
Now we  deal with the case $\sigma$ is supercuspidal under the assumption $|3|=1$. We will use notations in  Subection \ref{K-type Section} and  add the superscript $\pm$ to all components of the restriction of $\sigma_\pm$ to $\CK=ZK$ or $J$. We shall freely use the theory of Kutzko (see \cite[Section 6]{Pra90} for details) on $K$-types. 
We start with the following result on finite filed. Let $k$ be the residue field of $F$ and write
$\overline G=\mathrm{GL}_2(k)$, with standard subgroups $\overline B=\overline T\overline N$ and $\overline Z$. For any character $\alpha$ of $\overline T$, let $I(\alpha)=I_{\overline B}^{\overline G}(\alpha)$.
\begin{lem}\label{finite}When central characters match, 
\[\dim\Hom_{\overline G}
 (I(\alpha)\otimes I(\beta),\rho)=1\]
 for all supercuspidal $\overline G$-representation $\rho$. 
\end{lem}
\begin{proof}The $\overline G$-action on $\overline G/\overline B\times\overline G/\overline B$ has two orbits. Thus by the geometry lemma, one has 
\[I(\alpha)\otimes I(\beta)
 \simeq
 I(\alpha\beta)
 \oplus
 \Ind_{\overline T}^{\overline G}(\alpha\beta^w).
\]
By the construction of supercuspidal repsentation, one has that \[
 \dim\Hom_{\overline T}(\alpha,\rho)
 =\begin{cases}
 1,&\alpha|_{\overline Z}=\omega_\rho,\\
 0,&\text{otherwise}.
 \end{cases}
\]
Combining all these, the stated result follows.
\end{proof}
\begin{thm}\label{sc}Assume $|3|=1$ and $\omega\omega_+\omega_{-}=1$. When $\sigma$ is supercuspidal, 
\[\dim\Hom_G(\sigma\otimes\sigma_+\otimes\sigma_-,\BC)=1.\]
\end{thm}
\begin{proof}By simultaneous determinant twists, we can and will assume that the exceptional
characters $\chi_\pm$ are unramified and normalized. By \cite[Theorem 6,1]{Pra90}, $\sigma=I_{\CK}^G\tau^\vee$ for some very cuspidal representation $\tau^\vee$ of $\CK=J$ or $ZK$.  Then by Proposition \ref{ub},
\[\dim\Hom_{\CK}(\pi_+\otimes\sigma_-,\tau)=\dim\Hom_{G}(\sigma\otimes\pi_+\otimes\sigma_-,\BC)=3\]
 \[m(\sigma)=\dim\Hom_{G}(\sigma\otimes\sigma_+\otimes\sigma_-,\BC)=\dim\Hom_{\CK}(\sigma_+\otimes\sigma_-,\tau)\]
 By Lemma  \ref{K-type I} and Lemma \ref{K-type II}, one  has  the decomposition
\[ \sigma_+|_{\widetilde{\CK}}=L_+\oplus H_+,
 \qquad
 \pi_+|_{\widetilde{\CK}}=P_+\oplus H_+^{\oplus3},\]
where $H_+$ is the sum of the constituents with index $k\geq2$, and
$L_+\subset P_+$ are the corresponding low-level summands. Let \[
 K(1)=1+\varpi M_2(\CO),\qquad
 B(1)=B(\CO)\cap K(1).
\]
We claim that if $\sigma$ is not of depth-zero, then $\dim\Hom_{\widetilde{\CK}}(P_+\otimes\sigma_-,\tau)=0$, which clearly forces $m(\sigma)=1$. 
\begin{itemize}
\item when  $\CK=ZK$ and $\tau$ is very cuspidal of level $n\geq2$, $\tau^{N(\varpi^{n-1}\CO)}=0$
and   $\tau^{B(1)}=0$.  Every $K$-constituent of $P_+$ is trivial on $K(1)$and  $\pi_-|_K$ is a finite direct sum of representations
$\Ind_{B(\CO)}^K\xi$, with $\xi$ tame. Thus \[
 \Hom_{K}
 \left(
   P_+\otimes\operatorname{Ind}_{B(\mathcal O)}^K\xi,\tau
 \right)
 \simeq
 \operatorname{Hom}_{B(\mathcal O)}(P_+\otimes\xi,\tau)=0,
\]
because $B(1)$ acts trivially on the source and has no invariant vectors
in the target. 
\item when $\CK=J$,  $P_+|_{K_1}$ is a sum of tame characters and 
s $\pi_-|_{K_1}$ is 
finite sum of representations induced from
\[
 B(\mathcal O)=T(\mathcal O)N(\mathcal O)
 \quad\text{or}\quad
 T(\mathcal O)N^{-}(\varpi\mathcal O),
\]
with inducing characters trivial on the indicated unipotent subgroup.
The very cuspidal property gives  $\tau^{N(\CO)}=
 \tau^{N^{-}(\varpi\CO)}=0$. Thus 
$\Hom_{K_1}
 (P_+\otimes\pi_-,\tau)=0.$
\end{itemize}
Since (the restriction of) $\sigma_-$ is a direct summand of $\pi_-$, the claim follows.

The remaining case is $\CK=ZK$ and $\tau$ is inflated from a cuspidal representation of $\overline G$.
Note that  \[L_{\pm}=V_{0,0}^{\pm}\oplus V_{1,1}^{\pm}\subset P_{\pm}=V_{0,0}^\pm\oplus V_{1,0}^{\pm}\oplus V_{1,1}^{\pm,\oplus 2}\]
where the  $K$-constituents are respectively characters,
special representations, and irreducible finite principal series. By Lemma \ref{finite}, one has 
\[
 \dim\Hom_{\widetilde{ZK}}
 (V_{1,1}^{+}\otimes V_{1,1}^-,\tau)=
 \dim\operatorname{Hom}_{\widetilde{ZK}}
 (V_{1,0}^+\otimes V_{1,1}^-,\tau)=1\]
 \[\dim\Hom_{\widetilde{ZK}}
 (V_{0,0}^{+}\otimes V_{1,1}^-,\tau)=
 \dim\Hom_{\widetilde{ZK}}
 (V_{0,0}^+\otimes V_{0,0}^-,\tau)=0
\]
Thus one deduce that \[
 \dim\Hom_{\widetilde{ZK}}(P_+\otimes L_-,\tau)
 =1+2\cdot1=3, \quad \dim\Hom_{\widetilde{ZK}}(L_+\otimes L_-,\tau)
 =1.
\]
By Proposition \ref{ub}, one has $\Hom_{\widetilde{ZK}}(H_+\otimes L_-,\tau)=0$ and consequently $m(\sigma)=1$.
\end{proof}
\subsection{Local  co-period: non-archimedean case}\label{Tri}
 Pertain the notations in Subection \ref{Mul-one}. In this subsection, we   explicitly construct distinguished linear functionals in $\Hom_G(\sigma\otimes\sigma_+\otimes\sigma_-,\BC)$ under the central character assumption $\omega\omega_+\omega_-=1$. Without lose of generality, we shall assume the central characters $\omega_*$, $*=\pm, \emptyset$ are all unitary. Let $\lambda(\sigma)=0$ when $\sigma$ is discrete and $\lambda(\sigma)$ be the real number such that $\theta\delta^{-\lambda(\sigma)}$ is unitary.  Note that $\sigma$ is tempered if and only $\lambda(\sigma)=0$ and if $\sigma$ is a component of a cuspidal automorphic $\GL_2(\BA)$ representation, $|\lambda(\sigma)|<1/6$ by \cite{KS99}. \textit{We shall assume $-1/6<\lambda(\sigma)<1/6$ in the following}.
 
For $*=\emptyset,\pm$, take a non-degenerate $\wt{G}$-invariant pairing $(-,-)_*$ on $\sigma_*$ and consider the  integration of matrix coefficient 
 \[I(W\otimes W_+\otimes W_-, W^\vee\otimes W_+^\vee\otimes W_-^\vee):=\int_{Z^{(c)}\bs G}( \sigma(g)W, W^\vee) ( \sigma_+(g)W_+, W_+^\vee)_+ (\sigma_-(g)W_-, W_-^\vee)_- dg\]
 for $W_*\in\sigma_*$ and $W_*^\vee\in\sigma_*^\vee$. The main result of this subsection is:
 \begin{thm}\label{Canonical local period}
   The integration converges absolutely. When $\Hom_{G}(\sigma\otimes\sigma_+\otimes\sigma_-,\BC)$ has multiplicity one, the element \[I\in \Hom_{G}(\sigma\otimes\sigma_+\otimes\sigma_-,\BC)\times \Hom_{G}(\sigma^\vee\otimes\sigma_+^\vee\otimes\sigma_-^\vee,\BC)\] is  non-zero.
\end{thm}

Embed $\sigma$ into $I_{ZN}^G\omega\boxtimes\bar{\psi}$
and consider the pairing 
\[(-,-):\ \sigma\times\sigma^\vee\lra\BC,\quad (W,W^\vee)\mapsto\int_{F^\times}WW^\vee\left[\begin{pmatrix}y & 0\\ 0 & 1\end{pmatrix}\right]d^\times y.\] 
Extend $\omega_{\pm}$ to $\wt{Z_*}$. Embed $\sigma_+$ (resp. $\sigma_+^\vee$) into $I_{\wt{Z_*}N}^{\wt{G}}\omega_{+}\boxtimes\psi$ (resp. $I_{\wt{Z_*}N}^{\wt{G}}\omega_{+}^{-1}\boxtimes\bar{\psi}$).
Consider the non-degenerate pairing \[(-,-)_+:\ \sigma_+\times\sigma_+^\vee\lra\BC, \quad (W_+,W_+^\vee)\mapsto\int_{F^\times}\sum_{z\in\wt{Z_*}\bs\wt{Z}}WW^\vee\left[z\begin{pmatrix}y &0 \\ 0 & 1\end{pmatrix}\right]d^\times y.\] on $\sigma_{+}\times\sigma_{+}^\vee$. 
Embed $\sigma_-$ into $I_{\wt{B}_*}^{\wt{G}}\chi_-^w$ and consider the pairing $(-,-)_-$ induced from the pairing 
\[I_{\wt{B_*}}^{\wt{G}}\chi_-^w\times I_{\wt{B_*}}^{\wt{G}}\chi_-^{w,-1}\lra\BC,\quad (f,f^\vee)\mapsto\int_{K}\sum_{t\in \wt{T_*}\bs \wt{T}}\delta^{-1}(t)f(tk)f^\vee(tk)dk.\]
As explained in Subsubsection \ref{InvP}, all these pairings are $\wt{G}$-invariant.
Let $T^\prime=\{(a,b)\in T| ab\in\Gamma\}$. Then $\wt{T^\prime}\subset\wt{T}$ is the subgroup fixing $\wt{Z_*}$ for the conjugation action.  
\begin{prop}\label{Tri-I}
For $*=\emptyset$ or $\vee$, the trilinear forms 
\[\Psi_t^*:\ \sigma^*\otimes\sigma_+^*\times (I_{\wt{B_*}}^{\wt{G}}\chi_-^{w})^*\to\BC,\quad (W^*,W_+^*,f_-^*)\mapsto\delta^{-1/2}(t)\int_{ZN\bs G}\sum_{z\in \wt{Z_*}\bs\wt{Z}}W^*(zg)W_+^*(zg)f_-^*(tzg)dg\]
forms a basis of $\Hom_{G}(\sigma^*\otimes\sigma_+^*\otimes (I_{\wt{B_*}}^{\wt{G}}\chi^{w}_-)^*,\BC)$ when $t$ running over a complete set  of representatives $\wt{T_*}\bs \wt{T'}$. Moreover, if one defines $I$ using $(-,-)_*$ for $*=\pm,\emptyset$ as above, then 
\[I(W\otimes W_+\otimes f_-, W^\vee\otimes W_+^\vee\otimes f_-^\vee)=\zeta_F(1)\sum_{t\in \wt{T_*}\bs \wt{T^\prime}}\Psi_t(W,W_+,f_-)\Psi_t^\vee(W^\vee,W_+^\vee,f_-^\vee)\]
for any $W^* \in \sigma^*$, $W_+^* \in \sigma_+^*$
and $f_-^* \in (I_{\wt{B_*}}^{\wt{G}}\chi^{w}_-)^*$ with $* = 
\emptyset$ and $\vee$.
\end{prop}
\begin{proof}By the gauge estimation in \cite[Theorem I.4.1]{KP84}, the integration converges absolutely and the trilinear form $\Psi_t$(resp. $\Psi_t^\vee$) $t\in \wt{T_*}\bs \wt{T}$ is well-defined. By construction, all $\Psi_t$ belong to $\Hom_G(\sigma\otimes\sigma_+\otimes I_{\wt{B_*}}^{\wt{G}}\chi_-^w,\BC)$. Applying this invariance property to $\wt{Z_*}$, we find $\Psi_t=0$ unless $t\in \wt{T^\prime}$. Take $U\subset K$ small enough such that the covering map $\wt{G}\to G$ admits a splitting $\kappa$ on $U$. Choose $W_{t^\prime}\in \sigma$ fixed by $U$ and when restricted to $\diag\{F^\times,1\}$, $W_{t^\prime}$ is the characteristic function on $\diag\{V,1\}$ where $V\subset \CO^\times$ is a small open subgroup. Since $i_{\wt{Z_*}N}^{\wt{B}}\omega_+\boxtimes\psi\subset \sigma_{+,\psi}$, we can choose $W_{+,t^\prime}\in\sigma_+$ fixed by $\kappa(U)$ and when restricted to $\wt{T}$ is supported on $s(\diag\{V,1\})\wt{Z_*}$. Choose $f_{-,t^\prime}\in I_{\wt{B_*}}^{\wt{G}}\chi_-^w$ such that $f_{-,t^\prime}$ is supported on $\wt{B_*}t^\prime U$, fixed by $U$ and $f_{-,t^\prime}(t^\prime)=1$. All these functions exists when $U$ is small. 
Since $\wt{G}=\sqcup_{t\in \wt{T_*}\bs \wt{T}} \wt{B_*} t \wt{K}$, we find $\Psi_t(W_{t^\prime},W_{+,t^\prime},f_{-,t^\prime})\neq0$ if and only $t=t^\prime$. By Proposition \ref{ub}, we have \[[\wt{T^\prime}:\wt{T_*}]=\sqrt{[F^\times:F^{\times,3}]}=\dim \Hom_G(\sigma\otimes\sigma_+\otimes I_{\wt{B_*}}^{\wt{G}}\chi_-^w,\BC).\]
Thus $\{\Psi_{t}, t\in \wt{T_*}\bs \wt{T^\prime}\}$ forms a 
basis of $\wt{G}$-invariant trilinear forms. Similar results holds for $\Psi_t^*$.

 To obtain the decomposition formula, one can  adapt the argument in  \cite[Proposition 5.2]{Hsi21}.
\end{proof}
Now we can prove Theorem \ref{Canonical local period}.
\begin{proof}[Proof of Theorem \ref{Canonical local period}] By the multiplicity one property of $\Hom_G(\sigma\otimes\sigma_+\otimes\sigma_-,\BC)$, we only need to show $I$ is non-zero. By Proposition \ref{PS} (for $\sigma$ principal series) and easy calculation in remaining cases, one can choose $t$ such that $\Psi_t$ defines a non-zero element in $\Hom_G(\sigma\otimes\sigma_+\otimes\sigma_-,\BC)$ and choose $W\in\sigma$, $W_+\in\sigma_+$,$f_-\in\sigma_-$ such that $\Psi_t(W,W_+,f_-)\neq0$. As in the proof of Proposition \ref{Tri-I}, we can choose $W^\vee_t\in\sigma^\vee$, $W_{t,+}^\vee\in\sigma_+^\vee$, $f_{t,-}^\vee\in I_{\wt{B_*}}^{\wt{G}}\chi_-^{w,-}$ such that $\Psi_{t^\prime}^\vee(W^\vee_t,W_{t,+}^\vee,f_{t,-}^\vee)\neq0$ if and only if $t=t^\prime$. By the decomposition formula, 
$I(W\otimes W_+\otimes f_-,W^\vee_t\otimes W_{t,+}^\vee\otimes f_{t,-}^\vee)\neq0$. Together with the description as integration of matrix coefficients, we find $I$ is non-zero.
\end{proof}
When $\sigma=I_B^G\theta$ is a (irreducible) principal series, we have another decomposation of the integration of matrix coefficients $I$, which is more suitable for unramified calculations.   Consider the $G$-invariant pairing  \[(-,-):\quad I_B^G\theta\times I_B^G\theta^{-1}\lra\BC,\quad (f,f^\vee)\mapsto\int_{K}f(k)f^\vee(k)dk.\]
Embed $\sigma_\pm$ (resp. $\sigma_\pm^\vee$) into $I_{\wt{Z_*}N}^{\wt{G}}\omega_{\pm}\boxtimes\psi^{\pm}$ (resp. $I_{\wt{Z_*}N}^{\wt{G}}\omega_{\pm}^{-1}\boxtimes\psi^{-\pm}$).
Consider the pairing $(-,-)_\pm$ on $\sigma_{\pm}\times\sigma_{\pm}^\vee$ given in term of Whittaker functions.
\begin{prop}\label{Tri-II}
For $*=\emptyset$ or $\vee$, the trilinear form 
\[\Psi^*:\ \sigma_+^*\times\sigma_-^*\times\sigma^*\lra\BC,\quad (W_1^*,W_2^*,f^*)\mapsto\int_{ZN\bs G}\sum_{z\in \wt{Z_*}\bs\wt{Z}}W_1^*(zg)W_2^*(zg)f_3^*(zg)dg\]
is a generator of $\Hom_{G}(\sigma_+^*\otimes\sigma_-^*\otimes\sigma^*,\BC)$. When $I$ is defined  using $(-,-)_*$ for $*=\pm,\emptyset$, then 
\[I(W_+\otimes W_-\otimes f, W_+^\vee\otimes W_-^\vee\otimes f^\vee)=\zeta_F(1)\Psi(W_+,W_-,f)\Psi^\vee(W_+^\vee,W_-^\vee,f^\vee)\]
for any $W_\pm^* \in \sigma_\pm^*$, $f^* \in \sigma^*$ with $* = \emptyset$ and $\vee$.
\end{prop}
\begin{proof} By construction,  $\Psi\in \Hom_G(\sigma_+\otimes\sigma_-\otimes\sigma,\BC)$. Since $i_{\wt{Z_*}N}^{\wt{B}}\omega\boxtimes\psi\subset \CK(\sigma_+,\psi)$, we can choose $W_+\in \sigma_+$ such that $W_1|_{\wt{T}}$ is supported on $s(\diag\{U,1\})\wt{Z_*}$ and constant on $s(\diag\{U,1\})$ for $U\subset\CO^\times$ small enough. Similarly, we can choose $W_-\in\sigma_-$ satisfying the same condition. Assume $\theta$ is trivial on $\diag\{U,1\}$. Choose $K^\prime\subset K$ small enough  such that the covering $p: \wt{G}\to G$ admits a splitting $\kappa$ on $K^\prime$ and $\kappa(K^\prime)$ fixes $W_1$,$W_2$. Let $f\in\sigma^{K^\prime}$ be the function supported on $BK^\prime$ and $f(1)=1$. For such choices, $\Psi(W_+,W_-,f)$ is clearly non-zero and consequently, $\Psi$ is a generator of $\Hom_G(\sigma_+\otimes\sigma_-\otimes\sigma,\BC)$. Similarly, we can show $\Psi^\vee$ is a generator of $\Hom_G(\sigma_+\otimes\sigma_-\otimes\sigma,\BC)$.

Adapting the argument in \cite[Proposition 5.2]{Hsi21}, we have 
\[I(W_+\otimes W_-\otimes f, W_+^\vee\otimes W_-^\vee\otimes f^\vee)=\zeta_F(1)\sum_{z\in \wt{Z_*}\bs\wt{Z}}\Psi(W_+(z\cdot),W_-,f)\Psi^\vee(W_+^\vee(z\cdot),W_-^\vee,f^\vee).\]
Applying the $G$-invariance of $\Psi(W_+(z\cdot),W_-,f)$ to $\wt{Z_*}\subset \wt{G}$, we find $\Psi(W_+(z\cdot),W_-,f)=0$ unless $z\in \wt{Z_*}$. We are done.
\end{proof}
We end this subsection by the unramified calculation. From now on, assume  $|3|=1$.  Assume  $\sigma=I_B^G\theta$ and $\sigma_\pm=I_{\wt{B_*}}^{\wt{G}}\chi_{\pm}$ is unramified. Take $f_*^\prime\neq0\in \sigma_*^{\prime,K}$  spherical  for $\prime=\emptyset,\vee$ and $*=\pm,\emptyset$ and assume the $\wt{G}$-invariant pairing on $\sigma_*\times\sigma_*^\vee$ satisfies that $(f_*,f_*^\vee)=1$.
\begin{thm}\label{unramified calculation}Notations as above. Then
\begin{align*}
    &I(f_+\otimes f_-\otimes f, f_+^\vee\otimes f_-^\vee\otimes f^\vee)\\
    &=\frac{1}{L(1,\sigma,\ad)}\frac{1}{(1-|\varpi|^{1/2}\chi_{+,1}\chi_{-,2}(\varpi^3)\theta_1^2\theta_2(\varpi))(1-|\varpi|^{1/2}\chi_{+,1}\chi_{-,2}
\theta_1(\varpi^3))}\\
&\times\frac{1}{(1-|\varpi|^{1/2}\chi_{+,1}\chi_{-,2}(\varpi^3)\theta_1\theta_2^2(\varpi))(1-|\varpi|^{1/2}\chi_{+,1}\chi_{-,2}\theta_2(\varpi^3))}
\end{align*}
In particular when $\omega_\pm=1$ (thus $\chi_{+,1}\chi_{-,2}(\varpi^3)^2=1$), 
\[I( f_+\otimes f_-\otimes f, f_+^\vee\otimes f_-^\vee\otimes f^\vee)=\begin{cases}\frac{L(1/2,\sigma,\Sym^3)}{L(1,\sigma,\ad)}, &\  \text{if }\chi_{+,1}\chi_{-,2}(\varpi^3)=1;\\
\frac{L(1,\sigma,\Sym^3)}{L(1/2,\sigma,\Sym^3)L(1,\sigma,\ad)}, &\ \text{if }\chi_{+,1}\chi_{-,2}(\varpi^3)=-1.\end{cases}\]
\end{thm}
We interlude with the inner product of spherical vectors. Assume  $\psi$ has conductor zero. Take
\[C_{\pm,\psi^{-\pm}}\in S(\chi_\pm,\psi^{\pm}):=\{C:\ \wt{T}\to\BC| C(ht)=(\delta^{1/2}\chi_\pm(h))^{-1}C(t)\quad\forall h\in \wt{T_*},\ t\in \wt{T}\}\] be the unique function satisfying  the formula \ref{cons}, with support  in $\wt{T_*}\sqcup \wt{T_*}\eta_{-1,1}$ and   $C_{\pm,\psi^{-\pm}}(1)=1$. Define $C_{\pm,\psi^{-\pm}}^\vee\in S(\chi_{\pm}^{-1,w},\psi^{-\pm})$ similarly.
Set $W_{\pm}=(C_{\pm,\psi^{\pm}},f_\pm)$ for the spherical function $f_\pm\in V(\chi_\pm)^{K^*}$ with $f(1)=1$. Set $W_{\pm}^\vee=(C_{\pm,\psi^{-\pm}},f_\pm^\vee)$
with  $f_{\pm}^\vee\in V(\chi_{\pm}^{w,-1})^{K^*}$ and $f_\pm^\vee(1)=1$. Then  by \cite[Theorem I.4.2]{KP84}, for relevant $C$ and $W$ as above   \[W(\eta_{a,b})=\begin{cases}\delta(\eta_{a,b})C(\eta_{a,b}^{w,-1}), & 
\text{if }a\geq b;\\
0, &\ \text{otherwise.}\end{cases}\]
Normalized this way, we have \[(W_{\pm},W_\pm^\vee)=\int_{F^\times}\int_{\wt{Z_*}\bs\wt{Z}}W_\pm W_\pm^\vee\left[\begin{pmatrix}y & 0\\ 0 & 1\end{pmatrix}z\right]d^\times y=\zeta_F(1).\]
Take  $f^*\in\sigma^{K,*}$ such that $f^*(1)=1$ for $*=\emptyset, \vee$.  Note that $(f,f^\vee):=\int_Kf(k)f^\vee(k)dk=1$. Now we can prove Theorem \ref{unramified calculation}.
\begin{proof}
By definition, \cite[Theorem I.4.2]{KP84} and the choice of $C$ 
\begin{align*}
    &\Psi(W_+,W_-,f)=\int_{F^\times}\int_{K}\sum_{z\in \wt{Z_*}\bs \wt{Z}}W_+W_-f\left[z\begin{pmatrix}y & 0\\ 0 & 1\end{pmatrix}k\right]\frac{d^\times y}{|y|}dk\\
    &=\int_{F^\times}\sum_{z\in \wt{Z_*}\bs\wt{Z}}
    W_+W_-f\left[z\begin{pmatrix}y & 0\\ 0 & 1\end{pmatrix}\right]\frac{d^\times y}{|y|}\\
    &=\begin{cases}\sum_{n\geq0}C_{+,\psi}C_{-,\bar{\psi}}(\eta_{0,-3n})\theta_1(\varpi)^{3n} |\varpi|^{\frac{9n}{2}}+\sum_{n\geq0}C_{+,\psi}C_{-,\bar{\psi}}(\eta_{0,-3n-1})\theta_1(\varpi)^{3n+1} |\varpi|^{\frac{9n+3}{2}}, & d=1;\\
    \sum_{n\geq0}C_{+,\psi}C_{-,\bar{\psi}}(\eta_{0,-3n})\theta_1(\varpi)^{3n} |\varpi|^{\frac{9n}{2}}+\sum_{n\geq0}C_{+,\psi}C_{-,\bar{\psi}}(\eta_{0,-3n-2})\theta_1(\varpi)^{3n+1} |\varpi|^{\frac{9n+3}{2}}, & d=3
    \end{cases}\\
    &=\begin{cases}\frac{1+C_{+,\psi}C_{-,\bar{\psi}}(\eta_{0,-1})\theta_1(\varpi)|\varpi|^{3/2}}{1-|\varpi|^{3/2}\chi_{+,2}\chi_{-,2}\theta_1(\varpi^3)}, & d=1;\\
    \frac{1+C_{+,\psi}C_{-,\bar{\psi}}(\eta_{-1,-2})\theta_1(\varpi)|\varpi|^{3/2}}{1-|\varpi|^{3/2}\chi_{+,2}\chi_{-,2}\theta_1(\varpi^3)}, & d=3
    \end{cases}\\
&=\frac{1+|\varpi|^{-1/2}\chi_{+,1}\chi_{-,1}(\varpi^3)\theta_1^2\theta_2(\varpi)}{1-|\varpi|^{3/2}\chi_{+,2}\chi_{-,2}\theta_1(\varpi^3)}\\
&=\frac{1-|\varpi|\theta_1\theta_2^{-1}(\varpi)}{(1-|\varpi|^{-1/2}\chi_{+,1}\chi_{-,1}(\varpi^3)\theta_1^2\theta_2(\varpi))(1-|\varpi|^{3/2}\chi_{+,2}\chi_{-,2}\theta_1(\varpi^3))}.
\end{align*}
Here in the last equalities, we use  $\omega\omega_+\omega_-=1$ on $\wt{Z^{(c)}}$ and $\chi_{\pm}$ is exceptional. 

Use the invariant pairing on $\sigma_{\pm}\times \sigma_{\pm}^\vee$ in term of the Whittaker models and on $\sigma\times\sigma^\vee$ in term of induced model. By the decomposition formula in Proposition \ref{Tri-II}, we have
\begin{align*}
&I( W_+\otimes W_-\otimes f, W_+^\vee\otimes W_-^\vee\otimes f^\vee)\\
&=\zeta_F(1)\frac{1-|\varpi|\theta_1\theta_2^{-1}(\varpi)}{(1-|\varpi|^{-1/2}\chi_{+,1}\chi_{-,1}(\varpi^3)\theta_1^2\theta_2(\varpi))(1-|\varpi|^{3/2}\chi_{+,2}\chi_{-,2}\theta_1(\varpi^3))}\\
&\times \frac{1-|\varpi|\theta_1^{-1}\theta_2(\varpi)}{(1-|\varpi|^{-1/2}\chi_{+,2}^{-1}\chi_{-,2}^{-1}(\varpi^3)\theta_1^{-2}\theta_2^{-1}(\varpi))(1-|\varpi|^{3/2}\chi_{+,1}^{-1}\chi_{-,1}^{-1}\theta_1^{-1}(\varpi^3))}\\
&=\frac{\zeta_F(1)^2}{L(1,\sigma,\ad)}\frac{1}{(1-|\varpi|^{-1/2}\chi_{+,1}\chi_{-,1}(\varpi^3)\theta_1^2\theta_2(\varpi))(1-|\varpi|^{3/2}\chi_{+,2}\chi_{-,2}
\theta_1(\varpi^3))}\\
&\times\frac{1}{(1-|\varpi|^{-1/2}\chi_{+,1}\chi_{-,1}(\varpi^3)\theta_1\theta_2^2(\varpi))(1-|\varpi|^{3/2}\chi_{+,2}\chi_{-,2}\theta_2(\varpi^3))}.
\end{align*}
Here again in the last equalities, we use the central character condition $\omega\omega_+\omega_-=1$ on $\wt{Z^{(c)}}$. By the assumption $\chi_\pm$ is exceptional and $\omega_{\pm}$ is trivial, we obtain when $\omega_\pm=1$, 
\[I( W_+\otimes W_-\otimes f, W_+^\vee\otimes W_-^\vee\otimes f^\vee)=\begin{cases}\frac{\zeta_F^2(1)L(1/2,\sigma,\Sym^3)}{L(1,\sigma,\ad)}, &\text{if } \chi_{+,1}\chi_{-,1}(\varpi^3)=|\varpi|;\\
\frac{\zeta_F^2(1)L(1,\sigma,\Sym^3)}{L(1/2,\sigma,\Sym^3)L(1,\sigma,\ad)}, &\text{if } \chi_{+,1}\chi_{-,1}(\varpi^3)=-|\varpi|.\end{cases}\]
Note that $(W_\pm,W_\pm^\vee)=\zeta_F(1)$ and the stated formula follows.
\end{proof}
\subsection{Local co-period: Archimedean case}\label{Archimedean}
In this subsection, we study the trilinear period in the Archimedean case. Since $F$ is archimedean and contains $\mu_3$, we have $F\cong\BC$ and the metaplectic cover splits, i.e. $\wt{G}\cong \mu_3\times G$. Again, a smooth character $\chi: \wt{T_*}\to\BC^\times$ is called exceptional if $\chi(s(\diag\{x^3,x^{-3}\}))=|x|$ for all $x\in \BC^\times$. As $\wt{G}=\mu_3\times G$, one has  $V(\chi)=I_{\wt{B}}^{\wt{G}}\chi$ is irreducible for $\chi$ exceptional. 

Take $\chi_{\pm}$ on $\wt{T_*}$ exceptional and $\chi_\pm|_{\mu_3}=\epsilon^\pm$. Take any smooth character $\theta: T\to\BC^\times$ and set $\sigma=I_B^G\theta$. Let $\omega_*$ be the central character of $\sigma_*$ for $*=\pm,\emptyset$. Clearly, $\Hom_G(\sigma\otimes\sigma_+\otimes\sigma_-,\BC)=0$ unless $\omega\omega_+\omega_-=1$. In the following, we assume $\omega\omega_+\omega_-=1$.
\begin{prop}[\cite{Lok01}] $\dim \Hom_G(\sigma\otimes\sigma_+\otimes\sigma_-,\BC)=1$.
\end{prop}
Assume now moreover $\omega_*$ is unitary. Let $\lambda(\sigma)$ be the real number such that $\theta\delta^{-\lambda(\sigma)}$ is unitary. Assume in the following $|\lambda(\sigma)|<1/6$, which holds if $\sigma$ is the local component of a cuspidal automorphic $\GL_2(\BA)$-representation. The following result follows from \cite[Proposition 2.1]{Ich08}:
\begin{prop}\label{CanA}Fix  $\wt{G}$-invariant pairings $(-,-)_*$ on $\sigma_*\times\sigma_*^\vee$ for $*=\emptyset,\pm$. The integration of matrix coefficients 
\[\int_{Z\bs G}(f,f^\vee)(f_+,f_+^\vee)_+(f_-,f_-^\vee)_-dg\]
converges absolutely for any $f_*^\prime\in\sigma_*^\prime$, $*=\emptyset,\pm$ and $\prime=\emptyset,\vee$. The space \[\Hom_{G}(\sigma\otimes\sigma_+\otimes\sigma_-,\BC)\times \Hom_{G}(\sigma^\vee\otimes\sigma_+^\vee\otimes\sigma_-^\vee,\BC)\] has a generator given by 
\[I:\ (f\otimes f_+\otimes f_-,f^\vee\otimes f_+^\vee\otimes f_-^\vee)\mapsto \int_{Z\bs G}(f,f^\vee)(f_+,f_+^\vee)_+(f_-,f_-^\vee)_-dg.\]
\end{prop}
Embed  $\sigma_\pm$ (resp, $\sigma_\pm^\vee$) into $I_{\wt{ZN}}^{\wt{G}}\omega_{\pm}\boxtimes\psi^{\pm}$ (resp. $I_{\wt{ZN}}^{\wt{G}}\omega_{\pm}^{-1}\boxtimes\psi^{-\pm}$) and consider the pairing 
\[(-,-)_\pm: \sigma_\pm\times\sigma_\pm^\vee,\quad (W_\pm,W_\pm^\vee)\mapsto\int_{F^\times}W_\pm W_\pm^\vee\left[\begin{pmatrix}y & 0\\ 0 & 1\end{pmatrix}\right]d^\times y\]
Consider the pairing 
\[(-,-): \sigma\times\sigma^\vee\lra\BC, \quad (f,f^\vee)\mapsto\int_{K}ff^\vee(k)dk\]
where $K=\RU(2)\subset G$ is the maximal compact subgroup.
\begin{prop}\label{Tri-III}Notations as above. For $*=\emptyset,\vee$, the integration 
\[\Psi^*(W_+^*,W_-^*,f^*):=\int_{ZN\bs G}W_+^*W_-^*f^*dg\]
converges absolutely for any $W_\pm^*\in \sigma_\pm^*$ and $f^*\in\sigma^*$. Moreover, 
\[I(W_+\otimes W_-\otimes f, W_+^\vee\otimes W_-^\vee\otimes f^\vee)=\zeta_F(1)\Psi(W_+,W_-,f)\Psi^\vee(W_+^\vee,W_-^\vee,f^\vee).\]
\end{prop}
\begin{proof}By the gauge estimation in \cite[Proposition 2.1]{Ich08}, we know the integration $\Psi^*(W^*_+,W^*_-,f^*)$ converges absolutely. One can show the decomposition formula adapting \cite[Proposition 5.2]{Hsi21}. Alternatively, we can use the following deformation argument. The formula holds when $\sigma$ is tempered  by \cite[Lemma 4.2]{MV10}. Consider $\sigma_{+,s}^*=(I_{\wt{B_*}}^{\wt{G}}\chi_+\delta^s)^*$ and a flat section $W_{+,s}^*\in \sigma_{+,s}^*$. When the absolute value of $s$ is small, both $I(W_{+,s}\otimes W_-\otimes f, W_{+,s}^\vee\otimes W_-^\vee\otimes f^\vee)$ and $\Psi(W_{+,s},W_-,f)\Psi^\vee(W_{+,s}^\vee,W_-^\vee,f^\vee)$ are holomorphic functions in   $s$. Choose   test vectors $f^*$, $W^*_\pm$ for  $\Psi^*$. Then the ratio $\frac{I}{\Psi\Psi^\vee}$ is meromorphic and  equals to $\zeta_F(1)$ when $\chi\delta^s$ is unitary. This forces it to be  $\zeta_F(1)$ when  $\Re(s)$ is small. We are done.
\end{proof}
We end this subsection with an equality of $L$-factors. Recall that
\[L(s,\phi)=2(2\pi)^{-(s+t+\frac{|\ell|}{2})}\Gamma\left(s+t+\frac{|\ell|}{2}\right)\]
 with the character \[\phi=(t,\ell):\ \BC^\times\to \BC^\times,\quad re^{i\theta}\mapsto r^{2t}e^{i\ell \theta }\quad t\in\BC, \ell\in\BN.\] 
\begin{prop}Assume  on $\mu_3\times T=\wt{T}$,  $\chi_{\pm}=\left(\epsilon^\pm, (|\cdot|^{1/6},|\cdot|^{-1/6})\right)$ and $\theta=((t,\ell),(-t,-\ell))$. Then 
\begin{align*}
(2\pi)^2\frac{L(1/2,\sigma,\Sym^3)}{L(1,\sigma,\ad)}=3^{3|\ell|}\frac{\zeta_\BC^2(2)L(1/2,\sigma\times\sigma_+\times\sigma_-)}{L(1,\sigma,\ad)L(1,\sigma_+,\ad)L(1,\sigma_-,\ad)}.
\end{align*}
\end{prop}
\begin{proof}Under the assumption, \[L(1,\sigma_\pm,\ad)=\zeta_\BC(1)\zeta_\BC(4/3)\zeta_\BC(2/3),\]
\[L(\frac{1}{2},\sigma,\Sym^3)=\zeta_\BC\left(\frac{1}{2}+3t+\frac{3|\ell|}{2}\right)\zeta_\BC\left(\frac{1}{2}-3t+\frac{3|\ell|}{2}\right)\zeta_\BC\left(\frac{1}{2}+t+\frac{|\ell|}{2}\right)\zeta_\BC\left(\frac{1}{2}-t+\frac{|\ell|}{2}\right)\]
and
\begin{align*}L(1/2,\sigma\times\sigma_+\times\sigma_-)&=\zeta_\BC\left(\frac{5}{6}+t+\frac{|\ell|}{2})\zeta_\BC(\frac{1}{6}+t+\frac{|\ell|}{2}\right)\zeta_\BC^2\left(\frac{1}{2}+t+\frac{|\ell|}{2}\right)\\
&\times \zeta_\BC\left(\frac{5}{6}-t+\frac{|\ell|}{2}\right)\zeta_\BC\left(\frac{1}{6}-t+\frac{|\ell|}{2}\right)\zeta_\BC^2\left(\frac{1}{2}-t+\frac{|\ell|}{2}\right).
\end{align*}
From the multiplication formula of $\Gamma$-function:
\[\Gamma(z)\Gamma\left(z+\frac{1}{3}\right)\Gamma\left(z+\frac{2}{3}\right)=(2\pi)3^{\frac{1}{2}-3z}\Gamma(3z)\]
we deduce the stated formula.
\end{proof}
\section{Global co-period}
In this section, we shall study the global co-period. We shall formulate an Ichino-Ikeda type conjecture relating the co-period  and the central critical value of symmetric cube L-function. We will verify this conjecture for Eisenstein series.

Let  $F$ be a number field containing all $3$-th roots of unity $\mu_3$. Let $\CM$ be the set of all places of $F$.  Identify $\mu_3$ with  $\mu_3(F_v)$ for all places $v\in\CM$ and fix an embedding $\epsilon:\ \mu_3\to\BC^\times$. Let $\BA=\otimes_{v\in\CM}^\prime F_v$ be the adeles.  Fix an additive character $\psi:\ F\bs\BA\to\BC^\times$  such that $\psi_{\BC}=\exp^{2\pi i(z+\bar{z})}$. Let $\BA^\times=\otimes_{v\in\CM}^\prime F_v^\times$ be the ideles and $|\cdot|_\BA:=\prod_v |\cdot|$ be the norm character on $\BA^\times$.

Let $G=\GL_2$ be the reductive group and $Z\subset T=AZ\subset B=TN\subset G$ be the center, the diagonal torus and the upper Borel subgroup respectively. Here  $A=\diag\{*,1\}\subset T$ and $N=\begin{pmatrix}1 & * \\ 0 &1\end{pmatrix}$. Let $\delta=\prod_v\delta_v:\ B(\BA)\to\BC^\times$
be the modulus character. Let $K=\prod_v K_v$ where $K_v=\GL_2(\CO_{F_v})$ when $v<\infty$ and $K_v=U(2)(F_v)$ when $v\mid \infty$. Let $H:\ G(\BA)\to \BC$ be the height function given by 
\[ntk\mapsto\log\delta(t),\quad n\in N(\BA),\ t\in T(\BA),\ k\in K.\]

In the previous section, we choose Haar measures $d x_v$ (resp. $d^\times y$ resp. $d g_v$) on $F_v$ (resp. $F_v^\times$. resp. $G(F_v)$) for each $v$. Let $dx=\prod dx_v$ (resp.  $d^\times y=\prod d^\times y_v$ resp.  $dg=\prod_v dg_v$) be the Haar measure on $\BA$ (resp. $\BA^\times$, resp.  $G(\BA)$).  Let $d^\tau g$ be the Haar measure on $G(\BA)$ such that $\Vol(Z(\BA)G(F)\bs G(\BA),d^\tau g)=2$. 
Note that \[\Vol(Z(\BA)G(F)\bs G(\BA),d g)=\zeta_F(2)\prod_{v<\infty}\Vol^{-1}(\CO_{F_v},dx)\Vol(Z(\BA)G(F)\bs G(\BA),d^\tau g)\]

\subsection{Preliminaries on theta series}
Fix an integer $c\in\BZ/3\BZ$ and define the metaplectic cover  $p:\ \wt{G}(F_v)\to G(F_v)$ for all $v$ with respect to the Kubota 2-cocycle $\sigma_K^{(c)}$. 
For each $v\in\CM$, fix a  subgroup $K_v^{\prime}\subset K_v$ such that $p:\ \wt{K_v^\prime}\to K_v^\prime$ admits a group section $\kappa_v$. Assume for almost all $v$, $\kappa_v$ is the Kubota lifting and  set $K_v^{\prime,*}=\kappa_v(K_v^\prime)$.
Formulate the restricted tensor product $\prod_{v\in\CM}^\prime \wt{G}(F_v)$ with respect to the chosen $K_v^{\prime,*}$. Let \[\wt{G(\BA)}:= M\bs \prod_v^\prime \wt{G}(F_v),\quad M=\{(\xi_v)\in\prod_v\mu_3| \prod_v\xi_v=1,\ \xi_v=1\quad\forall’  v\in \CM\}\]
Then we have an exact sequence 
\[1\lra\mu_3\lra\wt{G}(\BA)\stackrel{p}{\lra} G(\BA)\lra 1.\]
According to \cite[Section 0.II]{KP84}, the preferred sections $\{s_v|v\in\CM\}$ induces a group splitting 
\[s:\ G(F)\to \wt{G}(\BA).\] 
Via $s$, we will view $G(F)$ as a subgroup of $\wt{G}(\BA)$.

For group $H\subset G(\BA)$, let $\wt{H}=p^{-1}(H)$. Recall that for each place $v\in\CM$, 
\[T_3(F_v)=\{\diag\{a,b\}\in T(F_v)|a,b\in F_v^{\times,3}\},\]
\[ A_3(F_v)=\{\diag\{a,1\}\in A(F_v)|a\in F_v^{\times,3}\},\quad Z_3(F_v)=\{\diag\{z,z\}\in Z(F_v)|z\in F_v^{\times,3}\}.\]

Let $B_*(\BA)=\prod_v^\prime B_*(F_v)$ and $T_3(\BA)=\prod_v^\prime T_3(F_v)\subset T_*(\BA)=\prod^\prime_v T_*(F_v)$. Let 
\[A_3(\BA)=\prod^\prime_v A_3(F_v)\subset A_*(\BA):=A(F)A_3(\BA),\quad Z^{(c)}(\BA)=\prod_v^\prime Z^{(c)}(F_v)\subset Z_*(\BA)=\prod^\prime_v Z_{*}(F_v).\]  
By the construction, $Z_*(\BA)\subset T_*(\BA)$  and $T_*(\BA)=Z_*(\BA)A_*(\BA)$.   Let $T^\prime(\BA)=\prod_v^\prime T^\prime(F_v)$ where \[T^\prime(F_v)=\{(a,b)\in T(F_v)| ab\in Z_*(F_v)\}.\] 

According to \cite[Lemma II.1.1]{KP84}, $\wt{T_*}(\BA)=T(F)\wt{T_3Z^{(c)}}(\BA)$ is a maximal abelian subgroup in $\wt{T}(\BA)$. We also record the following result, which follows from the same  argument:
\begin{lem}\label{center}$\wt{Z_*}(\BA)=Z(F)\wt{Z^{(c)}}(\BA)$ is a maximal abelian subgroup in $\wt{Z}(\BA)$.
\end{lem}
Let $\chi$ be a smooth character on $\wt{T_3Z^{(c)}}(\BA)$, trivial on $T(F)\cap\wt{T_3Z^{(c)}}(\BA)$. Extend $\chi$ to $\wt{T_*}(\BA)$ by requiring $\chi(T(F))=1$. Consider the induced representation on the space of $\wt{K}$-smooth functions
\[I_{\wt{B_*}(\BA)}^{\wt{G}(\BA)}\chi=\left\{f:\ \wt{G}(\BA)\to\BC| f(bg)=\chi(b)\delta^{1/2}(b)f(g)\quad  \forall\ b\in\wt{B_*}(\BA), g\in\wt{G}(\BA), f\ \text{is}\ \wt{K}-\text{smooth}\right\}.\]
We have the intertwining operator $I=\otimes I_v: I_{\wt{B_*}(\BA)}^{\wt{G}(\BA)}\chi\to I_{\wt{B_*}(\BA)}^{\wt{G}(\BA)}\chi^w$ with  
\[I_v:\ I_{\wt{B_*}(F_v)}^{\wt{G}(F_v)}\chi_v\lra I_{\wt{B_*}(F_v)}^{\wt{G}(F_v)}\chi_v^w, \quad f\mapsto \left( g\mapsto\int_{N(F_v)}f(wng)dn\right).\]
Assume in the following $\chi|_{\wt{Z^{(c)}}(\BA)}$ is unitary. Let $\lambda(\chi)\in \BR$ such that $\chi\delta^{-\lambda(\chi)}$ is unitary. When  $\lambda(\chi)>1/2$, the infinite sum $\sum_{\gamma\in B(F)\bs G(F)}f(\gamma g)$ converges  absolutely for any $\wt{K}$-smooth $f\in I_{\wt{B_*}(\BA)}^{\wt{G}(\BA)}\chi$ and any $g\in\wt{G}(\BA)$. Thus when $\lambda(\chi)>1/2$, we have the  Eisenstein series machinery:
\[I_{\wt{B_*}(\BA)}^{\wt{G}(\BA)}\chi\lra\CA(G(F)\bs \wt{G}(\BA),\BC),\quad  f\mapsto \left(E(-,\chi,f):\ g\mapsto \sum_{\gamma\in B(F)\bs G(F)}f(\gamma g)\right).\]
In general, we consider the meromorphic family $E(-,\chi\delta^s, f_s)$ where $f_s$ is the flat section such that \[f_s(ntk):=f(ntk)\delta^s(t),\quad n\in N(\BA),\ t\in T(\BA),\ k\in\prod_v K_v.\]

The character $\chi$ is called \textbf{exceptional} if $\chi\left(s(x^3,x^{-3})\right)=|x|_{\BA}$ for all $x\in \BA^\times$. In this case, $E(-,\chi\delta^s, f_s)$ has a simple pole at $s=0$. Denote  by $\theta(-,\chi,f):=\Res_{s=0} E(-,\chi\delta^s,f_s)$.  Then by \cite[Theorem II.2.1]{KP84},  $\theta(-,\chi,f)$ forms an irreducible $\wt{G}(\BA)$-representation $\Theta(\chi)$ when $f$ varies in $I_{\wt{B_*}(\BA)}^{\wt{G}(\BA)}(\chi)$. Moreover one has the tensor product decomposition over $\BC[\mu_3]$
\[V_0(\chi)=\otimes_{\BC[\mu_3]}^\prime V_0(\chi_v).\]

We end this subsection by  the Whittaker-Fourier coefficients of $\theta(g,\chi,f)$. Let $\CS(\wt{T}(\BA),\BC)$ be the space of functions $f:\ \wt{T}(\BA)\to\BC$ such that
\begin{itemize}
    \item $f(th)=\chi\delta^{1/2}(t)^{-1}f(h)$  for all $t\in \wt{T_*}(\BA)$, $h\in \wt{T}(\BA)$,
    \item for all place $v\in\CM$, it satisfies the relation given by Formula \ref{cons}.
\end{itemize}
By Proposition \ref{Jac}, $\CS(\wt{T}(\BA),\BC)$ is an irreducible $\wt{Z}(\BA)$-representation. Since $\wt{T^\prime}(\BA)\wt{Z}(\BA)=\wt{T}(\BA)$, we deduce from \cite[Corollary I.3.4]{KP84} and the very construction of $\wt{T^\prime}(\BA)$ that up to scalar, there exists a unique function $C=\otimes_v C_v\in \CS(\wt{T}(\BA),\BC)$ supported on $\wt{T^\prime}(\BA)$. 
\begin{prop}\label{Fourier coefficients} For any $f=\otimes_vf_v\in I_{\wt{B_*}(\BA)}^{\wt{G}(\BA)}\chi$ and $g\in \wt{G}(\BA)$
\[\int_{N(F)\bs N(\BA)}\theta(ng,\chi,f)dn=\zeta_F(2)^{-1}\prod_v \zeta_{F_v}^{-1}(1)\zeta_{F_v}(2)I_v(f_v)(g).\]
Moreover, up to scalar \[\int_{N(F)\bs N(\BA)}\bar{\psi}(n)\theta(ng,\chi,f)dn=\bigotimes_v(C_v, f_v)(g_v):=\bigotimes_v\left(\sum_{t_v\in \wt{T_*}(F_v)\bs\wt{T^\prime}(F_v)}C_v(t_v)\int_{N(F_v)}f_v(t_vwn_vg_v)dn_v\right).\]
\end{prop}
\begin{proof}As explained in \cite[Page 114]{KP84}, the first formula follows from \cite[Proposition II.1.2]{KP84}. For the second formula, let $W(g)=\int_{N(F)\bs N(\BA)}\bar{\psi}(n)\theta(ng,\chi,f)dn$. By \cite[Theorem II.2.2]{KP84}, there exists $C\in \CS(\wt{T_*}(\BA),\BC)$ such that 
\[W(g)=\sum_{\eta\in \wt{T_*}(\BA)\bs\wt{T}(\BA)}C(\eta)\prod_v (\lambda_{\eta_v},f_v)(g),\quad (\lambda_{\eta_v},f_v)(g):=\int_{N(F_v)}f(\eta_vwng)dn.\] Extend the central character $\xi$ of $V_0(\chi)$ to $\wt{Z_*}(\BA)$ by requiring $\xi(Z(F))=1$. Then for all $z\in\wt{Z_*}(\BA)$, $\theta(zg,\chi,f)=\xi(z)\theta(g,\chi,f)$ and hence $W(zg)=\xi(z)W(g)$.  Since $(\lambda_{\eta_v},f_v)(zg)=\xi_v(z)(\lambda_{\eta_v},f_v)(g)$ for all $z\in\wt{Z_*}(F_v)$ if and only if $\eta_v\in \wt{T_*}(F_v)\bs \wt{T^\prime}(F_v)$, we find that $C(\eta)=0$ unless $\eta\in \wt{T^\prime}(\BA)$. By the uniqueness of such function, we are done.
\end{proof}
\subsection{Co-periods for Eisenstein series}
In this subsection, we consider the co-period integral for an Eisenstein series on $G(\BA)$ and deduce the Petersson inner product formula for metaplectic theta series.  We first fix some notations. Let $\chi_{\pm}$ be exceptional characters on $T(F)\bs \wt{T_*}(\BA)$ such that $\chi_{\pm}|_{\mu_3}=\epsilon^\pm$
and  $\chi_{\pm}|_{\wt{Z^{(c)}}}$ is unitary. Let $\sigma_{\pm}=V_0(\chi_\pm)=\otimes^\prime_{\BC[\mu_3]}\sigma_{\pm,v}$ be the irreducible automorphic $\wt{G}$-representation introduced in the previous subsection. Denote by $\omega_{\pm}$ the central character of $\sigma_{\pm}$. Let $\theta=(\theta_1,\theta_2):\ T(F)\bs T(\BA)\to\BC$ be a smooth character such that $\omega:=\theta_1\theta_2$ is unitary. Let $\lambda(\theta)$  be the real number such that $\theta\delta^{-\lambda(\theta)}$ is unitary. If $\lambda(\theta)>1/2$, the summation $\sum_{\gamma\in B(F)\bs G(F)}f(\gamma g)$ is absolutely convergent for all $g\in G(\BA)$. In this case, let $E(g,\theta,f)=\sum_{\gamma\in B(F)\bs G(F)}f(\gamma g)$ for $f\in I_{B(\BA)}^{G(\BA)}\theta$.  In general, we shall consider the meromorphic family of Eisenstein series $E(g,\theta\delta^s,f_s)=\sum_{\gamma\in B(F)\bs G(F)}f_s(\gamma g)$ where $f_s$ is the flat section attached to $f\in I_{B(\BA)}^{G(\BA)}\theta$.   Assume throughout  that $\omega\omega_+\omega_-=1$ and $\chi_+\chi_-(s(\diag\{x^3,y^3\}))=|x|_{\BA}|y|_{\BA}^{-1}$. The co-period we shall consider is an $\wt{G}(\BA)$-invariant   form on $I_{B(\BA)}^{G(\BA)}\theta\delta^s \times \sigma_+\times\sigma_-$, realized as an integration of automorphic forms. Take $\varphi_\pm\in \sigma_{\pm}$ and  $\varphi_s:=E(-,\theta\delta^s,f_s)$ for $f\in I_{B(\BA)}^{G(\BA)}\theta$. The naive  co-period integral 
\[\int_{Z_*(\BA)G(F)\bs G(\BA)}\varphi_s(g)\varphi_+(g)\varphi_-(g)dg\]
may be divergent. To regularize it, we shall apply a mixed truncation operator  following \cite{Yam18}. 

To define the truncation operator, we need some notations: 
\begin{itemize}
    \item For any automorphic form $\phi$ on $G(\BA)$, set $\varphi_N(g)=\int_{N(F)\bs N(\BA)}\phi(ng)dn$.
    \item Let $\tau(x) = \begin{cases} 1 & x > 0,\\
0 & x\leq0\end{cases}$  be the characteristic function on $(0,\infty)\subset\BR$.
    \item Let $H$ be the height function on $G(\BA)$:
    \[H(ntk)=\log \delta(t),\quad n\in N(\BA),\ t\in T(\BA),\ k\in K=\prod_vK_v\] 
    \end{itemize} 
Take any two genuine automorphic forms $\phi_+$, $\phi_-$ on $\wt{G}(\BA)$ such that $\phi_+\otimes\phi_-|_{\mu_3}=1$ (thus  $\phi_+\otimes\phi_-$ is actually an automorphic form on the digonal $G(\BA)$).
For any positive number $T\gg0$, set
\[\wedge^T (\phi_+\otimes\phi_-)(g) = \phi_+\otimes\phi_-(g) - \sum_{\gamma \in B(F) \bs G(F)} \phi_{+,N}(\gamma g)\phi_{-,N}\tau(H(\gamma g) - T).\]
According to \cite[Lemma 3.1]{Yam18}, $\wedge^T (\phi_+\otimes\phi_-)(g)$ is rapidly decreasing on $G(F)\bs G(\BA)^1$ with \[G(\BA)^1=\{g\in G(\BA)||\det(g)|_\BA=1\}.\] 
Take an automorphic form $\phi$ on $G(\BA)$. Assume all $\phi_?$, $?=\pm,\emptyset$ have central characters, denoted by $\omega_?$ and $\omega\omega_+\omega_-=1$. Note that $(Z_*(\BA)\cap G(\BA)^1)\bs G(\BA)^1\cong Z_*(\BA)\bs G(\BA)$. Thus by \cite[Proposition 3.2]{Yam18},  the integral  $\int_{G(F)Z_*(\BA)\bs G(\BA)}\phi(g)\wedge^T(\phi_+\otimes\phi_-)(g)dg$ converges and has the form $\sum_{\lambda\in S}P_\lambda(T)e^{\lambda T}$ for some finite set $S\subset \BC$. Here $P_\lambda(T)$ are polynomials of $T$ and the set $S$ depends only on $\phi$ and $\phi_\pm$. The regularized period integral  
$P(\phi,\phi_+,\phi_-)$ is defined to be $P_0(0)$.

Now we apply this general construction to $\varphi=E(-,\theta\delta^s,f_s)$ and $\varphi_\pm\in \sigma_\pm$.
\begin{thm}\label{co-period for Eisenstein family}For $\varphi_{\pm}=\otimes^\prime_v \varphi_{\pm,v}\in V_0(\chi_\pm)=\otimes_{\BC[\mu_3]}^\prime V_0(\chi_{\pm,v})$, let \[W_{\pm}(g):=\int_{N(F)\bs N(\BA)}\overline{\psi^{\pm}}(n)\varphi_{\pm}(ng)dn=\prod_v W_{\pm,v}(g_v).\] 
Take  $f=\otimes_v^\prime f_v\in I_{B(\BA)}^{G(\BA)}\theta$. Then for $\Re(s)\gg0$,
\begin{align*}P(\varphi_s,\varphi_+,\varphi_{-})&=\prod_{v}\int_{Z_*(F_v)N(F_v)\bs G(F_v)}f_{s,v}(g_v)W_{+,v}(g_v)W_{-,v}(g_v)dg_v.
\end{align*}
\end{thm}
\begin{proof}
Consider the Whittaker-Fourier expansion 
\[\varphi_{\pm}(g)=\sum_{\gamma\in F} W_{\pm,\gamma}(g),\quad W_{\pm,\gamma}(g):=\int_{N(F) \bs N(\BA)}\overline{\psi^{\pm}}(\gamma n)\varphi_{\pm}(ng)dn.\]
Set $\varphi_{\pm,N}=W_{\pm,0}(g)$ and $W_{\pm}(g)=W_{\pm,1}(g)$. Since $G(F)=B(F)\sqcup B(F)wN(F)$ and $H(ng)=H(g)$ for all $n\in N(\BA)$ and $g\in G(\BA)$
\begin{align*}
\int_{N(F)\bs N(\BA)}\wedge^T(\varphi_{+}\otimes\varphi_{-})(ng)dn&=\sum_{\gamma\in F^\times}W_{+,\gamma} W_{-,\gamma}(g)+\varphi_{+,N}\varphi_{-,N}(g)(1-\tau(H(g)-T))\\
&-\int_{N(\BA)}\varphi_+\varphi_-(wng)\tau(H(wng)-T)dn.
\end{align*}

By this formula, we have
\begin{align*}
&\int_{Z_*(\BA)G(F)\bs G(\BA)} \varphi_s(g)\wedge^T(\varphi_{+}\otimes\varphi_{-})(g)dg\\
&=\int_{Z_*(\BA)B(F)\bs G(\BA)}f_s(g)\wedge^T(\varphi_{+}\otimes\varphi_{-})(g)dg\\
&=\int_{Z_*(\BA)A(F)N(\BA)\bs G(\BA)}f_s(g)\int_{N(F)\bs N(\BA)}\wedge^T(\varphi_{+}\otimes\varphi_{-})(ng)dn dg\\
&=\int_{Z_*(\BA)A(F)N(\BA)\bs G(\BA)}f_s(g)\left(\sum_{\gamma\in F^\times}W_{+,\gamma}W_{-,\gamma}(g)\right) dg\\
&+\int_{Z_*(\BA)A(F)N(\BA)\bs G(\BA)}f_s(g)\varphi_{+,N}\varphi_{-,N}(g)(1-\tau(H(g)-T))dg\\
&-\int_{Z_*(\BA)A(F)N(\BA)\bs G(\BA)}f_s(g)\int_{N(\BA)}\varphi_{+,N}\varphi_{-,N}(wng)\tau(H(wng)-T)dndg.\\
\end{align*}

The first term is independent of $T$ and 
\[\int_{Z_*(\BA)A(F)N(\BA)\bs G(\BA)}f_s(g)\left(\sum_{\gamma\in F^\times}W_{+,\gamma}W_{-,\gamma}(g)\right) dg=\int_{Z_*(\BA)N(\BA)\bs G(\BA)}f_s(g)W_+(g)W_-(g)dg.\]
Note that the integration in the right hand side  converges absolutely for $\Re(s)\gg0$ by the gauge estimation of $W_{\pm}(g)$ in \cite[Theorem I.4.1]{KP84}.

For the second term, note that
\begin{align*}
&\int_{Z_*(\BA)A(F)N(\BA)\bs G(\BA)}f_s(g)\varphi_{+,N}(g)\varphi_{-,N}(g)(1-\tau(H(g)-T)) dg\\
&=\int_{Z_*(\BA)A_*(\BA)N(\BA)\bs G(\BA)}\int_{A(F)\bs A_*(\BA)}f_s(ag)\varphi_{+,N}(ag)\varphi_{-,N}(ag)(1-\tau(H(ag)-T))\frac{d^\times a}{|a|_\BA}dg\\
&=\int_{Z_*(\BA)A_*(\BA)N(\BA)\bs G(\BA)}\left(\int_{ 0<|a|_\BA<e^{T-H(g)}, a\in A(F)\bs A_*(\BA)}\theta(a)\chi_+^w\chi_-^w(a)|a|^{1/2+s}d^\times a\right)f_s(g)\varphi_{+,N}(g)\varphi_{-,N}(g)dg.
\end{align*}

Here for the second equality, we use the formula of $\varphi_{\pm,N}(g)$ in Proposition \ref{Fourier coefficients}. Note that by the $\wt{K}$-smooth condition on $I_{\wt{B_*}(\BA)}^{\wt{G}(\BA)}\chi_{\pm}$, the outer integration is actually a finite sum.
Clearly, the inner integration is zero unless $\theta_1(a)\chi_{+,2}\chi_{-,2}(x)=|x|_\BA^{s_0}$ on $A_*(\BA)$ for some $s_0\in\BC$. In this case,   the whole integration has the form $P_{\lambda}(T)e^{\lambda T}$ for $\lambda=s_0+s+1/2$.

For the third term, note that \begin{align*}&\int_{Z_*(\BA)A(F)N(\BA)\bs G(\BA)}f_s(g)\int_{N(\BA)}\varphi_{+,N}\varphi_{-,N}(wng)\tau(H(wng)-T)dndg\\
=&\int_{Z_*(\BA)A(F)N(\BA)\bs G(\BA)}f_s(g)\varphi_{+,N}\varphi_{-,N}(wg)\tau(H(wg)-T)dg\\
=&\int_{Z_*(\BA)A_*(\BA)N(\BA)\bs G(\BA)}\int_{A(F)\bs A_*(\BA)}f_s(ag)\varphi_{+,N}\varphi_{-,N}(wag)\tau(H(wag)-T)\frac{d^\times a}{|a|_\BA}dg\\
=&\int_{Z_*(\BA)A_*(\BA)N(\BA)\bs G(\BA)}\left(\int_{A(F)\bs A_*(\BA), 0<|a|_\BA<e^{H(wg)-T}}\theta(a)\chi_+\chi_-(a)|a|^{1/2+s}d^\times a\right)f_s(g)\varphi_{+,N}\varphi_{-,N}(wg)dg.
\end{align*}
Again the outer integration is actually a finite sum. By the exceptional assumption on $\chi_{\pm}$, the inner integration is zero unless $\theta_1(a)\chi_{+,1}\chi_{-,1}(a)=|a|_\BA^{s_0+2/3}$. In this case, the whole integration has the form $P_{\lambda}(T)e^{\lambda T}$ for $\lambda=s_0+s+7/6$. 
 
Summing all these up, we deduce that when $\Re(s)\gg0$, 
\begin{align*}P(\varphi_s,\varphi_+,\varphi_-)=\int_{Z_*N(\BA)\bs G(\BA)}f_s(g)W_+(g)W_-(g)dg.
\end{align*}
\end{proof}
It is well-known (see \cite[Theorem 3.7.1]{Bum97}) that  the poles of $\varphi_s(g)=E(g,\theta\delta^s, f_s)$ with $f$ running over $I_{B(\BA)}^{G(\BA)}\theta$ are collected by $L(2s,\theta_1\theta_2^{-1})$. Moreover when $\theta_1\theta_2^{-1}=|\cdot|_\BA$, $\Res_{s=0}E(g,\chi\delta^s, f_s)$ generates the one-dimensional $G(\BA)$-representation $\theta_1|\cdot|_\BA^{-1/2}$. Up to twisting, we can assume $\theta=\delta^{1/2}$. Then taking residue of $P(\varphi_s,\varphi_+,\varphi_-)$ at $s=0$ gives the Petersson inner product formula of metaplectic theta series. 

\begin{cor}\label{Petersson}Notations as above.  Then
\[\int_{Z_*(\BA)G(F)\bs G(\BA)}\varphi_+\varphi_-(g)dg=\zeta_F(2) 
\prod_{v<\infty}\frac{(W_{+,v}, W_{-,v})_v}{\zeta_{F_v}(1)\Vol(\CO_{F_v},dx_v)}\prod_{v\mid \infty}\frac{(W_{+,v}, W_{-,v})_v}{\zeta_{F_v}(1)}.\]
The local pairing $(-,-)_v$ in term of Whittaker model is given in Subsubsection \ref{InvP}:
\[(W_{+,v}, W_{-,v})_v:=\int_{A(F_v)}\sum_{z\in Z_*(F_v)\bs Z(F_v)}W_{+,v}W_{-,v}\left[z\begin{pmatrix}y & 0\\ 0 & 1\end{pmatrix}\right]d^\times y.\]
\end{cor}
Note that both sides in the above formula can  be non-zero only when $\sigma_+^\vee=\sigma_-$. 
\begin{proof}We only deal with the non-trivial case $\sigma_+^\vee=\sigma_-$. Under the assumption $\theta=\delta^{1/2}$, one has $\theta_1(a)\chi_{+,2}(a)\chi_{-,2}(a)=|a|^{1/6}$ on $A_*(\BA)$. By the proof of  Theorem \ref{co-period for Eisenstein family},
\begin{align*}&\int_{Z_*(\BA)G(F)\bs G(\BA)}\zeta_F^{-1}(1+2s)\varphi_s\wedge^T(\varphi_+\otimes\varphi_-)(g)dg\\
&=\zeta_F^{-1}(1+2s)\int_{Z_*(\BA)N(\BA)\bs G(\BA)}f_s(g)W_+(g)W_-(g)dg\\
&+P_{s+2/3}(T)e^{(s+2/3)T}+P_{s+4/3}(T)e^{(s+4/3)T}
\end{align*}
for some polynomial $P_{s+2/3}(T)$ and $P_{s+4/3}(T)$ of $T$.

Set $C(f)= \zeta_F^{-1}(1+2s)\prod_vI_s(f_v)(1)|_{s=0}$. Since the image of $I:\ I_{B(\BA)}^{G(\BA)}\theta\to I_{B(\BA)}^{G(\BA)}\theta^{-1}$ is the trivial representation, we have $\zeta_F^{-1}(1+2s)\varphi_s(g)|_{s=0}=C(f)$. Since  $\zeta_F^{-1}(1+2s)\varphi_s(g)$ has no poles for $\Re(s)\gg0$, \begin{align*}
&C(f)\int_{Z_*(\BA)G(F)\bs G(\BA)}\wedge^T(\varphi_+\otimes\varphi_-)(g)dg\\
=&\int_{Z_*(\BA)G(F)\bs G(\BA)}\zeta_F^{-1}(1+2s)\varphi_s\wedge^T(\varphi_+\otimes\varphi_-)(g)dg\Big|_{s=0}\\
=&\zeta_F^{-1}(1+2s)\int_{Z_*(\BA)N(\BA)\bs G(\BA)}f_s(g)W_+(g)W_-(g)dg\Big|_{s=0}+P_{2/3}(T)e^{\frac{2T}{3}}+P_{4/3}(T)e^{\frac{4T}{3}}
\end{align*}
for some polynomials $P_{2/3}$ and $P_{4/3}$. By Theorem \ref{unramified calculation}, we have for almost all $v$,
  \[\int_{Z_*(F_v)N(F_v)\bs G(F_v)}f_{s,v}(g)W_{+,v}(g)W_{-,v}(g)dg_v=\frac{\zeta_{F_v}(1+s)\zeta_{F_v}(2+3s)}{\zeta_{F_v}(2+2s)}.\]
Since  $\int_{Z_*(F_v)N(F_v)\bs G(F_v)}f_v(g)W_{+,v}(g)W_{-,v}(g)dg$ converges absolutely for all $v$ by the gauge estimation in \cite[Theorem I.4.1]{KP84}, we have 
 \begin{align*} &\zeta_F^{-1}(1+2s)\int_{Z_*(\BA)N(\BA)\bs G(\BA)}f_s(g)W_+(g)W_-(g)dg\Big|_{s=0}\\
 &=\frac{\zeta_{F}(1+s)\zeta_{F}(2+3s)}{\zeta_{F}(2+2s)\zeta_F(1+2s)}\big|_{s=0} \prod_{v}\frac{\int_{Z_*(F_v)N(F_v)\bs G(F_v)}f_{s,v}(g)W_{+,v}(g)W_{-,v}(g)dg}{\frac{\zeta_{F_v}(1+s)\zeta_{F_v}(2+3s)}{\zeta_{F_v}(2+2s)}}\Bigg|_{s=0}\\
&= \prod_{v}\zeta_{F_v}^{-1}(1)\int_{Z_*(F_v)N(F_v)\bs G(F_v)}f_v(g)W_{+,v}(g)W_{-,v}(g)dg\\
&= \prod_v\zeta_{F_v}^{-1}(1)(W_{+,v}, W_{-,v})_v\int_{K_v}f_v(k)dk.
 \end{align*}
 Here we use Lemma \ref{int} and the $\wt{G}(F_v)$-invariance of the pairing
 $(-, -)_v$ to deduce the last equality.
 
 By  \cite[Page 115-116]{KP84}, the integration \[\int_{Z_*(\BA)G(F)\bs G(\BA)}\varphi_+\varphi_-(g)dg=\lim_{T\to\infty} \int_{Z_*(\BA)G(F)\bs G(\BA)}\wedge^T(\varphi_+ \otimes\varphi_-)(g)dg\]converges absolutely. Thus we must have $P_{2/3}=P_{4/3}=0$ and 
\[C(f)\int_{Z_*(\BA)G(F)\bs G(\BA)}\varphi_+\varphi_-(g)dg=\prod_v\zeta_{F_v}^{-1}(1)(W_{+,v}, W_{-,v})_v\int_{K_v}f_v(k)dk.\]

To deduce the stated formula,  note that $f\mapsto C(f)$ and $f\mapsto \int_Kf(k)dk$ both lie in the one-dimensional space $\Hom_{G(\BA)}(I_{B(\BA)}^{G(\BA)}\delta^{1/2},\BC)$, it suffices to do the calculation when $f$ is the unique K-invariant function with $f(1)=1$. In this case, $\int_Kf(k)dk=1$. Using the matrix identities 
 \[w\begin{pmatrix} 1 & x \\ 0 & 1\end{pmatrix}=\begin{pmatrix}-x^{-1} & 1\\ 0 & x\end{pmatrix}\begin{pmatrix} 1 & 0\\ x^{-1} & 1\end{pmatrix}\quad x\neq0\in F_v, v<\infty\]
  \[w\begin{pmatrix} 1 & x \\ 0 & 1\end{pmatrix}=\begin{pmatrix}-\Delta_x^{-1} & \bar{x}\Delta_x^{-1}\\ 0 & \Delta_x\end{pmatrix}\begin{pmatrix} \bar{x}\Delta_x^{-1} & -\Delta_x^{-1}\\\ \Delta_x^{-1} & x\Delta_x^{-1}\end{pmatrix}\quad x\in \BC,\ \Delta_x=\sqrt{1+x\bar{x}},\]
 we can compute that  \[C(f)=\zeta_F^{-1}(2)\prod_{v<\infty}\Vol(\CO_{F_v}, dx_v).\]
 Combining all these, we deduce the stated Petersson inner product formula.
\end{proof}
Now  we  consider the co-period 
\[P(\varphi,\varphi_+,\varphi_{-})=P(\varphi_s,\varphi_+,\varphi_{-})|_{s=0},\quad \varphi=\varphi_s|_{s=0}\]
when $s=0$ is not a pole of $\varphi_s$. Without lose of generality,   we shall assume $\theta_1\theta_2^{-1}\neq|\cdot|_\BA$ in the following. This condition implies that   $\varphi_s$ has no pole at $s=0$. \begin{prop}\label{EisenI}Assume  $\theta_1\theta_2^{-1}\neq|\cdot|_\BA$. When   $\lambda(\theta)>-1/6$,
\[P(\varphi,\varphi_+,\varphi_{-})=\frac{L(1/2,\theta_1^2\theta_2)L(1/2,\theta_1^3)}{L(1,\theta_1\theta_2^{-1})}\prod_{v}\frac{\int_{Z_*(F_v)N(F_v)\bs G(F_v)}f_v(g)W_{v,+}(g)W_{v,-}(g)dg}{\frac{L_v(1/2,\theta_{1,v}^2\theta_{2,v})L_v(1/2,\theta_{1,v}^3)}{L_v(1,\theta_{1,v}\theta_{2,v}^{-1})}}.\]
\end{prop}
\begin{proof}
When $\lambda(\theta)>-1/6$, $\int_{Z_*(F_v)N(F_v)\bs G(F_v)}f_{s,v}(g_v)W_{+,v}(g_v)W_{-,v}(g_v)dg_v$ converges absolutely when $\Re(s)\geq0$ by the gauge estimation in \cite[Theorem I.4.1]{KP84}. By Theorem \ref{unramified calculation}, 
  \[\int_{Z_*(F_v)N(F_v)\bs G(F_v)}f_{s,v}(g)W_{+,v}(g)W_{-,v}(g)dg=\frac{L_v(1/2+s,\theta_{1,v}^2\theta_{2,v})L_v(1/2+3s,\theta_{1,v}^3)}{L_v(1+2s,\theta_{1,v}\theta_{2,v}^{-1})}\]
 for almost all $v$. Thus by Proposition \ref{co-period for Eisenstein family}, we deduce 
 \[P(\varphi,\varphi_+,\varphi_{-})=\frac{L(1/2,\theta_1^2\theta_2)L(1/2,\theta_1^3)}{L(1,\theta_1\theta_2^{-1})}\prod_{v}\frac{\int_{Z_*(F_v)N(F_v)\bs G(F_v)}f_v(g)W_{+,v}(g)W_{-,v}(g)dg}{\frac{L_v(1/2,\theta_{1,v}^2\theta_{2,v})L_v(1/2,\theta_{1,v}^3)}{L_v(1,\theta_{1,v}\theta_{2,v}^{-1})}}\]
\end{proof}
Now assume $-1/6<\lambda(\theta)<1/6$. This implies that the $G(\BA)$-representation $\sigma:=I_{B(\BA)}^{G(\BA)}\theta$ (resp.  $\sigma^\vee:=I_{B(\BA)}^{G(\BA)}\theta^{-1}$) is irreducible. Also the assumption guarantees $\theta_1\theta_2^{-1}\neq|\cdot|_\BA^{\pm}$ and we can realize $\sigma$ (resp. $\sigma^\vee$) as  $\{E(g,\theta, f)| f\in I_{B(\BA)}^{G(\BA)}\theta\}$ (resp. $\{E(g,\theta^{-1}, f^\vee)| f^\vee\in I_{B(\BA)}^{G(\BA)}\theta^{-1}\}$).  \begin{itemize}
    \item Consider the inner product on $\sigma\times\sigma^\vee$:
\[\left(E(g,\theta,f),E(g,\theta^{-1},f^\vee)\right)\mapsto \int_K ff^\vee(k)dk\]
\item Consider the Petersson inner product on $\sigma_{\pm}\times\sigma_{\pm}^\vee$ with respect to the measure $d^\tau g$, i.e. 
\[(f_\pm,f_\pm^\vee)\mapsto \int_{Z_*(\BA)G(F)\bs G(\BA)}f_\pm f_\pm^\vee(g)d^\tau g\]
\end{itemize}
 Fix a decomposition $(-,-)_*=\prod_v (-,-)_{*,v}$ for $*=\pm,\emptyset$ and fix a decomposation $d^\tau g=\prod_v d^\tau g_v$.  For any $f_{v,*}^\prime\in\sigma_{v,*}^\prime$ with $v\in\CM$, $*=\emptyset,\pm$ and $\prime=\emptyset,\vee$, set 
\begin{align*}&I^\sharp(f_v\otimes f_{+,v}\otimes f_{-,v},f_v^\vee\otimes f_{+,v}^\vee\otimes f_{-,v}^\vee)\\
=\frac{L(1,\sigma_v,\ad)}{L(1/2,\sigma_v,\Sym^3)}&\int_{Z^{(c)}(F_v)\bs G(F_v)}(\sigma_v(g)f_v,f_v^\vee)_v(\sigma_{+,v}(g)f_{+,v},f_{+,v}^\vee)_{+,v}(\sigma_{-,v}(g)f_{-,v},f_{-,v}^\vee)_{-,v}d^\tau g_v.
\end{align*}
Consider the co-period integral for $*=\emptyset,\vee$ using the measure $d^\tau g$,\[\CP^*(\varphi^*,\varphi_+^*,\varphi_-^*)=\int_{Z_*(\BA)G(F)\bs G(\BA)}\varphi(g)\varphi_+(g)\varphi_-(g)d^\tau g.\]
\begin{thm}\label{EisenII}Assume $-1/6<\lambda(\theta)<1/6$. Then for any pure tensor $\varphi_*^\prime=\otimes \varphi_{v,*}^\prime\in\sigma_*^\prime=\otimes_{\BC[\mu_3]}^\prime \sigma_{v,*}^\prime $ with $*=\pm,\emptyset$ and $\prime=\emptyset,\vee$, 
\begin{align*}
\CP(\varphi,\varphi_+,\varphi_-)\CP^\vee(\varphi^\vee,\varphi_+^\vee,\varphi_-^\vee)&=\frac{L(1/2,\sigma,\Sym^3)}{L(1,\theta_1\theta_2^{-1})L(1,\theta_2\theta_1^{-1})} \prod_v I^\sharp(\varphi_{+,v}\times \varphi_{-,v}\times \varphi_v, \varphi_{+,v}^\vee\times \varphi_{-,v}^\vee\times \varphi_v^\vee).
\end{align*}

\end{thm}
\begin{proof}By the relation between $dg$ and $d^\tau g$, we have 
\[\CP(\varphi,\varphi_+,\varphi_-)\CP^\vee(\varphi^\vee,\varphi_+^\vee,\varphi_-^\vee)=\frac{P(\varphi,\varphi_+,\varphi_-)P^\vee(\varphi^\vee,\varphi_+^\vee,\varphi_-^\vee)}{\zeta_F(2)^2\prod_{v<\infty}\Vol(\CO_{F_v},dx_v)^{-2}}.\]
By Proposition \ref{EisenI} and Proposition \ref{Tri-II},\ref{Tri-III}, we have 
\begin{align*}&P(\varphi,\varphi_+,\varphi_-)P^\vee(\varphi^\vee,\varphi_+^\vee,\varphi_-^\vee)\\
&=\frac{L(1/2,\sigma,\Sym^3)}{L(1,\theta_1\theta_2^{-1})L(1,\theta_2\theta_1^{-1})}\prod_v\frac{\zeta_{F_v}(1)\Psi_v(W_{+,v},W_{-,v},f_v)\Psi_v^\vee(W_{+,v}^\vee,W_{-,v}^\vee,f_v^\vee)}{\frac{L(1/2,\sigma_v,\Sym^3)}{L_v(1,\sigma_v,\ad)}\zeta_{F_v}^2(1)}\\
&=\frac{L(1/2,\sigma,\Sym^3)}{L(1,\theta_1\theta_2^{-1})L(1,\theta_2\theta_1^{-1})}\prod_v\frac{I_v(W_{+,v}\otimes W_{-,v}\otimes f_v,W_{+,v}^\vee\otimes W_{-,v}^\vee\otimes f_v^\vee)}{\frac{L(1/2,\sigma_v,\Sym^3)}{L_v(1,\sigma_v,\ad)}\zeta_{F_v}^2(1)}.
\end{align*}
Here the integration $I_v$ of matrix coefficients is defined using the measure  $d g_v$, the $\wt{G}$-invariant pairing on $\sigma_{\pm,v}\times\sigma_{\pm,v}^\vee$ in term of Whittaker model and the G-invariant pairing on $\sigma_v\times\sigma_v^\vee$ in term of induced model. By Corollary \ref{Petersson} and the relation between $dg$ and $d^\tau g$, we deduce the stated equality.
\end{proof}

\s{\bf Acknowledgement} 
We express our deep gratitude to Prof. Freydoon Shahidi and Prof. Henry Kim, whose pioneering works \cite{Sha88, Sha89, KS99} profoundly influenced this research. We sincerely thank Prof. Ye Tian for his unwavering encouragement and support.  

We are grateful to Prof. Ashay Burungale for bringing reference \cite{HJL23} to our attention, to Prof. Wei He for informing us the reference \cite{MS00} and to Prof. Fan Gao and Mr. Sihan Li for insightful discussions on Whittaker functionals for covering groups. Our thanks also extend to Prof. Dipendra Prasad and Prof. Lei Zhang for valuable conversations regarding the multiplicity one theorem, and to Mr. Jiewen Wang for productive discussions on mixed truncation operators. 
 
This work was partially supported by the National Key R\&D Program of China (Grant No. 2023YFA1009702). L. Cai acknowledges additional support from the National Natural Science Foundation of China (Grant No. 12371012), and Y. Fan from the National Natural Science Foundation of Beijing, China (Grant No. 24A10020), the National Natural Science Foundation of China (Grant No. 12501015, No.12526208).

\appendix
\section{Local root numbers of symmetric cubes}\label{epsilon-factor}
In this appendix, we compute local root numbers of symmetric cubes to support the epsilon-dichotomy result summarized in Theorem \ref{Mulone}. Let  $F$ be a $p$-adic field with ring of integers $\CO$, maximal ideal $\fp_F$ and residue field $k_F$. Let $\pi$ be an irreducible admissible generic $\GL_2(F)$-representation with  central character $\omega$. Assume throughout that $\mu_3\subset F$ and $\omega^3$=1.

Fix a non-trivial additive character $\psi$ of $F$. We shall study the local root number (which is actually independent of the choice of $\psi$ by our assumption on $\omega$)
\[\varepsilon(\pi,\Sym^3):= \varepsilon\left( \frac{1}{2},\pi,\Sym^3, \psi \right)\in\{\pm1\}\]
 Note that $\pi\otimes\omega$ is self-dual and one has the following equality
\[\varepsilon(\pi,\Sym^3) = \frac{\varepsilon(\pi \times \pi \times \pi)}{\varepsilon(\pi \otimes \omega)^2}=\epsilon(\pi\times\pi\times \pi)\in\{\pm1\}\]
 Denote the Steinberg representation on $\GL_2(F)$ by $\St$.
 \begin{prop}\label{epsilon nsc}If $\pi$ is non-supercuspidal, then $\varepsilon(\pi,\Sym^3) = -1$ if and only if $\pi = \St \otimes \chi$ with $\chi^3=1$.
\end{prop}
\begin{proof}  By the work of Prasad \cite{Pra90}, 
\begin{itemize}
\item if $\pi$ is a principal series, then $\varepsilon(\pi \times \pi \times \pi) = +1$ 
\item if $\pi = \St \otimes \chi$, $\varepsilon(\pi \times \pi \times \pi) = -1$ if and only if $\Hom_{B^\times}(\chi \boxtimes \chi \boxtimes \chi,\BC) \not=0$. Here $B$ the division algebra over $F$. This is equivalent to $\chi^3=1$.
\end{itemize}
\end{proof}
Our main result concerns the supercuspidal case. Note that when $p\neq2$, such  $\pi$ is always dihedral for some  quadratic field extension $E/F$, i.e. $\pi=\pi_\mu$ is the binary theta lifting of  a smooth character $\mu$ on $E^\times$. Let $
 e_3=\val_E(3)$ and denote the  conductor of $\mu$ by $n(\mu)$.
\begin{thm} Assume $\pi$ is supercuspidal. 
 Then  $\varepsilon(\pi,\Sym^3)= +1$ when $p\neq 3$ or $p=3$ and  $\pi=\pi_\mu$ with $n(\mu)=1$ or $ n(\mu)>\frac{3e_3}{2}+1$.
\end{thm}
\begin{example}The conductor assumption in the case $p=3$ is necessary. Let $
 F=\mathbb Q_3(\zeta_3)$ and $E/F$ be the unramified quadratic extension. In this case $e_3=2$. Let $U_E^r=1+\fp_E^r$ where $\fp_E$ is the maximal ideal. As $U_E^r\cap F^\times=U_F^r$ and $(U_E^1)^3\subset U_E^3$, the  quotient group
\[
 A=U_E^1/(U_E^3U_F^1)\cong \frac{U_E^1/U_E^3}{U_F^1/U_F^3}
\]
has exponent three and size $9$. A diagram chasing shows that $A$ contains the subgroup  $B=U_E^2/(U_E^3U_F^2)$ of order $3$.   Choose an order-three character $\vartheta$ of $A$ that is
nontrivial on $B$ and consider its extension to $E^\times$ which is trivial on $F^\times$. This is possible since $[E^\times:F^\times U_E^1]=4$. Then
$n(\vartheta)=3$.
Extend $\eta$ to the unramified quadratic character of $E^\times$ and let 
 $\mu=\eta\vartheta$.  By Proposition \ref{app:unramified}, $\varepsilon(\pi_\mu,\Sym^3)=-1.$
\end{example}
\subsection{The dihedral case}\label{app:dihedral-section}
Let $E/F$ be a quadratic field extension with associated quadratic character $\eta$. 
Denote the ring of integers, the maximal ideal and  the residue field  of $E$ by $\CO_E$, $\fp_E$ and $k_E$ respectively. Let  $c$ be the nontrivial element in $\Gal(E/F)$.

In this subsection, we shall assume the supercuspidal representation  $\pi$ is dihedral, i.e. $\pi=\pi_\mu$ for some smooth character $\mu:E^\times\to\BC^\times$ with $\mu\ne\mu^c$. 
The central character condition $\omega^3=1$ is equivalent to $\mu^3|_{F^\times}=\eta$. The assumption $\mu\neq\mu^c$ forces $n(\mu)\geq1$.

 For any smooth character $\chi:\ E^\times\to\BC^\times,$ denote 
\[
 \varepsilon_E(\chi)=\varepsilon\left(\frac12,\chi,\psi_E\right),\quad \psi_E=\psi\circ\Tr_{E/F}
.\] 
Actually  for $\chi$ ramified, Tate's finite Gauss-sum formula \cite{Tate79} gives
\begin{equation}\label{app:Tate}
 \varepsilon_E(\chi)
 =\chi(c_\chi)\frac{T(\chi)}{|T(\chi)|},\qquad  T(\chi)=\int_{\CO_E^\times}
 \chi^{-1}(u)\psi_E(u/c_\chi)\,du.
\end{equation}
Here   $c_\chi\in E^\times$ satisfies that
\[
 \val_E(c_\chi)=n(\chi)+n(\psi_E).
\]

\begin{lem}\label{app:decomposition}
Assume $\mu^3|_{F^\times}=\eta$ and  put
$\rho_\mu=\Ind_{W_E}^{W_F}\mu$. Then
\begin{equation*}
 \Sym^3\rho_\mu\simeq
 \Ind_{W_E}^{W_F}(\mu^3)\oplus
 \Ind_{W_E}^{W_F}(\mu^2\mu^c),
\end{equation*}
and consequently
\begin{equation*}\label{app:universal}
 \varepsilon(\pi_\mu,\Sym^3)
 =\eta(-1)\varepsilon_E(\mu^3)
                  \varepsilon_E(\mu^2\mu^c).
\end{equation*}
\end{lem}

\begin{proof}
When restricted  to $W_E$, $\Sym^3\rho_\mu$  decomposes as the direct sum of 
\[
 \mu^3,\quad \mu^2\mu^c,\quad
 \mu(\mu^c)^2,\quad (\mu^c)^3.
\]
The $c$-action gives the stated  decomposition.  The induction
formula for local constants gives
\[
 \varepsilon_F(\Ind_{W_E}^{W_F}\chi)
 =\lambda_{E/F}(\psi)\varepsilon_E(\chi),
 \qquad \lambda_{E/F}(\psi)^2=\eta(-1),
\]
which proves the formula of epsilon-factors.
\end{proof}

Put
\[
 \chi_0=\mu^3,
 \qquad
 \chi_1=\mu^2\mu^c.
\]
Note that both characters restrict to $\eta$ on
$F^\times$ and $
 \chi_i^c=\chi_i^{-1}.$ 
\begin{lem}\label{app:conductor-one}
When $n(\mu)=1$, $
 n(\chi_0)=n(\chi_1)=1.$
\end{lem}

\begin{proof}
If $E/F$ is ramified, both characters restrict to the ramified character
$\eta$ of $F^\times$. Thus both have conductor one.

Suppose that $E/F$ is unramified and regard $\mu$ as a non-trivial
character of $k_E^\times$. Denote $|k_F|$ by $q$. Then $c(x)=x^q$ for $x\in k_E$ and $|k_E^\times|=q^2-1$.

When $p=3$, $(3, q^2-1)=(q+2,q^2-1)=1$. When $p\neq 3$, $\mu_3\subset F^\times$ forces $q\equiv 1\mod 3$. The assumption $\mu\neq\mu^c$ forces \[\frac{\chi_1}{\chi_0}=\mu^{q-1}=\mu^{(q-1)(q+2)}\neq 1.\]  Thus both $\chi_i$ are non-trivial on $k_E^\times$ and consequently, $n(\chi_0)=n(\chi_1)=1$.
\end{proof}
Note that $e_3=0$ when $p\neq 3$ and when $p=3$, the condition $\mu_3\subset F$ forces $e_3$ to be even.
\begin{lem}\label{app:conductors}
If $
 n(\mu)>\frac{3e_3}{2}+1$,
then
\begin{equation}\label{app:conductor-formula}
 n(\chi_0)=n(\mu)-e_3,
 \qquad
 n(\chi_1)=n(\mu).
\end{equation}
\end{lem}

\begin{proof}
Let $\fp_E$ be the maximal ideal of $E$. Then for  $n>\frac{3e_3}{2}$, one has an isomorphism
\[
 1+\fp_E^{n-e_3}\longrightarrow1+\fp_E^{n},
 \qquad u\longmapsto u^3.
\]
  This proves  $n(\chi_0)=n(\mu)-e_3$. To see $n(\chi_1)=n(\mu)$, it suffices to show $n(\mu\mu^c)<n(\mu)$. If
$n(\mu\mu^c)=n(\mu)$, the same argument would force the conductor of $(\mu\mu^c)^3$ to be $n(\mu)-e_3$, while we have
\[
 (\mu\mu^c)^3
 =\bigl(\mu^3|_{F^\times}\bigr)\circ\Nm_{E/F}=1.
\]
\end{proof}

\subsubsection{The unramified case}

\begin{prop}\label{app:unramified}
If $E/F$ is unramified, then
\begin{equation*}\label{app:unramified-formula}
 \varepsilon(\pi_\mu,\Sym^3)
 =(-1)^{n(\chi_0)+n(\chi_1)}.
\end{equation*}

\end{prop}

\begin{proof}Fix a unit  $\theta\in\CO_E$ such that $\CO_E=\CO[\theta]$ and choose $x_0\in\CO_E^\times$ such that
$x_0^c=-x_0$.   Then by \cite[Proposition 3.7(iii)]{MS00}, \[
 \varepsilon_E(\chi_i)=(-1)^{n(\chi_i)+n(\psi)+\val_E(\frac{x_0}{\theta-\theta^c})}\chi_i(x_0).
\] Since
$x_0^2=-\Nm_{E/F}(x_0)$ and $\eta(-1)=1$, one has
\[
 \mu^6(x_0)
 =\mu^3(-1)\mu^3(\Nm_{E/F}(x_0))=1,
 \qquad
 \frac{\mu^c}{\mu}(x_0)=\mu(-1)=1.
\]
It follows that
$(\chi_0\chi_1)(x_0)=\mu^5\mu^c(x_0)=1$. 
\end{proof}

\subsubsection{The ramified case: $p\neq2$}

Now suppose that $E/F$ is ramified and $p\neq 2$. Without loss of generality, assume $\psi$ has conductor zero and consequently, $n(\psi_E)=1$.    Fix a
trace-zero element $x_0\in E^\times$  and choose a conductor-one character
\[
 \widetilde\eta:E^\times\longrightarrow\BC^\times,
 \qquad \widetilde\eta|_{F^\times}=\eta.
\]
If $\chi|_{F^\times}=\eta$ and $n(\chi)\ge2$, then $n(\chi)$ is even by the isomorphism $U_F^r/U_F^{r+1}\cong U_E^{2r}/U_E^{2r+1}$.
Choose $y_\chi\in E^\times$ such that
\begin{equation}\label{app:ychi}
 \val_E(y_\chi)=-n(\chi)-1,
 \qquad
 (\chi\widetilde\eta^{-1})(1+x)=\psi_E(y_\chi x)
 \quad(\val_E(x)\ge n(\chi)/2).
\end{equation}
Then  Deligne's twisting formula \cite{Deligne73} and the
Fr\"ohlich--Queyrut formula \cite{FQ73} give
\begin{equation}\label{app:DFQ}
 \varepsilon_E(\chi)
 =\widetilde\eta^{-1}(y_\chi)
  (\chi\widetilde\eta^{-1})(x_0).
\end{equation}

\begin{prop}\label{app:tame}
Assume that $E/F$ is ramified and $p\neq 2$. If $n(\mu)=1$ or $n(\mu)>\frac{3e_3}{2}+1$, then $\varepsilon(\pi_\mu,\Sym^3)=1.$

\end{prop}

\begin{proof} When $n(\mu)=1$, both $\chi_i$ has conductor one by Lemma \ref{app:conductor-one}. Moreover one can check that $\{\chi_0,\chi_1\}=\{\wt{\eta}, \wt{\eta}\xi\}$ where $\xi$ is the unramified quadratic character on $E^\times$. 
Then by Lemma \ref{app:decomposition} and Deligne's twisting lemma, \[\varepsilon(\pi,\Sym^3)=\eta(-1)\varepsilon(\wt{\eta})\varepsilon(\wt{\eta}\xi)=\eta(-1)\eta(-1)\xi(\varpi_E)^{n(\wt{\eta})-n(\psi_E)}=(-1)^{n(\psi_E)-1}=1.\]
When $n(\mu)>\frac{3e_3}{2}+1$, one has  \[\varepsilon(\pi_\mu,\Sym^3)
 =\widetilde\eta^{-1}(y_{\chi_0}y_{\chi_1})\]
by Equation \eqref{app:DFQ}, Lemma \ref{app:decomposition}  and
the identities \[
 (\chi_0\chi_1)(x_0)=1,
 \qquad
 \widetilde\eta^{-2}(x_0)=\eta(-1).
\]
Take $y_\mu\in E^\times$ such that \[\mu(1+x)=\psi_E(y_\mu x), \quad\forall x\in\fp_E^{\frac{n(\mu)}{2}}. \]  By Lemma \ref{app:conductors}, one may take
\[
 \frac{y_{\chi_0}}{3y_\mu}\in1+\fp_E,
 \qquad
 y_{\chi_1}=2y_\mu+y_\mu^c.
\]
Since $n(\mu)$ is even, $y_\mu/y_\mu^c\equiv-1\pmod{\fp_E}$, and consequently \[\varepsilon(\pi,\Sym^3)=\wt{\eta}^{-1}(3(2\frac{y_\mu}{y_\mu^c}+1))=\wt{\eta}^{-1}(-3).\]
Since $\mu_3\subset F$, $-3$ is a square in
$F$ and $\eta(-3)=1$.
\end{proof}

\subsubsection{The ramified case: $p\neq 3$}
To deal with the remaining case $p=2$ and $E/F$ is ramified, we use a Galois conjugation argument (which works for $p\neq 3$). The following straightforward observation (which works in general) allows us to assume $\omega=1$.

\begin{lem}\label{app:cubic-conjugacy}
Let $\chi$ be a finite-order character of $E^\times$ such that
 $\chi|_{F^\times}=\eta$.
Assume  the order of $\chi$ is prime to $3$.  Then $\varepsilon_E(\chi)=\varepsilon_E(\chi^3)$
\end{lem}

\begin{proof}By assumption $\chi$ is ramified. If $\chi$ has conductor one, then $\chi^3=\chi^c$ and  $\varepsilon_E(\chi)=\varepsilon_E(\chi^3)$. 

If $n(\chi)\geq 2$, then  $n(\chi)$ is even and $|T(\chi)|=|T(\chi^3)|\in\BZ$. Since  the order of $\chi$ is prime to $3$, cubing the roots of unity induces an automorphism $\sigma_3$ on $\BQ(\chi)$.
Since  $p\neq3$, one  choose $a_\chi=a_{\chi^3}$ in Tate's finite sum. Thus   \[\sigma_3(\varepsilon_E(\chi,\psi_E))=\varepsilon_E(\chi^3,\psi_E^3)\]
By the independence of the choice of additive character, one deduce $\varepsilon_E(\chi)=\varepsilon_E(\chi^3)$.
\end{proof}

\begin{prop}\label{app:dyadic-dihedral} Let $\mu:\ E^\times\to\BC^\times$ be a smooth character such that $\mu\neq\mu^c$ and $\mu^3|_{F^\times}=\eta$
\[
 \varepsilon_F(\Sym^3\Ind_{W_E}^{W_F}\mu)=1.
\]

\end{prop}

\begin{proof}Write $\rho=\Ind_{W_E}^{W_F}\mu$. Then \[\Sym^3\rho\simeq\Ind\chi_1^3\oplus\Ind\chi_1,\quad \chi_1=\mu^2\mu^c.\]
By Lemma \ref{app:cubic-conjugacy}, it suffices to check the order of $\chi_1$. Note that $E^\times=\langle\varpi_E\rangle\times\mu_{q-1}\times U_E^1$ where $\varpi_E$ is a uniformizer of $E$ and $q=|k_F|$. Since $\chi_1|_{F^\times}=\eta$ and the group $U_E^1$ is pro-$p$, we deduce the order of $\chi_1$ is prime-to-$3$.
\end{proof}

\subsection{The primitive case}
In this subsection, we deal with case $\pi$ is primitive,i.e. $\pi$ is not dihedral. This only happens when $p=2$.
\begin{prop}\label{app:dyadic-primitive}
Let $\rho$ be a primitive irreducible two-dimensional Weil parameter with
$(\det\rho)^3=1$.  Then its projective image is $A_4$, and
\[
 \varepsilon_F(\Sym^3\rho)=1.
\]
\end{prop}

\begin{proof}
The monodromy operator is zero, since otherwise its kernel would be a
$W_F$-stable line. By twisting $\rho$ by  its determinant, we may and will assume $\det\rho=1$.  The classification of primitive
two-dimensional local Weil parameters gives projective image $A_4$ or
$S_4$; the other finite possibilities are imprimitive or nonsolvable
\cite{Koch77,BH06,Serre79}.

  Let $P\triangleleft I \triangleleft G$ be the
projective images of  wild inertia group, inertial group and  $W_F$ respectively. Then $P$ is pro-$2$ group, $I/P$ is cyclic of odd order and $G/I$ is cyclic.  This forces 
\[
 P=V_4,
 \qquad I=A_4，
\]
where $V_4$ is the Klein group.
Projective Frobenius acts by inversion on
$I/P\simeq C_3$.  On the other hand, the tame relation says that it acts
by raising to the $q$-th power.  This would give
$q\equiv-1\pmod3$, contrary to $q\equiv1\pmod3$.  Hence the projective
image is $A_4$.

Since $A_4$ has no faithful representation of dimension two, the image of $W_F$ in $\SL_2(\BC)$ is a double cover of $A_4$, which is the binary tetrahedral group
$2T\cong\SL_2(\BF_3)\subset\mathrm{SL}_2(\BC)$.  Let $Q_8\triangleleft2T$ be the inverse image of $V_4$,
and let $\nu$ be a nontrivial character of
$2T/Q_8\simeq C_3$.  A character calculation gives
\begin{equation*}\label{app:binary-tetrahedral}
 \Sym^3\rho\simeq(\rho\otimes\nu)\oplus(\rho\otimes\nu^{-1}).
\end{equation*}
The two summands are contragredient.  The functional equation therefore
gives
\[
 \begin{aligned}
 \varepsilon_F(\rho\otimes\nu)
 \varepsilon_F((\rho\otimes\nu)^\vee)
 &=\det(\rho\otimes\nu)(-1)\\
 &=\nu^2(-1)=1.
 \end{aligned}
\]
\end{proof}

\end{document}